%% file: manuscript.tex
\documentclass[hidelinks,onefignum,onetabnum]{siamart220329}
\usepackage{comment}
\usepackage{multirow}
\usepackage{color}
\usepackage{float}
\usepackage{amsmath}
\usepackage{lineno}
\usepackage{subfigure}
\usepackage{makecell}
\usepackage{xcolor,caption,graphicx,booktabs,array,color}
\usepackage{enumerate}
\usepackage{algorithm,algorithmic}

\graphicspath{{figures/}}

\def\epsilon{\varepsilon}
\def\m{\textbf{m}}
\def\e{\textbf{e}}

\def\bu{\textbf{u}}
\def\ba{\textbf{a}}
\def\d{\,\mathrm{d}}
\def\bh{\textbf{h}}

\def\bx{\boldsymbol{x}}
\def\by{\boldsymbol{y}}

\def\bdiv{\mathrm{div}}

\def\ha{\mathbf{h}_{\mathrm{a}}}

\newcommand{\abs}[1]{\lvert#1\rvert}

\input{ex_shared}

\ifpdf
\hypersetup{
  pdftitle={Two-scale analysis for multiscale Landau-Lifshitz-Gilbert equation},
  pdfauthor={X. Guan, H. Qi, and Z. Sun}
}
\fi

\begin{document}
\begin{sloppypar}

\maketitle

\begin{abstract}
  This paper discusses the theory and numerical method of two-scale analysis for the multiscale Landau-Lifshitz-Gilbert equation in composite ferromagnetic materials. 
  The novelty of this work can be summarized in three aspects:
  Firstly, the more realistic and complex model is considered, including the effects of the exchange field, anisotropy field, stray field, and external magnetic field. The explicit convergence orders in the $H^1$ norm between the classical solution and the two-scale solution are obtained.
  Secondly, we propose a robust numerical framework, which is employed in several comprehensive experiments to validate the convergence results for the Periodic and Neumann problems.
  Thirdly, we design an improved implicit numerical scheme to reduce the required number of iterations and relaxes the constraints on the time step size, which can significantly improve computational efficiency. Specifically, the projection and the expansion methods are given to overcome the inherent non-consistency in the initial data between the multiscale problem and homogenized problem. 
\end{abstract}



\section{Introduction}

The composite ferromagnetic materials are widely used in practical engineering, such as electric motors \cite{HGSS:JMMM:2006}, sensors \cite{MSLML:SAP:1991}, and data storage devices \cite{VZSN:JAP:2014}, 
due to their ability to achieve physical and chemical properties difficult for homogeneous materials \cite{RSB:CM:2022}.
The state of composite ferromagnetic materials is defined by the magnetic moment $\mathbf{m}^\varepsilon\in S^2$, which is the sphere-valued vector field, and $\epsilon\ll 1$ represents the spatial period of microstructure. Then, the dynamics of $\mathbf{m}^\varepsilon$ are characterized by the multiscale Landau-Lifshitz-Gilbert (LLG) equation
\begin{equation}\label{multi model}
    \begin{aligned}
        &\partial_t \mathbf{m}^{\varepsilon}-\alpha \mathbf{m}^{\varepsilon} \times \partial_t \mathbf{m}^{\varepsilon} =-\left(1+\alpha^2\right) \mathbf{m}^{\varepsilon} \times \mathbf{h}_{\mathrm{eff}}^{\varepsilon}\left(\mathbf{m}^{\varepsilon}\right).
    \end{aligned}
\end{equation}
Here, $\alpha>0$ denotes the damping constant. 
$\mathbf{h}_{\mathrm{eff}}^{\varepsilon}\left(\mathbf{m}^{\varepsilon}\right)$ is the effective field, which typically includes the exchange field, anisotropy field, stray field, and external magnetic field.
To overcome the difficulties arising from strong degeneracy and nonconvex constraint $|\m|=1$ of LLG equation without considering multiscale characteristics \cite{LL:E:1935}, extensive studies have been conducted on the well-posedness of LLG equation, and the solutions are proved to exist both in the weak sense \cite{FA:NATMA:1992,CEF:MMAS:2011} and in the strong sense \cite{SSB:CMP:1986,CF:DIE:2001,FT:SJMA:2017,CF:CAA:2001,CV:NATMA:2008}.

When delving into the multiscale LLG equation \eqref{multi model}, additional difficulties in analysis and simulation arise from the singular perturbation induced by $\varepsilon$ \cite{DYGC:JCP:2022,LH:JCAM:2024}. 
To address these problems, one can apply multiscale method to find the approximate solution of \eqref{multi model}

\begin{equation}\label{two scale approximate}
    \mathbf{m}^{\varepsilon}(\bx) \approx \m_0(\bx)+ \epsilon\boldsymbol{\chi}\left(\frac{\bx}{\epsilon}\right)\cdot\nabla\m_0(\bx),
\end{equation}
where $\m_0$ is the homogenized solution, describing the macro behavior of \eqref{multi model}. $\boldsymbol{\chi}\left(\frac{\bx}{\epsilon}\right)$ is the $\epsilon$-scale corrector, as given in \eqref{first-order corrector}. 
Then, heterogeneous multiscale methods (HMM) of multiscale LLG equation \cite{WE:CMS:2003,LR:MMS:2022} have been proposed, which can significantly reduce the computational costs. The works \cite{CMT:DCDS:2018,S:JMAA:2007,AD:PRSMPES:2015,CDMSZ:MMS:2022} have proved the approximation \eqref{two scale approximate} in the sense of two-scale convergence.
Suppose only the exchange field is considered in \eqref{multi model}, and the periodic boundary condition is posed. In this case, it has been proved that the convergence order $O(\epsilon)$ of \eqref{two scale approximate} holds in the $L^2(\Omega)$ sense, and remains uniformly bounded in the $H^1(\Omega)$ sense \cite{LR:CMS:2022}.
Subsequently, \cite{CLS:apa:2022} considered the effect of both the exchange field and the stray field in $\mathbf{h}_{\mathrm{eff}}^{\varepsilon}\left(\mathbf{m}^{\varepsilon}\right)$ under the Neumann boundary condition. An approximation error in the $H^1$ norm is derived by designing a second-order corrector which satisfies specific geometric constraints. Moreover, it was found that approximation \eqref{two scale approximate} has a convergence deterioration under the Neumann boundary condition. To address this scenario, a more accurate approximation is chosen as 
\begin{equation}\label{Neumann approximate}
\mathbf{m}^{\varepsilon}(\bx) \approx \m_0(\bx)+ \big(\boldsymbol{\Phi}^\varepsilon(\bx) - \bx\big)\cdot\nabla\m_0(\bx),
\end{equation}
where $\boldsymbol{\Phi}^\epsilon$ is the Neumann corrector defined in \eqref{neumann correct}.

In the theoretical part of this work, we extend the analysis in \cite{CLS:apa:2022} to a more realistic model based on Landau and Lifshitz's original consideration \cite{LL:E:1935}. Compared to the simplified model of \eqref{multi model} with only the exchange field, our model further considers the contribution of the anisotropy field, stray field, and external magnetic field in the effective field $\mathbf{h}_{\mathrm{eff}}^{\varepsilon}\left(\mathbf{m}^{\varepsilon}\right)$ of \eqref{effective field}. 
It is worth noting that we focus on the periodic boundary problem of \eqref{multi model} and the corresponding approximation error of \eqref{two scale approximate}. The main difficulties arise from the specific second-order corrector of our multiscale LLG equation. Applying a similar strategy in \cite{CLS:apa:2022}, the total consistency error is decomposed into two parts: the two-scale approximation error and the stray field error. Finally, we derive the approximation error of \eqref{two scale approximate} in the $H^1$ norm. Additionally, a uniform $W^{1,6}$ estimate of \eqref{multi model} is deduced by the inner elliptic estimate and the compactness of the torus.

Although the convergence order of approximation \eqref{two scale approximate}-\eqref{Neumann approximate} have been proved in \cite{CLS:apa:2022}, a gap exists in their numerical simulation.
It is worth noting that the theoretical analysis usually relies on specific smoothness assumptions, such as:
\begin{itemize}
    \item The periodic perturbation coefficients in $\mathbf{h}_{\mathrm{eff}}^{\varepsilon}\left(\mathbf{m}^{\varepsilon}\right)$ being $C^2$ continuous.
    \item The boundary of the area $\Omega$ being $C^{1,1}$ continuous.
    \item The initial data being sufficiently smooth.
\end{itemize}
However, it is hard to handle the above assumptions in the numerical simulation.
Several numerical methods have been used to solve the LLG equation \cite{BP:SJNA:2006,YCH:JCP:2021}. Specifically, we refer to the work \cite{BP:SJNA:2006}, which provides an implicit finite element method with unconditional convergence.
When delving into the simulation of the multiscale LLG equation, it becomes computational prohibitive and eventually infeasible to achieve adequate numerical resolution \cite{WFC:CMAME:2023,YC:MMS:2023}. 
Thus, we design an improved implicit numerical scheme inspired by \cite{BP:SJNA:2006}, where a semi-implicit iteration approach is employed to solve the nonlinear system that results from implicit temporal discretization. In contrast to the implicit treatment in \cite{BP:SJNA:2006}, the improved implicit approach significantly reduces the required number of iterations and relaxes the constraints on the time step size, as shown in \cref{tab: time and niter}.
Merging the improved implicit scheme with the two-scale method, a numerical framework is provided to deal with the multiscale LLG equation. 
It is worth noting that the initial data of the multiscale problem \eqref{multi model} and its homogenized problem exhibit inherent non-consistency. Therefore, we propose two viable methods, the projection method and the expansion method, to overcome this difference.
Finally, the convergence orders of both the Periodic and Neumann problems of \eqref{multi model} with complicated effective field $\mathbf{h}_{\mathrm{eff}}^{\varepsilon}\left(\mathbf{m}^{\varepsilon}\right)$
are successfully verified, which implies the robustness of our numerical framework. 




This paper is organized as follows. In \cref{sec2}, the homogenized model and correctors of multiscale LLG equation are derived by the two-scale method. Moreover, the boundary and initial setting of the \eqref{multi model}, together with the numerical framework are given. In \cref{sec3}, we present 
the convergence results of both the Periodic and Neumann problems under approximation \eqref{two scale approximate}-\eqref{Neumann approximate}.
In \cref{sec4}, numerical results are conducted to validate the error estimates presented in the previous section, which imply the effectiveness and robustness of the proposed framework. Finally, conclusions are given in \cref{sec5}.

\section{Two-scale analysis and numerical discretization}\label{sec2}

In this work, we consider the multiscale LLG equation \eqref{multi model}, with the effective field $\mathbf{h}_{\mathrm{eff}}^{\varepsilon}$ 
\begin{equation}\label{effective field}
	\mathbf{h}_{\mathrm{eff}}^{\varepsilon}\left(\mathbf{m}^{\varepsilon}\right)
	=
	\mathrm{div}\left(\ba^{\varepsilon} \nabla \mathbf{m}^{\varepsilon}\right)-K^{\varepsilon}\left(\mathbf{m}^{\varepsilon} \cdot \mathbf{u}\right) \mathbf{u}+M^\varepsilon\mathbf{h}_{\mathrm{a}}
	+\mu^\varepsilon \mathbf{h}_{\mathrm{d}}\left[M^\varepsilon \mathbf{m}^{\varepsilon}\right],
\end{equation}
where $\m^\epsilon:\Omega \rightarrow S^2$ is the magnetic moment, and $\Omega \subset \mathbb{R}^n\ (n=2,3)$ is the open bounded domain. The dominant second-order term  $\operatorname{div}\left(\ba^{\varepsilon} \nabla \mathbf{m}^{\varepsilon}\right)$ represents the exchange field, where $\ba^\varepsilon$ is the exchange coefficient. 
$K^\varepsilon$ is the material anisotropy coefficient, and $\mathbf{u}$ denotes the easy-axis direction.  $\mathbf{h}_{\mathrm{a}}$ represents the external magnetic field, and $M^\varepsilon$ is the magnetization magnitude.
In the last term, $\mu^\varepsilon$ is the material coefficient, and the stray field $\mathbf{h}_{\mathrm{d}}\left[\mathbf{m}\right]$ denotes a non-local term derived from the static Maxwell equation. 
When $n=3$, for the open bounded domain $\Omega$ with Lipschitz boundary, the stray field can be generated by the potential function $U^\epsilon$
\begin{equation}\label{stray field 3D}
	\left\{\begin{aligned}
		&\mathbf{h}_{\mathrm{d}}\left[M^\varepsilon \mathbf{m}^{\varepsilon}\right]
		=
		\nabla U^\varepsilon,\\
		&\Delta U^\varepsilon
		=-
		\mathrm{div}\left(M^\varepsilon \mathbf{m}^\varepsilon \cdot \mathcal{X}_{\Omega}\right), \quad \text { in } D^{\prime}\left(\mathbb{R}^3\right),
	\end{aligned}\right.
	\qquad \text{Stray Field in 3D}.
\end{equation}
Here, $M^\varepsilon \mathbf{m}^\varepsilon \cdot \mathcal{X}_{\Omega}$ represents the extension of $M^\varepsilon \mathbf{m}$ to $\mathbb{R}^3$, which becomes zero outside $\Omega$. $D^{\prime}\left(\mathbb{R}^3\right)$ denotes a distribution space, which consists of 
all continuous linear functional of space $C^{\infty}_0(\mathbb{R}^3)$.
When considering the stray field on the thin film (n=2), the magnetization component along $\mathbf{e}_3$ direction disappears. In this case, it has been proved that the stray field degenerates into the simple form \cite{GJ:PRSLSMPES:1997}
\begin{equation}\label{stray field 2D}
	\mathbf{h}_{\mathrm{d}}\left[M^\varepsilon \mathbf{m}^{\varepsilon}\right]
	=
	M^\varepsilon \left(\mathbf{m}^{\varepsilon}\cdot \mathbf{e}_3 \right)\mathbf{e}_3,
	\qquad \text{Stray Field in 2D}.
\end{equation}



Throughout this article, we focus on composite ferromagnetic materials exhibiting the $\varepsilon-$periodic structure. In other words, material coefficients satisfy the following periodic perturbation assumption
\begin{equation}\label{periodic perturbation assumption}
	M^{\varepsilon}=M_s\left(\frac{\mathbf{x}}{\varepsilon}\right), \quad \mu^{\varepsilon}=\mu\left(\frac{\mathbf{x}}{\varepsilon}\right), \quad \text { and } \quad \mathbf{a}^{\varepsilon}=\mathbf{a}\left(\frac{\mathbf{x}}{\varepsilon}\right).
\end{equation}
Here, $M_s(\by)$, $\mu(\by)$ and $ \mathbf{a}(\by)$ are periodic functions over the square domain $Y=[0,1]^n$.


\subsection{Two-scale method}
In order to address the computational burden of multiscale modeling, the two-scale method is employed to obtain the macroscale and microscale problems corresponding to \eqref{multi model}.
Moreover, the variable separation method is applied to find the second-order corrector, which satisfies specific geometric constraint. This suitable second-order corrector enables us to conduct the rigorous error analysis in \cref{sec3}.
\subsubsection{Homogenized equation in 3D case}
For the 3D case, we rewrite \eqref{multi model}, \eqref{effective field}-\eqref{stray field 2D} into a parabolic-elliptic type system
\begin{equation}\label{rewrite LLG}
	\left\{\begin{aligned}
		\partial_t \mathbf{m}^{\varepsilon}
		=& 
		\alpha \mathbf{m}^{\varepsilon} \times \partial_t \mathbf{m}^{\varepsilon} \\
		& -\left(1+\alpha^2\right) \mathbf{m}^{\varepsilon} \times
		\big[ \mathrm{div}\left(\ba^{\varepsilon} \nabla \mathbf{m}^{\varepsilon}\right)-K^{\varepsilon}\left(\mathbf{m}^{\varepsilon} \cdot \mathbf{u}\right) \mathbf{u} + M^\epsilon \mathbf{h}_{\mathrm{a}}
		+\mu^\epsilon \nabla U^\varepsilon \big],\\
		\Delta U^\varepsilon
		=& -
		\mathrm{div}\left(M^\epsilon \mathbf{m}^\varepsilon \cdot \mathcal{X}_{\Omega}\right), \qquad \text { in } D^{\prime}\left(\mathbb{R}^3\right).
	\end{aligned}\right.
\end{equation}
To utilize the two-scale method, the approximating solution $\widetilde{\m}^\epsilon(\bx)$ of $\m^\epsilon$ can be defined as
\begin{equation}\label{def two scale approximate sln}
	\left\{\begin{aligned}
		\widetilde{\m}^\epsilon(\bx) =& \m_0\left(\bx, \frac{\bx}{\epsilon}\right) + \epsilon \m_1\left(\bx,\frac{\bx}{\epsilon}\right) + \epsilon^2\m_2\left(\bx,\frac{\bx}{\epsilon}\right),\\
		\widetilde{U}^\epsilon(\bx) =& U_0\left(\bx, \frac{\bx}{\epsilon}\right) + \epsilon U_1\left(\bx,\frac{\bx}{\epsilon}\right).
	\end{aligned}\right.
\end{equation}
Here, $\m_i(\bx,\by)\ (i=0,1,2)$ is defined over $\Omega\times Y$, and $U_i(\bx,\by)\ (i=0,1)$ is defined over $\mathbb{R}^3\times Y$. We assume that $\m_i(\bx,\by)$ and $U_i(\bx,\by)$ are $Y$-periodic with respect to $\by$. Furthermore, $\m_i$ satisfies the geometric ansatz
\begin{equation}\label{geo property}
	\abs{\m_0} = 1,\quad
	\m_0\cdot\m_1 = 0,
	\quad\mbox{and}\quad
	\m_0\cdot\m_2 = -\frac{1}{2}\abs{\m_1}^2.
\end{equation}
This ansatz constrains the approximating solution $\widetilde{\m}^\epsilon(\bx)$ such that $\abs{\widetilde{\m}^\epsilon(\bx)}\approx 1$. Let us denote the fast variable $\boldsymbol{y}=\frac{\boldsymbol{x}}{\varepsilon}$ and recall the chain rule
\begin{equation}
    \nabla \mathbf{m}\left(\boldsymbol{x}, \frac{\boldsymbol{x}}{\varepsilon}\right)=\left(\nabla_{\boldsymbol{x}}+\varepsilon^{-1} \nabla_{\boldsymbol{y}}\right) \mathbf{m}(\boldsymbol{x}, \boldsymbol{y}).
\end{equation}
Then, let $\mathcal{A}_{\varepsilon}=\operatorname{div}\left(a^{\varepsilon} \nabla\right)$, we have
\begin{equation}\label{twoscale-expansion}
    \mathcal{A}_{\varepsilon} \mathbf{m}\left(\boldsymbol{x}, \frac{\boldsymbol{x}}{\varepsilon}\right)=\left(\varepsilon^{-2} \mathcal{A}_0+\varepsilon^{-1} \mathcal{A}_1+\mathcal{A}_2\right) \mathbf{m}(\boldsymbol{x}, \boldsymbol{y}),
\end{equation}
where 
\begin{equation}
    \left\{\begin{aligned}
        &\mathcal{A}_0=\mathrm{div}_{\boldsymbol{y}}\big(\ba(\boldsymbol{y}) \nabla_{\boldsymbol{y}}\big), \\
        &\mathcal{A}_1=\mathrm{div}_{\boldsymbol{x}}\big(\ba(\boldsymbol{y}) \nabla_{\boldsymbol{y}}\big)+\mathrm{div}_{\boldsymbol{y}}\big(\ba(\boldsymbol{y}) \nabla_{\boldsymbol{x}}\big), \\
        &\mathcal{A}_2=\mathrm{div}_{\boldsymbol{x}}\big(\ba(\boldsymbol{y}) \nabla_{\boldsymbol{x}}\big) .
    \end{aligned}\right.
\end{equation}
Substituting \eqref{twoscale-expansion} into \eqref{rewrite LLG}, the $\epsilon^{-2}$-order equation can be written as
\begin{equation*}\label{eqn: f-2}
		\m_0\times \mathcal{A}_0 \m_0 =0,\qquad
		\Delta_{\by} U_0 = 0.
\end{equation*}
The equation can be easily checked by setting 
$\m_{0}(\bx,\by) = \m_{0}(\bx)$ and $U_{0}(\bx,\by) = U_{0}(\bx)$.
In fact, one can prove the uniqueness of the solution (up to a constant) by the Lax-Milgram Theorem and the same argument as in \cite{CDMSZ:MMS:2022}. 
Using the above results, the $\varepsilon^{-1}$-order equation can be given as  
\begin{equation*}
	\left\{\begin{aligned}
		& \m_0\times \big( \mathcal{A}_1 \m_0 +  \mathcal{A}_0 \m_1\big) =0 ,\\
		& \Delta_{\by} U_1(\bx,\by) + \nabla_{\by} M_s(\by)\cdot \big(\m_0(\bx)\mathcal{X}_{\Omega}(\bx)\big) = 0.
	\end{aligned}\right.
\end{equation*}
Here, the first-order correctors $\m_1$ and $U_1$ are defined as
\begin{equation}\label{first-order corrector}
	\m_1(\bx,\by) = \sum_{j= 1}^3 \chi_j(\by)\frac{\partial}{\partial x_j}\m_0(\bx),
	\qquad 
	U_1(\bx,\by) = \textbf{u} (\by) \cdot \m_0 (\bx)\mathcal{X}_{\Omega}(\bx),
\end{equation}
where $\chi_i(\by)$, $i=1,2, 3$, and $\textbf{u}(\by)=\nabla U^*(\by)$ are the auxiliary functions satisfying the first-order cell problem
\begin{equation}\label{first-order cell problem}
	\left\{ \begin{aligned}
		& \bdiv \big( \ba(\by)\nabla \chi_j(\by)\big) = - \sum_{i= 1}^3 \frac{\partial}{\partial y_i} a_{ij}(\by),\quad \chi_j \mbox{ is $Y$-periodic},\\
		& \Delta U^*(\by) = - M_s(\by)+M^0, \quad  U^* \mbox{ is $Y$-periodic},
		\quad  M^0 = \int_{Y} M_s(\by) \d \by.
	\end{aligned} \right.
\end{equation}
Furthermore, the property $\m_0 \cdot\m_1 = 0$ holds under the assumption $\abs{\m_0} = 1$.
Now, let us denote
\begin{equation}\label{define h}
	\begin{aligned}
		\bh(\bx,\by) &= \mathcal{A}_1 \m_1 +  \mathcal{A}_2 \m_0 + \mu\nabla_{\bx} U_0 + \mu\nabla_{\by} U_1 - K\left(\mathbf{m}_0 \cdot \mathbf{u}\right) \mathbf{u} + M_s \mathbf{h}_{\mathrm{a}},
	\end{aligned}
\end{equation}
the $\epsilon^{0}$-order equation can be given as
\begin{equation}\label{order of 1}
	\left\{\begin{aligned}
		& \partial_t \m_0 - \alpha \m_0\times \partial_t \m_0
		= - (1 + \alpha^2) \m_0 \times \big\{ \mathcal{A}_0 \m_2 + \bh(\bx,\by) \big\},\\
		& \Delta_{\bx} U_0
		+ 2\mathrm{div}_{\by} (\nabla_{\bx} U_1)
		+ M_s \mathrm{div}_{\bx}\big(\m_0 \mathcal{X}_{\Omega}(\bx) \big)
		+ \mathrm{div}_{\by} \big(M_s \m_1 \mathcal{X}_{\Omega}(\bx)\big) = 0.
	\end{aligned}\right.
\end{equation}
Taking integration of \eqref{order of 1} with respect to $\by$ over $Y$,  the corresponding homogenized LLG equation can be obtained as
\begin{equation}\label{eqn:homogenized LLG system short}
	\left\{\begin{aligned}
		& \partial_t\m_0 - \alpha \m_0 \times \partial_t\m_0
		= 
		- \left(1+\alpha^2\right)\m_0 \times \bh_{\mathrm{eff}}^0,\\
		& \Delta U_0(\bx) = - M^0 \mathrm{div}_{\bx}\big(\m_0 \mathcal{X}_{\Omega}(\bx) \big),
	\end{aligned}\right.
\end{equation}
with $\bh_{\mathrm{eff}}^0(\bx) = \int_Y \bh(\bx,\by) \d \by$, which can be computed by
\begin{equation}\label{homogenized effective field in 3d}
	\bh_{\mathrm{eff}}^0
	= \mathrm{div}\left( \ba^0 \nabla \m_0 \right) 
	+ \mu^0 \nabla U^0 + \mathbf{m}_0 \cdot \mathbf{H}^0_{\mathrm{d}}
	- K^0 \left(\mathbf{m}_0 \cdot \mathbf{u}\right) \mathbf{u} 
	+ M^0 \mathbf{h}_{\mathrm{a}}.
\end{equation}
Here, the homogenized coefficients $\mathbf{a}^0 = \{a^0_{ij}\}_{1\le i,j\le3}$, $\mu^0$, $K^0$ and matrix $\mathbf{H}_{\mathrm{d}}^0$ are given by
\begin{equation}\label{homogenized coefficient}
	\left\{\begin{aligned}
		&a^0_{ij} = \int_Y\bigg( a_{ij} + \sum_{k= 1}^3 a_{ik}\frac{\partial \chi_j}{\partial y_k} \bigg) \d \boldsymbol{y},\qquad
		\mu^0 = \int_Y \mu(\boldsymbol{y}) \d \boldsymbol{y},\\
		&K^0 =
		\int_{Y}K(\mathbf{y}) \d \boldsymbol{y},
		\qquad \mathbf{H}_{\mathrm{d}}^0 =
		\int_{Y} \mu(\mathbf{y}) \nabla \bu(\by) \d \boldsymbol{y}.
	\end{aligned}\right.
\end{equation}

The second-order corrector $\m_2(\bx,\by)$ are given as follows. Combining \eqref{order of 1} with the homogenized equation \eqref{eqn:homogenized LLG system short}, it yields
\begin{equation}\label{Orthometric direction of A0m2}
	\m_0 \times \big\{ \mathcal{A}_0 \m_2 + \bh(\bx,\by) - \bh_{\mathrm{eff}}^0(\bx)\big\} = 0.
\end{equation}
Under the ansatz $\m_2 \cdot \m_0 =-\frac{1}{2} \abs{\m_1}^2$ in \eqref{geo property}, we have
\begin{equation}\label{eqn:substituting of m2}
	\m_0\cdot \mathcal{A}_0 \m_2
	= - \ba(\by) \nabla_{\by}\m_1 \cdot \nabla_y \m_1 - \m_1\cdot \mathcal{A}_0 \m_1.
\end{equation}
Using the orthogonal decomposition of $\mathcal{A}_0 \m_2$, together with \eqref{Orthometric direction of A0m2} and \eqref{eqn:substituting of m2}, it yields 
\begin{equation}\label{eqn:system m2 form3}
	\begin{aligned} 
		\mathcal{A}_0 \m_2 =& 
		- \m_0 \times \left( \m_0 \times  \mathcal{A}_0 \m_2 \right)  
		+ \left(\m_0\cdot \mathcal{A}_0 \m_2\right) \m_0\\
		=& - \left( \bh(\bx,\by) - \bh_{\mathrm{eff}}^0(\bx) \right)
		+
		\big\{ \m_0 \cdot \big( \bh(\bx,\by) - \bh_{\mathrm{eff}}^0(\bx)\big) \big\} \m_0\\
		& - \big( \m_1\cdot \mathcal{A}_0 \m_1 
		+ \ba(\by) \nabla_{\by}\m_1 \cdot \nabla_{\by} \m_1 \big) \m_0.
	\end{aligned}
\end{equation}
By the variable separation method, the solution of \eqref{eqn:system m2 form3} is given by
	\begin{equation}\label{second-order corrector}
 \begin{aligned}
		\m_2(\bx, \by) = &\bdiv_{\bx}\big(\boldsymbol{\theta}(\by) \nabla_{\bx} \m_0(\bx)\big)
		+ \Big\{\big(\boldsymbol{\theta}(\boldsymbol{y})+\frac{1}{2} \boldsymbol{\chi}(\boldsymbol{y}) \otimes \boldsymbol{\chi}(\boldsymbol{y})\big)\left|\nabla \mathbf{m}_0(\boldsymbol{x})\right|^2\Big\} \mathbf{m}_0(\boldsymbol{x})\\
		& + \big( I - \m_0(\bx)\otimes \m_0(\bx) \big)\mathcal{T}_{\mathrm{low}}(\bx, \by),
  \end{aligned}
	\end{equation}
 where $I$ is the identity transformation, $\otimes$ denotes the tensor product, and the low-order term $\mathcal{T}_{\mathrm{low}}$ is given by
 	\begin{equation*}
		\mathcal{T}_{\mathrm{low}}(\bx, \by) = 
		- \kappa(\by) (\m_0(\bx)\cdot \bu)\bu
		+
		\rho(\by) \nabla_{\boldsymbol{x}} U_0(\boldsymbol{x})
		+ \m_0(\bx)\cdot \boldsymbol{\Lambda}(\by)
		+ U^*(\by) \ha.
	\end{equation*}
Here, $U^*(\by)$ is defined in \eqref{first-order cell problem}. The auxiliary functions $\boldsymbol{\theta} = \{\theta_{ij}\}_{1\le i,j\le3}$,  $\rho$, $\boldsymbol{\Lambda}$, $\kappa$ are given by the second-order cell problem
\begin{equation}\label{second-order cell problem}
	\left\{ \begin{aligned}
		&\mathcal{A}_0\theta_{ij} = 
		a^0_{ij} - \bigg(a_{ij} + \sum_{k=1}^3 a_{ik}\frac{\partial\chi_j}{\partial y_k}\bigg)
		- \sum_{k=1}^3 \frac{\partial (a_{ik}\chi_j)}{\partial y_k}
		,\\
		&\mathcal{A}_0 \rho  = \mu (\by)-\mu^0,\quad
		\mathcal{A}_0 \boldsymbol{\Lambda} =
		\mu(\by) \nabla \bu(\by) - \mathbf{H}_{\mathrm{d}}^0,\\
  & \mathcal{A}_0\kappa = K(\by)-K^0,\\
		&\boldsymbol{\theta},\ \rho, \ \boldsymbol{\Lambda},\ \kappa, \quad \mbox{are $Y$-periodic}.
	\end{aligned} \right.
\end{equation}

\subsubsection{Homogenized equation in 2D case}
For the 2D case, we can apply the same argument for the multiscale equation \eqref{multi model}, \eqref{effective field} with the degenerated stray field \eqref{stray field 2D}. The homogenized equation reads as
\begin{equation*}
	\partial_t\m_0 - \alpha \m_0 \times \partial_t\m_0
	= 
	- \left(1+\alpha^2\right)\m_0 \times \bh_{\mathrm{eff}}^0,
\end{equation*}
where the homogenized effective field $\bh_{\mathrm{eff}}^0$ is given by
\begin{equation}\label{eqn:homogenized LLG system n=2}
	\begin{aligned}
		\bh_{\mathrm{eff}}^0
		=
		\mathrm{div}\left( \ba^0 \nabla \m_0 \right) 
		+ \widetilde{M}^0 \left(\mathbf{m}^{\varepsilon}\cdot \mathbf{e}_3 \right)\mathbf{e}_3
		- K^0 \left(\mathbf{m}_0 \cdot \mathbf{u}\right) \mathbf{u} 
		+ M^0 \mathbf{h}_{\mathrm{a}},
	\end{aligned}
\end{equation}
with $\mathbf{a}^0 = \{a^0_{ij}\}_{1\le i,j\le3}$, $\mu^0$, $K^0$ defined in \eqref{homogenized coefficient}, and $\widetilde{M}^0$ is given by
\begin{equation*}
	\widetilde{M}^0
	= \int_Y \mu(\boldsymbol{y}) M_s(\by) \d \boldsymbol{y}.
\end{equation*}
Moreover, the second order corrector can also be written in the form of \eqref{second-order corrector}, with the low-order term $\mathcal{T}_{\mathrm{low}}(\bx, \by)$ given by
\begin{equation*}
  \begin{aligned}
		\mathcal{T}_{\mathrm{low}}(\bx, \by) = &
		- \kappa(\by) \big(\m_0(\bx)\cdot \bu\big)\bu
  + \beta(\by) \big(\m_0(\bx)\cdot  \mathbf{e}_3\big) \mathbf{e}_3
		+
		\rho(\by) \nabla_{\boldsymbol{x}} U_0(\boldsymbol{x})\\
		& + \m_0(\bx)\cdot \boldsymbol{\Lambda}(\by)
		+ U^*(\by) \ha.
  \end{aligned}
\end{equation*}
Here, $\beta(\by)$ is a $Y$-periodic function satisfying the second-order cell problem
\begin{equation*}
    \mathcal{A}_0\beta(\by) = 
    \mu(\boldsymbol{y}) M_s(\by) - \widetilde{M}^0.
\end{equation*}

\subsection{Boundary value problems}
In this work, we consider two types of the boundary value problem: the Periodic problem and the Neumann problem. First, for $\Omega = [0,1]^n$, the multiscale problem \eqref{multi model} with periodic boundary setting is defined as
\begin{equation}\label{eqn: Periodic problem}
	\left\{\begin{aligned}
		&\partial_t \mathbf{m}^{\varepsilon}-\alpha \mathbf{m}^{\varepsilon} \times \partial_t \mathbf{m}^{\varepsilon}
		=-\left(1+\alpha^2\right) \m^\epsilon \times \mathbf{h}_{\mathrm{eff}}^{\varepsilon}\left(\mathbf{m}^{\varepsilon}\right)
		\quad \text {in } \Omega, \\
		&\mathbf{m}^{\varepsilon}(\boldsymbol{x}, 0) =\mathbf{m}_{\mathrm{init}}^{\varepsilon}(\boldsymbol{x}), \quad\left|\mathbf{m}_{\mathrm{init}}^{\varepsilon}(\boldsymbol{x})\right|=1 \quad \text {in } \Omega,\\
		&\mathbf{m}^{\varepsilon}(\boldsymbol{x}, t),\quad \mathbf{m}_{\mathrm{init}}^{\varepsilon}(\boldsymbol{x})  \quad \text {are periodic on } \partial \Omega ,
	\end{aligned}\right.
\end{equation}
where $\mathbf{m}_{\mathrm{init}}^{\varepsilon}(\boldsymbol{x})$ is initial condition, and the effective field $\mathbf{h}_{\mathrm{eff}}^{\varepsilon}\left(\mathbf{m}^{\varepsilon}\right)$ is given in \eqref{effective field}-\eqref{stray field 2D}.
The corresponding homogenized problem with periodic boundary can be written as
\begin{equation}\label{homogenized model periodic}
	\left\{\begin{aligned}
		&\partial_t \mathbf{m}_0-\alpha \mathbf{m}_0 \times \partial_t \mathbf{m}_0=-\left(1+\alpha^2\right) \mathbf{m}_0 \times \mathbf{h}_{\mathrm{eff}}^0\left(\mathbf{m}_0\right) \quad \text {in } \Omega, \\
		&\mathbf{m}_0(\boldsymbol{x}, 0) =\mathbf{m}_{\mathrm{init}}^0(\boldsymbol{x}), \quad\left|\mathbf{m}_{\mathrm{init}}^0(\boldsymbol{x})\right|=1 \quad \text {in } \Omega,\\
		&\mathbf{m}_0(\boldsymbol{x}, t),\quad \mathbf{m}_{\mathrm{init}}^0(\boldsymbol{x})  \quad \text {are periodic on } \partial \Omega ,
	\end{aligned}\right.
\end{equation}
where $\mathbf{m}_{\mathrm{init}}^0(\boldsymbol{x})$ is the initial condition, and the homogenized effective field $\mathbf{h}_{\mathrm{eff}}^0\left(\mathbf{m}_0\right)$ is defined in \eqref{homogenized effective field in 3d}, \eqref{eqn:homogenized LLG system n=2}.
The multiscale solution $\m^\epsilon$ can be approximated by $\m_0$  of \eqref{homogenized model periodic} and the first-order corrector $\m^1$ of \eqref{first-order corrector} in the following form
\begin{equation}\label{approximation of periodic case}
	\m^\epsilon(\bx) \approx \m_0(\bx) + \varepsilon\boldsymbol{\chi}\left(\frac{\bx}{\epsilon}\right) \nabla\mathbf{m}_0(\bx).
\end{equation}
The two-scale convergence of approximation \eqref{approximation of periodic case} has been proved in \cite{S:JMAA:2007,AD:PRSMPES:2015,CMT:DCDS:2018,CDMSZ:MMS:2022}.
Suppose only the exchange field is considered in \eqref{eqn: Periodic problem}, and $\m^\epsilon$ satisfies the certain boundedness assumption. In this case, it has been proved that \eqref{approximation of periodic case} has the convergence order $O(\epsilon)$ in the $L^2(\Omega)$ sense and remains uniformly bounded in the $H^1(\Omega)$ sense \cite{LR:CMS:2022}.

On the other hand, for an open bounded domain $\Omega$, the multiscale problem \eqref{multi model} with Neumann boundary setting is defined as
\begin{equation}\label{neumann}
	\left\{\begin{aligned}
		&\partial_t \mathbf{m}^{\varepsilon}-\alpha \mathbf{m}^{\varepsilon} \times \partial_t \mathbf{m}^{\varepsilon} =-\left(1+\alpha^2\right) \mathbf{m}^{\varepsilon} \times h_{\mathrm{eff}}^{\varepsilon}\left(\mathbf{m}^{\varepsilon}\right) \quad \text {in } \Omega,\\
		&\mathbf{m}^{\varepsilon}(\boldsymbol{x}, 0) =\mathbf{m}_{\mathrm{init}}^{\varepsilon}(\boldsymbol{x}), \quad\left|\mathbf{m}_{\mathrm{init}}^{\varepsilon}(\boldsymbol{x})\right|=1 \quad \text {in } \Omega,\\
		&\boldsymbol{\nu} \cdot \ba^{\varepsilon} \nabla \mathbf{m}^{\varepsilon}(\boldsymbol{x}, t) =0, \quad \text {on } \partial \Omega \times[0, T], \\
		&\boldsymbol{\nu} \cdot \ba^{\varepsilon} \nabla \mathbf{m}_{\mathrm{init}}^{\varepsilon}(\boldsymbol{x}) =0, \quad \text {on } \partial \Omega, \\
	\end{aligned}\right.
\end{equation}
where $\boldsymbol{\nu}$ represents the unit outer normal vector. The corresponding homogenized Neumann problem is given by
\begin{equation}\label{homogenized problem Neumann}
	\left\{\begin{aligned}
		&\partial_t \mathbf{m}_0-\alpha \mathbf{m}_0 \times \partial_t \mathbf{m}_0=-\left(1+\alpha^2\right) \mathbf{m}_0 \times h_{\mathrm{eff}}^{0}\left(\mathbf{m}_0\right) \quad \text { in } \Omega, \\
		&\mathbf{m}_{0}(\boldsymbol{x}, 0)=\mathbf{m}_{\mathrm{init}}^{0}(\boldsymbol{x}), \quad\left|\mathbf{m}_{\mathrm{init}}^{0}(x)\right|=1 \quad \text { in } \Omega,\\
		&\boldsymbol{\nu} \cdot \ba^{0} \nabla \mathbf{m}_0(\boldsymbol{x}, t) = 0, \quad \text { on } \partial \Omega \times[0, T], \\
		&\boldsymbol{\nu} \cdot \ba^{0} \nabla \mathbf{m}_{\mathrm{init}}^{0}(\boldsymbol{x})=0, \quad \text { on } \partial \Omega. \\
	\end{aligned}\right.
\end{equation}
For the Neumann problem, the convergence order in the $H^1$ sense of \eqref{approximation of periodic case} has been derived without  considering the boundary layer effect \cite{CLS:apa:2022}.
Moreover, to approximate the multiscale solution more accurately, the 
Neumann corrector $\boldsymbol{\Phi}^\epsilon= \{\Phi_i^\epsilon\}_{i=1}^n$ was introduced
\begin{equation}\label{neumann correct}
	\left\{\begin{aligned}
		&\operatorname{div}\left(\ba^{\varepsilon} \nabla \Phi_i^\epsilon\right)  =\operatorname{div}\left(\ba^0 \nabla x_i\right) = 0 \quad \text { in } \Omega, \\
		&\boldsymbol{\nu} \cdot \ba^{\varepsilon} \nabla \Phi_i^\epsilon  
		=\boldsymbol{\nu} \cdot \ba^{0} \nabla x_i \quad \text { on } \partial \Omega .
	\end{aligned}\right.
\end{equation}
Here, $x_i$ is the $i$-th component of the spatial variable $\mathbf{x}$, which can be seen as the homogenized solution of $\Phi_i^\epsilon$. To make $\Phi_i^\epsilon$ be unique up to a constant, it is assumed that $\Phi_i^\epsilon(\tilde{\boldsymbol{x}})-\tilde{\boldsymbol{x}}=0 $ for some $\tilde{\boldsymbol{x}} \in \Omega $.
Utilizing the Neumann corrector $\boldsymbol{\Phi}^\epsilon$, \cite{CLS:apa:2022} provided a sharper estimation for the approximation
\begin{equation}\label{approximation of Neumann case}
	\begin{gathered}
		\m^\epsilon \approx \m_0 + (\boldsymbol{\Phi}^\epsilon - \bx) \nabla\mathbf{m}_0.
	\end{gathered}
\end{equation}
In this work, the full system in \eqref{eqn: Periodic problem} is considered, and the convergence order of approximation \eqref{approximation of periodic case} in the $H^1(\Omega)$ sense is derived.

\subsection{Initial data}
To examine the asymptotic behavior of the solution from the multiscale system to the homogenized system under \eqref{approximation of periodic case} and \eqref{approximation of Neumann case}, it is natural to propose the same initial data, i.e. $\m_{\mathrm{init}}^\varepsilon = \m_{\mathrm{init}}^0$. However, this simple setting may violate the consistency of the boundary condition, because of $\boldsymbol{\nu}\cdot\ba^\epsilon\nabla\m_{\mathrm{init}}^\varepsilon \neq \boldsymbol{\nu}\cdot\ba^0\nabla\m_{\mathrm{init}}^0$ for the Neumann problem.
To address this problem, we can apply the following strategy to determine the proper initial data. Given a smooth function $\m_{\mathrm{init}}^0\in S^2$ of the homogenized problem such that the consistency condition holds, then $\m_{\mathrm{init}}^\epsilon\in S^2$ is defined by the non-homogeneous multiscale harmonic mapping
\begin{equation}\label{harmonic mapping}
	\left\{\begin{aligned}
		&\m_{\mathrm{init}}^\varepsilon\times\left(\mathcal{A}_\varepsilon\m_{\mathrm{init}}^\varepsilon-\mathcal{A}^0\m_{\mathrm{init}}^0\right)=0,\\
        & \abs{\m_{\mathrm{init}}^\varepsilon}=1,\\
	&\boldsymbol{\nu}\cdot\ba^\epsilon\nabla\m_{\mathrm{init}}^\varepsilon = 0,\qquad
        \text{or}\qquad \m_{\mathrm{init}}^\varepsilon \text{ is periodic on $\partial\Omega$.}
        \end{aligned}\right.
\end{equation}
By taking the outer product of both sides of \eqref{harmonic mapping} with $\m_{\mathrm{init}}^\varepsilon(\bx)$, 
\eqref{harmonic mapping} can be rewritten as
\begin{equation}\label{cross product}
	\mathcal{A}_\varepsilon \mathbf{m}_{\mathrm{init}}^{\varepsilon}=\mathbf{m}_{\mathrm{init}}^{\varepsilon}\times\left(\mathbf{m}_{\mathrm{init}}^{\varepsilon}\times\mathcal{A}^0\mathbf{m}_{\mathrm{init}}^0\right)-\ba^{\varepsilon}|\nabla \mathbf{m}_{\mathrm{init}}^{\varepsilon}|^2 \mathbf{m}_{\mathrm{init}}^{\varepsilon}.
\end{equation}
Then, \eqref{cross product} can be simplified with the vector triple product formula
\begin{equation}\label{eqn:vec product}
	\boldsymbol{\eta}\times (\boldsymbol{\iota} \times \boldsymbol{\zeta}) = (\boldsymbol{\eta} \cdot \boldsymbol{\zeta}) \boldsymbol{\iota} - (\boldsymbol{\eta}\cdot\boldsymbol{\iota}) \boldsymbol{\zeta},\qquad \forall \boldsymbol{\eta},\boldsymbol{\iota},\boldsymbol{\zeta}.
\end{equation}
Finally, the variational form of \eqref{cross product} can be expressed as
\begin{equation}\label{projection method}
	\begin{aligned}		
    	\big(\ba^{\varepsilon} \nabla\mathbf{m}_{\mathrm{init}}^{\varepsilon},\, \nabla \mathbf{v}\big)
     =&-(\boldsymbol{F},\,\mathbf{v})
     +
     \big(\left(\mathbf{m}_{\mathrm{init}}^{\varepsilon}\cdot \boldsymbol{F}\right)\mathbf{m}_{\mathrm{init}}^{\varepsilon},\,
     \mathbf{v}\big)\\
		&
  -\left(\Big\{\sum a^{\varepsilon}_{ij}\left(\partial_i \mathbf{m}_{\mathrm{init}}^{\varepsilon}\cdot \partial_j \mathbf{m}_{\mathrm{init}}^{\varepsilon}\right)\Big\}\mathbf{m}_{\mathrm{init}}^{\varepsilon},\,\mathbf{v}\right),
	\end{aligned}
	\quad \forall \mathbf{v} \in H^1(\Omega), 	
\end{equation}
where $\boldsymbol{F}=\bdiv(\ba^0\nabla\boldsymbol{m}_{\mathrm{init}}^0)$.
The equation \eqref{projection method} is named as the projection method, such that the computed magnetic moment is located on the unit sphere. However, due to the nonlinearity, \eqref{projection method} cannot be directly solved in closed form. 
Therefore, iteration methods, such as the Picard iteration, can be applied to obtain the desired solution. These methods not only require multiple iterations but also require suitable initial guesses to ensure convergence, significantly increasing the computational complexity.

To overcome the above difficulties, we design the expansion method, where the two-scale method is utilized on the initial data.
Given a smooth function $\mathbf{m}_{\mathrm{init}}^0(\mathbf{x})\in S^2$ of the homogenized problem,  $\mathbf{m}_{\mathrm{init}}^{\varepsilon}$ is defined as
\begin{equation}\label{expansion method}
    \mathbf{m}_{\mathrm{init}}^{\varepsilon}(\mathbf{x})=\mathbf{m}_{\mathrm{init}}^0(\mathbf{x})
    +\widetilde{\mathbf{m}}_{\mathrm{init}}^c\left(\mathbf{x},\frac{\bx}{\epsilon}\right),
\end{equation}
where $\widetilde{\mathbf{m}}_{\mathrm{init}}^c(\mathbf{x},\mathbf{y})$ is the correction term related to the boundary conditions of the multiscale equation, which can be represented as
\begin{equation}\label{define corrector for initial data}
	\widetilde{\mathbf{m}}_{\mathrm{init}}^c(\mathbf{x},\mathbf{y})=\left\{
		\begin{aligned}
			\varepsilon\boldsymbol{\chi}\left(\frac{\bx}{\epsilon}\right)\nabla\mathbf{m}_{\mathrm{init}}^{0}(\bx), \qquad &\text{with Periodic boundary},\\
			(\boldsymbol{\Phi}^\epsilon-\boldsymbol{x})\nabla\mathbf{m}_{\mathrm{init}}^{0}(\bx), \qquad &\text{with Neumann boundary}.
		\end{aligned}
	\right.
\end{equation}
It can be proved that the initial data $\mathbf{m}_{\mathrm{init}}^{\varepsilon}$ in \eqref{expansion method} satisfies the consistency of the boundary condition.
Specifically, for the Neumann problem, we have $\boldsymbol{\nu}\cdot\ba^\epsilon\nabla\m_{\mathrm{init}}^\varepsilon = \boldsymbol{\nu}\cdot\ba^0\nabla\m_{\mathrm{init}}^0$. However, $\mathbf{m}_{\mathrm{init}}^{\varepsilon}$ is not sphere-valued, which satisfies 
\begin{equation*}
    \abs{\mathbf{m}_{\mathrm{init}}^{\varepsilon}}^2
    =
    \abs{\mathbf{m}_{\mathrm{init}}^{0}}^2
    +
    \abs{\widetilde{\mathbf{m}}_{\mathrm{init}}^c}^2
    =
    1 + \abs{\widetilde{\mathbf{m}}_{\mathrm{init}}^c}^2.
\end{equation*}
Then, we have
\begin{equation*}
    \abs{\mathbf{m}_{\mathrm{init}}^{\varepsilon}}-1 = \frac{\abs{\widetilde{\mathbf{m}}_{\mathrm{init}}^c}^2}{\abs{\mathbf{m}_{\mathrm{init}}^{\varepsilon}}+1},
\end{equation*}
where   $\abs{\mathbf{m}_{\mathrm{init}}^{\varepsilon}}-1$ is non-zero but $O(\abs{\widetilde{\mathbf{m}}_{\mathrm{init}}^c}^2)$. 
Compared with the projection method \eqref{projection method}, the expansion method provides the initial data of the multiscale system in explicit form, significantly improving the computational efficiency.
In this work, we apply the initial setting \eqref{harmonic mapping} for the theoretical analysis, and apply setting \eqref{expansion method} for the numerical simulations. 

\subsection{Numerical scheme}

In this subsection, we design an improved implicit numerical  scheme to solve the homogenized LLG equation  \eqref{eqn:homogenized LLG system short}. 
The same scheme can be applied to solve the multiscale LLG equation \eqref{multi model}, the first-order cell problem \eqref{first-order cell problem}, and the second-order cell problem \eqref{second-order cell problem}.

In the spatial discretization, the finite element method is employed for each component of the homogenized magnetic moment $\m_0$. 
Given the lowest order finite element space $V_h \subset W^{1,2}(\Omega; \mathbb{R}^3) $ subordinate to the triangulation $T_h$ of the domain $\Omega$, the variational form reads as follows: To find $\m_{0,h}\in V_h$, such that
\begin{equation}\label{spatial discretization}
    \begin{aligned}
        &\left(\partial_t \m_{0,h}, \mathbf{v}_h\right)_h-\alpha\left(\m_{0,h} \times \partial_t \m_{0,h}, \mathbf{v}_h\right)_h \\
        =&-\left(1+\alpha^2\right)\left(\m_{0,h} \times \mathbf{h}^0_{\mathrm{eff}}\left(\m_{0,h}\right), \mathbf{v}_h\right)_h,
    \end{aligned}
\qquad \forall \mathbf{v}_h \in \mathbf{V}_h,
\end{equation}
where $(\cdot, \cdot)_h$ represents the discrete inner product in $L^2\left(\Omega\right)$. 

In the temporal discretization, 
implicit and semi-implicit schemes have often been used
due to their unconditional numerical stability \cite{JJLL:NMPDE:2022}. 
In more details, the implicit schemes preserve the magnetization's magnitude and the Lyapunov structure of the LLG equation. The work \cite{BP:SJNA:2006} provided an effective implicit temporal discretization scheme
\begin{equation}\label{implicit scheme}
	\begin{aligned}
		&\frac{1}{\Delta t}\big(\mathbf{m}_{0,h}^{j+1}, \mathbf{v}_h\big)_h-\frac{\alpha}{\Delta t}\big(\mathbf{m}_{0,h}^j \times \mathbf{m}_h^{j+1}, \mathbf{v}_h\big)_h \\
		=&\frac{1}{\Delta t}\big(\mathbf{m}_{0,h}^{j}, \mathbf{v}_h\big)_h-\left(1+\alpha^2\right)\big(\overline{\mathbf{m}}_{0,h}^{j+ 1/ 2} \times \mathbf{h}_{\mathrm{eff}}^0(\overline{\mathbf{m}}_{0,h}^{j+1/2}), \mathbf{v}_h\big)_h,
	\end{aligned}
\qquad \forall \mathbf{v}_h \in \mathbf{V}_h, 
\end{equation}
where $\Delta t$ is the time step size, $\mathbf{m}_{0,h}^j$ and $\overline{\mathbf{m}}_{0,h}^{j+1/2}$ are defined as
\begin{equation*}
    \begin{aligned}
        \mathbf{m}_{0,h}^j&:=\mathbf{m}_{0,h}^j(\bx,j\Delta t),\\
        \overline{\mathbf{m}}_{0,h}^{j+1/2}&:=\big(\mathbf{m}_{0,h}^{j+1}+\mathbf{m}_{0,h}^j\big)/2,
    \end{aligned}
    \qquad j=0,1,2,\cdots
\end{equation*}
To handle the nonlinear system in \eqref{implicit scheme}, a implicit iteration scheme was employed in each time step
\begin{equation}\label{fixed-point iteration 1}
    \begin{aligned}
        &\frac{1}{\Delta t}\big(\mathbf{m}_{0,h}^{j+1, \ell+1}, \mathbf{v}_h\big)_h-\frac{\alpha}{\Delta t}\big(\mathbf{m}_{0,h}^j \times \mathbf{m}_{0,h}^{j+1,\ell+1}, \mathbf{v}_h\big)_h\\
		&+\frac{1+\alpha^2}{4}\big(\mathbf{m}_{0,h}^{j+1, \ell+1} \times \mathbf{h}_{\mathrm{eff}}^0(\mathbf{m}^{j+1, \ell}_{0,h}) , \mathbf{v}_h\big)_h +\frac{1+\alpha^2}{4}\big(\mathbf{m}_{0,h}^{j+1, \ell+1} \times \mathbf{h}^0_{\mathrm{eff}}(\mathbf{m}^{j}_{0,h}), \mathbf{v}_h\big)_h\\
		&+\frac{1+\alpha^2}{4}\big(\mathbf{m}_{0,h}^j \times \mathbf{h}_{\mathrm{eff}}^0( \mathbf{m}^{j+1, \ell+1}_{0,h}) , \mathbf{v}_h\big)_h \\
        =&\frac{1}{\Delta t}\big(\mathbf{m}_{0,h}^j, \mathbf{v}_h\big)_h-\frac{1+\alpha^2}{4}\big(\mathbf{m}_{0,h}^j \times \mathbf{h}^0_{\mathrm{eff}}(\mathbf{m}_{0,h}^{j}) , \mathbf{v}_h\big)_h, \qquad	\forall \mathbf{v}_h \in \mathbf{V}_h,\quad \ell=0,1,2,\cdots
    \end{aligned}
    \vspace{-1mm}
\end{equation}
where $\ell$ denotes the $\ell$-th iteration step.
The equation \eqref{fixed-point iteration 1} has a unique solution when the following condition is satisfied
\begin{equation}
	\Delta t\leq \frac{h^2}{10(1+\alpha^2)}.
\end{equation}
where $h$ is the spatial mesh size. However, convergence efficiency has always been the emphasis to optimize the above scheme.

Based on the temporal discretization scheme \eqref{implicit scheme}, an improved implicit scheme is designed, and a novel semi-implicit iteration scheme is employed to handle the corresponding nonlinear system.
The semi-implicit iteration scheme is given by
\begin{equation}\label{fixed-point iteration 2}
    \begin{aligned}
        &\frac{1}{\Delta t}\big(\mathbf{m}_{0,h}^{j+1, \ell+1}, \mathbf{v}_h\big)_h-\frac{\alpha}{\Delta t}\big(\mathbf{m}_{0,h}^j \times \mathbf{m}_{0,h}^{j+1,\ell+1}, \mathbf{v}_h\big)_h\\
		&+\frac{1+\alpha^2}{4}\big(\mathbf{m}_{0,h}^{j+1, \ell} \times \mathbf{h}_{\mathrm{eff}}^0(\mathbf{m}_{0,h}^{j+1, \ell+1}) , \mathbf{v}_h\big)_h +\frac{1+\alpha^2}{4}\big(\mathbf{m}_{0,h}^{j+1, \ell+1} \times \mathbf{h}^0_{\mathrm{eff}}(\mathbf{m}_{0,h}^{j}), \mathbf{v}_h\big)_h\\
		&+\frac{1+\alpha^2}{4}\big(\mathbf{m}_{0,h}^j \times \mathbf{h}_{\mathrm{eff}}^0( \mathbf{m}_{0,h}^{j+1, \ell+1}) , \mathbf{v}_h\big)_h \\
        =&\frac{1}{\Delta t}\big(\mathbf{m}_{0,h}^j, \mathbf{v}_h\big)_h-\frac{1+\alpha^2}{4}\big(\mathbf{m}_{0,h}^j \times \mathbf{h}_{\mathrm{eff}}^0(\mathbf{m}_{0,h}^{j}) , \mathbf{v}_h\big)_h, \qquad	\forall \mathbf{v}_h \in \mathbf{V}_h, \quad \ell=0,1,2,\cdots
    \end{aligned}
\end{equation}


For the original scheme \eqref{implicit scheme},\eqref{fixed-point iteration 1} and the improved scheme \eqref{implicit scheme},\eqref{fixed-point iteration 2}, the converge time and iteration steps of the corresponding nonlinear system are depicted in \cref{tab: time and niter}. 
The table shows that the improved scheme significantly reduces the number of iteration steps and relaxes the constraints on the time step size, enabling it to handle the multiscale LLG equation.

\begin{table}[h!]
    \centering
	\begin{tabular}{|c|c|c|c|c|}
		\hline
		\multirow{2}*{$\Delta t$} & \multicolumn{2}{c|}{original scheme \cite{BP:SJNA:2006}}  &\multicolumn{2}{c|}{improved scheme}\\ \cline{2-5}  
                        &  time(s)  &  iteration steps  
                        &  time(s) &  iteration steps \\         
		\hline $10^{-4}$ & - & - & 49.2 & 9\\
		\hline $10^{-5}$ & 352.2 & 58 & 33.4 & 6\\
		\hline $10^{-6}$ & 54.8 & 10 & 22.8 & 4\\
		\hline
	\end{tabular}
	\caption{Comparison of the converge time and iteration steps for two numerical schemes.
   The spatial mesh size is $h=\frac{1}{180}$, $-$ represents that the scheme does not converge under the time step size $\Delta t$.}
    \label{tab: time and niter}
    \vspace{-4mm}
\end{table}

\subsection{The flowchart of the algorithm}
By combining the aforementioned two-scale method with the corresponding numerical scheme, the algorithm for solving \eqref{multi model} is summarized in Algorithm \ref{flowchart}.

\begin{algorithm}[h!]
      \caption{The flowchart of the algorithm to solve \eqref{multi model} with two-scale method.}
			\label{flowchart}
      \begin{algorithmic}[1] 
        \REQUIRE  $\mathbf{m}_{0,h}^0=\mathbf{m}_{\mathrm{init}}^0(\boldsymbol{x})$, $T=N_t k$, all parameters need; 
	      \ENSURE $\big\{\widetilde{\mathbf{m}}^{\varepsilon,j}_h\big\}_{j=0,1,\cdots,N_t}$; 
        
        \STATE Get the initial value $\widetilde{\mathbf{m}}_{\mathrm{init}}^{\varepsilon}(\mathbf{x})$ of the multiscale LLG equation by \eqref{expansion method};
        \STATE Solve the first-order cell problem \eqref{first-order cell problem} to get the auxiliary functions $\boldsymbol{\chi}$ and $U^*$;
        \STATE Compute the homogenized coefficient $\mathbf{a}^0,\mu^0,K^0$ and $\mathbf{H}_{\mathrm{d}}^0$ based on \eqref{homogenized coefficient};
        \STATE Solve the second-order cell problem \eqref{second-order cell problem} to get the auxiliary functions $\boldsymbol{\theta},\rho,\boldsymbol{\Lambda}$ and $\kappa$;
        \STATE Let $j=0$;
        \WHILE {$j<N_t$}
            \STATE Let $\ell=0,\ \mathbf{m}_{0,h}^{j+1,0}=\mathbf{m}_{0,h}^j$;
            \STATE Calculate $\mathbf{m}_{0,h}^{j+1,1}$ by solving \eqref{fixed-point iteration 2};
          \WHILE {$\big\|\mathbf{m}_{0,h}^{j+1, \ell+1} - \mathbf{m}_{0,h}^{j+1, \ell}\big\|\ge$ threshold}
              \STATE Let $\ell=\ell+1$;
              \STATE Calculate $\mathbf{m}_{0,h}^{j+1,\ell+1}$ by solving \eqref{fixed-point iteration 2};
          \ENDWHILE
          \STATE Update $\mathbf{m}_{0,h}^{j+1}=\mathbf{m}_{0,h}^{j+1,\ell}$;
          \STATE Let $j=j+1$;
        \ENDWHILE
        \STATE Utilize \eqref{first-order corrector} to get the first-order correctors $\m_1$ and $U_1$;
        \STATE Use \eqref{second-order corrector} to get the second-order corrector $\m_2$;
        \STATE Assemble the approximate solution $\{\widetilde{\mathbf{m}}^{\varepsilon,j}_h\}_{j=0,1,\cdots,N_t}$ based on \eqref{def two scale approximate sln}.
    \end{algorithmic}
\end{algorithm}
\begin{remark}
  As the auxiliary functions are independent of the time variable, the cell problems \eqref{first-order cell problem} and \eqref{second-order cell problem} only need to be computed once throughout the process.
\end{remark}

\section{Convergence analysis under different effective fields and boundary corrections}\label{sec3}

In this section, we present some theoretical results for both the Periodic and Neumann problems. 
For the Periodic problem, the new convergence results are presented in \cref{thm: convergence result of periodic case} of \cref{convergence result}, and the corresponding detailed proof is given in \cref{prof thm1}. 
For the Neumann problem, the results are presented in \cref{thm:neumann} and \cref{prop:neumann} of \cref{convergence result}. 
Theoretical results for both problems will be verified by the numerical experiments in next section.

\subsection{Convergence results}\label{convergence result}
In order to state the results, the following assumptions are firstly introduced:
\begin{enumerate}[(I).]
	\item\label{assumption: coefficient}\textbf{Coefficients.} The matrix $\mathbf{a}(\by)$ possesses symmetry, uniform coercivity, and boundedness, that is, there exist positive constants $a_{\mathrm{min}},a_{\mathrm{max}}>0$, such that $a_{\mathrm{min}}\le \mathbf{a}(\by)\le a_{\mathrm{max}}$. Moreover, the periodic coefficients satisfy
	\begin{equation*}
		\ba(\by),\, M_s(\by)\,\text{ and }\, \gamma(\by)\in C^2(Y).
	\end{equation*}
	Additionally, the auxiliary functions defined in \eqref{first-order cell problem} and \eqref{second-order cell problem} satisfy
	\begin{equation*}
		\boldsymbol{\chi}(\by),\, U^*(\by),\,\boldsymbol{\theta}(\by),\,
		\boldsymbol{\Lambda}(\by)\,\text{ and }\, \rho(\by)\in C^2(Y).
	\end{equation*}
	For simplicity, the constant $C_{\mathrm{coe}}$ denotes the shared $C^2(Y)$ upper bound of the above periodic coefficients and auxiliary functions.
	\item\label{assumption: initial data}\textbf{Initial data.} The initial data $\m_{\mathrm{init}}^0(\bx)\in C^4(\bar{\Omega})$ and $\m_{\mathrm{init}}^\epsilon(\bx)\in C^2(\Omega)$ satisfy  \eqref{harmonic mapping}. Moreover, the following estimate holds 
 \begin{equation*}
     \big\Vert \m_{\mathrm{init}}^\epsilon(\bx) - \big(\m_{\mathrm{init}}^0(\bx) + \widetilde{\mathbf{m}}_{\mathrm{init}}^c(\mathbf{x})\big)
     \big\Vert_{H^1(\Omega)} \le C_{\mathrm{coe}} \epsilon.
 \end{equation*}
 Here, corrector $\widetilde{\mathbf{m}}_{\mathrm{init}}^c(\mathbf{x})$ is defined in \eqref{define corrector for initial data}, and $\m_{\mathrm{init}}^0(\bx)$ is bounded by
	\begin{equation*}
		\Vert \m_{\mathrm{init}}^0 \Vert_{C^4(\bar{\Omega})}
		\le C_{\mathrm{coe}}.
	\end{equation*} 
 \item\label{assumption: boundary}\textbf{Boundary.} For the Neumann problem on the open bounded domain $\Omega$, the boundary $\partial\Omega$ satisfies $\partial\Omega \in C^{1,1}$.
\end{enumerate}

\subsubsection{Periodic problem} 
The theoretical results for the Periodic problem are given as follows
\begin{theorem}\label{thm: convergence result of periodic case}
	Let $\boldsymbol{m}^{\varepsilon} \in L^{\infty}\left(0, T ; H^2(\Omega)\right)$ be the unique solution of the multiscale LLG equation \eqref{eqn: Periodic problem}
	and $ \boldsymbol{m}_0 \in L^{\infty}\left(0, T ; H^6(\Omega)\right)$ be the unique solutions of the homogenized LLG equation \eqref{homogenized model periodic},
	respectively. When $n=3$, there exists some $T^* \in(0, T]$ independent of $\varepsilon$, such that for any $t \in\left(0, T^*\right)$, it holds
	\begin{equation}\label{convergence estimate of periodic}
		\begin{gathered}
			\left\|\boldsymbol{m}^{\varepsilon}(\bx,t)-\boldsymbol{m}_0(\bx,t)\right\|_{L^2(\Omega)} \leq C \varepsilon^{\frac{5}{6}}\ln (\varepsilon^{-1} + 1),\\
			\left\|\boldsymbol{m}^{\varepsilon}(\bx,t)-\boldsymbol{m}_0(\bx,t)-\varepsilon\boldsymbol{\chi}\left(\frac{\bx}{\epsilon}\right)\nabla\mathbf{m}_0(\bx,t)\right\|_{H^1(\Omega)} \leq C \varepsilon^{\frac{1}{2}}\ln (\varepsilon^{-1} + 1).
		\end{gathered}
	\end{equation}
	Furthermore, when $n=2$, there exists some $T^{**} \in(0, T]$ independent of $\varepsilon$, such that for any $t \in\left(0, T^{**}\right)$, it holds
	\begin{equation}\label{value of sigma 2}
		\begin{gathered}
			\left\|\boldsymbol{m}^{\varepsilon}(\bx,t)-\boldsymbol{m}_0(\bx,t)\right\|_{L^2(\Omega)} \leq C \varepsilon^{1},\\
			\left\|\boldsymbol{m}^{\varepsilon}(\bx,t)-\boldsymbol{m}_0(\bx,t)-\varepsilon\boldsymbol{\chi}\left(\frac{\bx}{\epsilon}\right)\nabla\mathbf{m}_0(\bx,t)\right\|_{H^1(\Omega)} \leq C \varepsilon^{1}.
		\end{gathered}
	\end{equation}
    In both case, the constant $C$ depends on $C_{\mathrm{coe}}, a_{\min }$ and $a_{\max }$ given in Assumptions \eqref{assumption: coefficient}-\eqref{assumption: initial data}, but is independent of $\varepsilon$.
\end{theorem}

\subsubsection{Neumann problem}
For the Neumann problem, the Neumann corrector $\boldsymbol{\Phi}^\epsilon= \{\Phi_i^\epsilon\}_{i=1}^n$ defined in \eqref{neumann correct} is employed to avoid the convergence deterioration on the boundary.
\begin{proposition}\label{thm:neumann}
	Let $\boldsymbol{m}^{\varepsilon} \in L^{\infty}\left(0, T ; H^2(\Omega)\right)$ be the unique solution of the multiscale LLG equation \eqref{neumann}
	and $ \boldsymbol{m}_0 \in L^{\infty}\left(0, T ; H^6(\Omega)\right)$ be the unique solutions of the homogenized LLG equation \eqref{homogenized problem Neumann},
	respectively. When $n=3$, there exists some $T^* \in(0, T]$ independent of $\varepsilon$, such that for any $t \in\left(0, T^*\right)$, it holds
	\begin{equation*}
		\begin{gathered}
			\left\|\boldsymbol{m}^{\varepsilon}(\bx,t)-\boldsymbol{m}_0(\bx,t)\right\|_{L^2(\Omega)} \leq C \varepsilon^{\frac{5}{6}}\ln (\varepsilon^{-1} + 1),\\
			\left\|\boldsymbol{m}^{\varepsilon}(\bx,t)-\boldsymbol{m}_0(\bx,t)
			-(\boldsymbol{\Phi}^\epsilon-\boldsymbol{x})\nabla\mathbf{m}_0(\bx, t)\right\|_{H^1(\Omega)} \leq C \varepsilon^{\frac{1}{2}}\ln (\varepsilon^{-1} + 1).
		\end{gathered}
	\end{equation*}
	Furthermore, when $n=2$, there exists some $T^{**} \in(0, T]$ independent of $\varepsilon$, such that for any $t \in\left(0, T^{**}\right)$, it holds
	\begin{equation*}
		\begin{gathered}
			\left\|\boldsymbol{m}^{\varepsilon}(\bx, t)
			-\boldsymbol{m}_0(\bx, t)\right\|_{L^2(\Omega)} \leq C \varepsilon^{1},\\
			\left\|\boldsymbol{m}^{\varepsilon}(\bx, t)-\boldsymbol{m}_0(\bx, t)
			-(\boldsymbol{\Phi}^\epsilon-\boldsymbol{x})\nabla\mathbf{m}_0(\bx, t)\right\|_{H^1(\Omega)} \leq C \varepsilon^{1}.
		\end{gathered}
	\end{equation*}
    In both case, the constant $C$ depends on $C_{\mathrm{coe}}, a_{\min }$ and $a_{\max }$ given in Assumptions \eqref{assumption: coefficient}-\eqref{assumption: initial data}, but is independent of $\varepsilon$.
\end{proposition}

It is worth noting that a multiscale elliptic problem with natural boundary conditions about $\boldsymbol{\Phi}^\epsilon$ is given in \eqref{neumann correct}. 
Moreover, the spatial variable $\bx$ can be seen as the homogenized solution of this problem.
By employing the elliptic homogenization theory, as outlined in Theorems 3.3.5 and 3.5.3 of \cite{S:SIP:2018}, the following inequality can be obtained 
\begin{equation}\label{phi inequality}
	\left\Vert \boldsymbol{\Phi}^\epsilon(\bx) - \bx - \epsilon\boldsymbol{\chi}\left(\frac{\bx}{\epsilon}\right) \right\Vert_{H^1(\Omega)} \le C \epsilon^{1/2}.
\end{equation}
Substituting \eqref{phi inequality} into Theorem \ref{thm:neumann}, the following proposition can be derived.
\begin{proposition}\label{prop:neumann}
	Under the condition in Theorem \ref{thm:neumann}, when $n=3$, there exists some $T^* \in(0, T]$ independent of $\varepsilon$, such that for any $t \in\left(0, T^*\right)$, it holds
	\begin{equation*}
		\begin{gathered}
			\left\|\boldsymbol{m}^{\varepsilon}(\bx,t)-\boldsymbol{m}_0(\bx,t)
			-\epsilon\boldsymbol{\chi}\left(\frac{\bx}{\epsilon}\right)\nabla\mathbf{m}_0(\bx, t)\right\|_{H^1(\Omega)} \leq C \varepsilon^{\frac{1}{2}}\ln (\varepsilon^{-1} + 1).
		\end{gathered}
	\end{equation*}
	Furthermore, when $n=2$, there exists some $T^{**} \in(0, T]$ independent of $\varepsilon$, such that for any $t \in\left(0, T^{**}\right)$, it holds
	\begin{equation*}
		\begin{gathered}
			\left\|\boldsymbol{m}^{\varepsilon}(\bx, t)-\boldsymbol{m}_0(\bx, t)
			-\epsilon\boldsymbol{\chi}\left(\frac{\bx}{\epsilon}\right)\nabla\mathbf{m}_0(\bx, t)\right\|_{H^1(\Omega)} \leq C \varepsilon^{1/2}.
		\end{gathered}
	\end{equation*}
    In both case, the constant $C$ depends on $C_{\mathrm{coe}}, a_{\min }$ and $a_{\max }$ given in Assumptions \eqref{assumption: coefficient}-\eqref{assumption: initial data}, but is independent of $\varepsilon$.
\end{proposition}
When employing the classical two-scale approximation to the multiscale LLG equation \eqref{neumann} under Neumann boundary condition, the convergence near the boundary will exhibit a degradation and a $1/2$-order loss of convergence order.
\noindent The detailed proof of the above propositions can refer to \cite{CLS:apa:2022}.
\begin{remark}
	Comparing \eqref{convergence estimate of periodic} with \eqref{value of sigma 2}, it can be found that there are $1/6$-order loss in the $L^2$ norm and $1/2$-order loss in the $H^1$ norm. This degradation of convergence order is caused by boundary layer effects resulting from the zero extension of the stray fields \eqref{stray field 3D}.
 
    Specifically, \cref{tab} presents the convergence orders of the 2D problems \eqref{eqn: Periodic problem}, \eqref{neumann} with the exchange field and degenerated stray field, where the two-scale corrector and the Neumann corrector are employed.
    
    \begin{table}[h!]
        \centering
        \begin{tabular}{c|c|c|c}
        \toprule[1.5pt]
             \makecell[c]{Boundary \\Condition }  & Approximation
             & Norm & \makecell[c]{Convergence \\Order }\\ 
             \hline \rule{0pt}{16pt}
             \multirow{2}{*}{Periodic} 
             & \multirow{3}{*}{\vspace{-3mm}$\boldsymbol{m}_0(\bx)
    			+\epsilon\boldsymbol{\chi}(\bx/\epsilon)\nabla\mathbf{m}_0(\bx)$ }
             & $L^2(\Omega)$
             & $\mathcal{O}( \epsilon)$ 
             \cite{LR:CMS:2022}\\
             \cline{3-4} \rule{0pt}{16pt}
              & 
              & $H^1(\Omega)$ & $\mathcal{O}( \epsilon)$ (this work) \\
               \cline{1-1} \cline{3-4}
               \rule{0pt}{16pt}
             \multirow{2}{*}{Neumann} 
             &  
             & $L^2(\Omega)$
             & $\mathcal{O}( \epsilon^{1/2})$ 
             \cite{CLS:apa:2022}\\
             \cline{2-4} \rule{0pt}{16pt}
              & $\boldsymbol{m}_0(\bx)
    			+\big(\boldsymbol{\Phi}^\epsilon(\bx) - \bx\big)\nabla\mathbf{m}_0(\bx)$ 
              & $H^1(\Omega)$ & $\mathcal{O}( \epsilon)$
              \cite{CLS:apa:2022}\\
              \bottomrule[1.5pt]
        \end{tabular}
        \caption{Demonstration of  convergence order for the 2D problem  with the exchange field and degenerated stray field.}
        \label{tab}
    \end{table}
    
\end{remark}
\begin{remark}
    In \cref{thm: convergence result of periodic case}, 
    by choosing the correctors satisfying specific geometric property \eqref{geo property}, the results \eqref{value of sigma 2} show that the approximation \eqref{two scale approximate} has the same convergence order in $L^2$ and $H^1$ norm. Here, the result in $L_2$ norm of \eqref{value of sigma 2} is consistent with the result of \cite{LR:CMS:2022}, but only the uniform boundedness in $H^1$ norm has been obtained in \cite{LR:CMS:2022}. 
\end{remark}

\subsection{Proof of \cref{thm: convergence result of periodic case}}\label{prof thm1}
The following proof is inspired by the Lax equivalence theorem \cite{LR:CPAM:1956}.
For the 3D case, the LLG operator is defined as
\begin{equation}
	\left\{\begin{aligned}
		\mathcal{L}_{\mathrm{LLG}}(\mathbf{m}^\varepsilon,U^\varepsilon)
		:=&
		\partial_t\mathbf{m}^\varepsilon - \alpha \mathcal{H}^\varepsilon_e(\mathbf{m}^\varepsilon,U^\varepsilon)
		+ \mathbf{m}^\varepsilon \times \mathcal{H}^\varepsilon_e(\mathbf{m}^\varepsilon,U^\varepsilon)
		- \alpha  g_l^\varepsilon(\mathbf{m}^\varepsilon,U^\varepsilon)
		\mathbf{m}^\varepsilon,\\
		\mathcal{L}_{\mathrm{SF}}(\mathbf{m}^\varepsilon, U^\varepsilon)
		:=&
		\Delta U^\varepsilon + \operatorname{div}(M^\varepsilon\mathbf{m}^\varepsilon \mathcal{X}_{\Omega}),\quad \mbox{in $D'(\mathbb{R}^3)$},
	\end{aligned}\right.
\end{equation}
where the effective field $\mathcal{H}^\varepsilon_e$ takes the following form
\begin{equation}\label{mathcal H}
\mathcal{H}^\varepsilon_e(\mathbf{m}^\varepsilon,U^\varepsilon)
=
	\mathrm{div}\left(\ba^{\varepsilon} \nabla \mathbf{m}^{\varepsilon}\right)-K^{\varepsilon}\left(\mathbf{m}^{\varepsilon} \cdot \mathbf{u}\right) \mathbf{u}+M^\varepsilon\mathbf{h}_{\mathrm{a}}
	+\mu^\varepsilon \nabla U^\varepsilon,
\end{equation}
and the $g_l^\varepsilon[\cdot]$ is the energy density calculated by
\begin{equation}\label{energy density of epsilon}
	g_l^\varepsilon(\mathbf{m}^\varepsilon,U^\varepsilon) =  \ba^{\varepsilon} \vert \nabla \mathbf{m}^\varepsilon\vert^2 
	+ 
	K^{\varepsilon} \left( \mathbf{m}^\varepsilon\cdot \boldsymbol{u} \right)^2
	-
	\mu^\epsilon \nabla U^\varepsilon\cdot M^\varepsilon\mathbf{m}^\varepsilon
	-
	\mathbf{h}_a\cdot M^\varepsilon\mathbf{m}^\varepsilon .
\end{equation}
The classical solution $\mathbf{m}^\varepsilon$ of \eqref{rewrite LLG} with the stray field $U^\varepsilon$ satisfies
\begin{equation}\label{eqn:LLG system form 3}
	\mathcal{L}_{\mathrm{LLG}}(\mathbf{m}^\varepsilon,U^\varepsilon) = 0,
	\qquad 
	\mathcal{L}_{\mathrm{SF}}(\mathbf{m}^\varepsilon, U^\varepsilon) = 0.
\end{equation}
With the two-scale approximation $\widetilde{\mathbf{m}}^\varepsilon$ and $\widetilde{U}^\varepsilon$ defined in \eqref{def two scale approximate sln}, the approximate stray field $\Gamma^\varepsilon$ induced by $\widetilde{\mathbf{m}}^\varepsilon$ can be defined as
\begin{equation}\label{def gamma}
	\mathcal{L}_{\mathrm{SF}}(\widetilde{\mathbf{m}}^\varepsilon, \Gamma^\varepsilon) = 0.
\end{equation}
Then, the consistency error $\boldsymbol{\Theta}^\varepsilon_\mathrm{total}$ is decomposed into two parts
\begin{equation}\label{eqn:equivalent system of m epsilon}
	\begin{aligned}
		\boldsymbol{\Theta}^\varepsilon_\mathrm{total}
		&:= 
		\mathcal{L}_{\mathrm{LLG}}\big(\widetilde{\mathbf{m}}^\varepsilon, \Gamma^\varepsilon\big)\\
		&=
		\mathcal{L}_{\mathrm{LLG}}\big(\widetilde{\mathbf{m}}^\varepsilon, \widetilde{U}^\varepsilon\big)
		+
		\big\{ \mathcal{L}_{\mathrm{LLG}}\big(\widetilde{\mathbf{m}}^\varepsilon, \Gamma^\varepsilon\big)
		-
		\mathcal{L}_{\mathrm{LLG}}\big(\widetilde{\mathbf{m}}^\varepsilon, \widetilde{U}^\varepsilon\big) \big\}\\
		&=:
		\boldsymbol{\Theta}^\varepsilon_\mathrm{ts}
		+
		\boldsymbol{\Theta}^\varepsilon_\mathrm{sf}.
	\end{aligned}
\end{equation}
where $\boldsymbol{\Theta}^\varepsilon_\mathrm{ts}$ is induced by the two-scale approximation, and $\boldsymbol{\Theta}^\varepsilon_\mathrm{sf}$ is induced by the stray field. 

On the other hand, the error $\e^\epsilon(\bx)$ and $\Psi^\epsilon(\bx)$ between the classical solution and the approximating solution is defined as
\begin{equation}\label{ansatz}
	\m^\epsilon(\bx) = 	\widetilde{\m}^\epsilon(\bx) + \e^\epsilon(\bx),
	\qquad 
	U^\epsilon(\bx) = \Gamma^\epsilon(\bx) + \Psi^\epsilon(\bx).
\end{equation}
By subtracting \eqref{eqn:equivalent system of m epsilon} from \eqref{eqn:LLG system form 3}, the equation about $\e^\epsilon$ and $\Psi^\epsilon$ is given by
\begin{equation}\label{system of e}
	\left\{\begin{aligned}
		&\partial_t\mathbf{e}^\varepsilon - \alpha \widetilde{\mathcal{H}}^\varepsilon_e(\mathbf{e}^\varepsilon, \Psi^\varepsilon)
		-
		\mathbf{D}_1(\mathbf{e}^\varepsilon, \Psi^\varepsilon) - \mathbf{D}_2(\mathbf{e}^\varepsilon, \Psi^\varepsilon)
		= - \big(\boldsymbol{\Theta}^\varepsilon_\mathrm{ts}
		+
		\boldsymbol{\Theta}^\varepsilon_\mathrm{sf}\big),\\
		&\mathcal{L}_{\mathrm{SF}}(\mathbf{e}^\varepsilon, \Psi^\varepsilon) = 0.
	\end{aligned}\right.
\end{equation}
Here, $\widetilde{\mathcal{H}}^\varepsilon_e$ is the linear part of $\mathcal{H}^\varepsilon_e $, i.e.,
\begin{equation*}
	\begin{aligned}
		\widetilde{\mathcal{H}}^\varepsilon_e(\mathbf{e}^\varepsilon, \Psi^\varepsilon)
		& :=  \mathcal{H}^\varepsilon_e(\mathbf{e}^\varepsilon, \Psi^\varepsilon) -
		M^{\varepsilon} \mathbf{h}_a .
	\end{aligned}
\end{equation*}
The precession term $\mathbf{D}_1$ is calculated by
\begin{equation}\label{define D_1}
	\begin{aligned}
		\mathbf{D}_1 (\mathbf{e}^\varepsilon, \Psi^\varepsilon)&= 
		\mathbf{m}^\varepsilon \times \mathcal{H}^\varepsilon_e(\mathbf{m}^\varepsilon,U^\varepsilon)
		-
		\widetilde{\mathbf{m}}^\varepsilon \times \mathcal{H}^\varepsilon_e(\widetilde{\mathbf{m}}^\varepsilon, \Gamma^\varepsilon) \\
		&= 
		\mathbf{m}^\varepsilon \times \widetilde{\mathcal{H}}^\varepsilon_e(\mathbf{e}^\varepsilon, \Psi^\varepsilon)
		+
		\mathbf{e}^\varepsilon \times \mathcal{H}^\varepsilon_e(\widetilde{\mathbf{m}}^\varepsilon, \Gamma^\varepsilon),
	\end{aligned}
\end{equation}
and the degeneracy term $\mathbf{D}_2$ reads as
\begin{equation*}
	\begin{aligned}
		\mathbf{D}_2 (\mathbf{e}^\varepsilon, \Psi^\varepsilon)=& 
		- \alpha  g_l^\varepsilon(\mathbf{m}^\varepsilon,U^\varepsilon)
		\mathbf{m}^\varepsilon
		+
		\alpha  g_l^\varepsilon(\widetilde{\mathbf{m}}^\varepsilon, \Gamma^\varepsilon)
		\widetilde{\mathbf{m}}^\varepsilon \\
		=& -\alpha ( \nabla \mathbf{e}^\varepsilon \cdot \ba^{\varepsilon} \nabla \mathbf{m}^\varepsilon
		+
		\nabla \mathbf{e}^\varepsilon \cdot \ba^{\varepsilon} \nabla \widetilde{\mathbf{m}}^\varepsilon
		) \mathbf{m}^\varepsilon\\
		&-\alpha \big(
		K^{\varepsilon} ( \mathbf{m}^\varepsilon\cdot \boldsymbol{u} )( \mathbf{e}^\varepsilon\cdot \boldsymbol{u} )
		+ 
		K^{\varepsilon} ( \widetilde{\mathbf{m}}^\varepsilon\cdot \boldsymbol{u} )( \mathbf{e}^\varepsilon\cdot \boldsymbol{u} )
		\big) \mathbf{m}^\varepsilon\\
		& -\alpha \big( 
		\mu^\varepsilon \nabla U^\varepsilon\cdot \mathbf{e}^\varepsilon
		+ 
		\mu^\varepsilon \nabla \Gamma^\varepsilon\cdot \widetilde{\mathbf{m}}^\varepsilon 
		+ M^\varepsilon (\mathbf{h}_a\cdot \mathbf{e}^\varepsilon)\big) \mathbf{m}^\varepsilon
		- \alpha g_l^\varepsilon(\widetilde{\mathbf{m}}^\varepsilon, \Gamma^\varepsilon)  \mathbf{e}^\varepsilon.
	\end{aligned}
\end{equation*}
By employing the idea of Lax equivalence theorem, the estimate of errors $\e^\epsilon$ and $\Psi^\epsilon$ can be derived by two steps: the consistency analysis of \eqref{eqn:equivalent system of m epsilon}, and the stability analysis of \eqref{system of e}.

Firstly, the result of consistency analysis is given in the following lemma.
\begin{lemma}\label{thm: Consistency estimate of two-scale approximation}
    Suppose Assumptions \eqref{assumption: coefficient}-\eqref{assumption: initial data} hold. For the consistency error $\boldsymbol{\Theta}^\varepsilon_\mathrm{total} = \boldsymbol{\Theta}^\varepsilon_\mathrm{ts}
	+
	\boldsymbol{\Theta}^\varepsilon_\mathrm{sf}$ given in \eqref{eqn:equivalent system of m epsilon}, it holds
	\begin{equation*}
		\begin{aligned}
			&\Vert \boldsymbol{\Theta}^\varepsilon_\mathrm{ts}(\boldsymbol{x}) \Vert_{L^2(\Omega)} 
			\le C \varepsilon,\\
			&\Vert \boldsymbol{\Theta}^\varepsilon_\mathrm{sf} \Vert_{L^{r}(\Omega)}
			\le C_r
			\varepsilon^{1/r}\ln (\varepsilon^{-1} + 1) ,
		\end{aligned}
	\end{equation*}
	for any $1< r < \infty$.
	Here, the constants $C$ and $C_r$ depend on $C_{\mathrm{coe}}$ and $\Vert \mathbf{m}_0 \Vert_{W^{4,\infty}(\Omega)}$, but are independent of $\varepsilon$.
\end{lemma}

Secondly, the results of stability analysis are presented in \cref{theorem:stability} and \cref{regularety 3}.
\begin{lemma}\label{theorem:stability}
	Let $\mathbf{e}^\varepsilon \in L^\infty(0,T;H^2(\Omega))$ be a strong solution to \eqref{system of e}. Suppose Assumptions \eqref{assumption: coefficient}-\eqref{assumption: initial data} hold. For the 3D case, it holds
	\begin{equation}\label{ineq of e^eps_b n=3}
		\begin{aligned}
			\Vert \mathbf{e}^\varepsilon \Vert^2_{L^\infty(0,T;L^2(\Omega))} 
			\le  
			C \big( 
			\Vert \mathbf{e}^\varepsilon(\boldsymbol{x},0) \Vert_{L^{2}(\Omega)}^2
			+
			\Vert\boldsymbol{\Theta}^\varepsilon_\mathrm{ts} \Vert_{L^2(0,T; L^{2} (\Omega))}^2
			+
			\Vert\boldsymbol{\Theta}^\varepsilon_\mathrm{sf} \Vert_{L^2(0,T; L^{6/5}(\Omega))}^2 \big) ,
		\end{aligned}
	\end{equation}
	and	\begin{equation}\label{stability result in H^1}
		\begin{aligned}
			\Vert \nabla \mathbf{e}^\varepsilon \Vert^2_{L^\infty(0,T;L^2(\Omega))} 
			\le 
			C \big( \Vert \mathbf{e}^\varepsilon(\boldsymbol{x},0) \Vert_{H^{1}(\Omega)}^2
			+
			\Vert\boldsymbol{\Theta}^\varepsilon_\mathrm{ts} \Vert_{L^2(0,T; L^{2} (\Omega))}^2
			+
			\Vert\boldsymbol{\Theta}^\varepsilon_\mathrm{sf} \Vert_{L^2(0,T; L^{2}(\Omega))}^2 \big),
		\end{aligned}
	\end{equation}
	where the constant $C$ depends on $C_{\mathrm{coe}}$, $\Vert \mathbf{m}_0 \Vert_{W^{4,\infty}(\Omega)}$ and $\Vert \nabla \m^\epsilon \Vert_{L^{6}(\Omega)}$.
\end{lemma}

Together with Lemma \ref{thm: Consistency estimate of two-scale approximation} and Lemma \ref{theorem:stability}, the estimate of $\Vert \e^\epsilon \Vert^2_{L^\infty(0,T;L^2(\Omega))} $ and $\Vert \nabla \e^\epsilon \Vert^2_{L^{\infty}(0,T; L^{2} (\Omega))}$ can be obtained with Gr\"{o}nwall's inequality:
\begin{equation}\label{eqn:conclution in thm 1}
	\begin{aligned}
		&\Vert \e^\epsilon(\cdot,t) \Vert_{L^\infty(0,T;L^2(\Omega))}
		\le C(\Vert \nabla \m^\epsilon \Vert_{L^{6}(\Omega)}) \,\epsilon^{5/6} \ln (\epsilon^{-1} + 1), \\\
		&\Vert \e^\epsilon(\cdot,t) \Vert_{L^\infty(0,T;H^1(\Omega))}
		\le C(\Vert \nabla \m^\epsilon \Vert_{L^{6}(\Omega)}) \,\epsilon^{1/2}\ln (\epsilon^{-1} + 1).
	\end{aligned}
\end{equation}
However, the above constants depend on the bound of $\Vert \nabla \m^\epsilon \Vert_{L^{6}(\Omega)}$. In order to derive an $\epsilon$-independent estimate, the following lemma is introduced. 
\begin{lemma}\label{regularety 3}
	Under Assumptions \eqref{assumption: coefficient}-\eqref{assumption: initial data},
	there exists $T^*\in(0, T]$ independent of $\epsilon$, such that for $0\le t\le T^*$,
	\begin{equation*}
		\Vert \nabla \m^\epsilon  (\cdot, t)\Vert_{L^{6}(\Omega)}^2
		\le
		C,
	\end{equation*}
	where the constant $C$ depends on $C_{\mathrm{coe}}$, $a_{\mathrm{min}}$ and $a_{\mathrm{max}}$, but is independent of $\epsilon$ and $t$.
\end{lemma}
\begin{remark}
Comparing with the similar results in \cite{CLS:apa:2022}, the more complex lower-order terms are cosindered in our multiscale LLG model. 
The proof of \cref{thm: Consistency estimate of two-scale approximation}, \cref{theorem:stability}, and \cref{regularety 3} can be found in \cref{first lemma}, \cref{second lemma}, and \cref{third lemma}, respectively. 
\end{remark}

Utilizing \eqref{eqn:conclution in thm 1} and \cref{regularety 3}, one can derive that for some $0<T^*\le T$, it holds
\begin{equation}\label{eqn:conclution in thm 1 form 2}
	\begin{aligned}
		&\Vert \e^\epsilon(\cdot,t) \Vert_{L^\infty(0,T^*;L^2(\Omega))}
		\le C \epsilon^{5/6} \ln (\epsilon^{-1} + 1), \\\
		&\Vert \e^\epsilon(\cdot,t) \Vert_{L^\infty(0,T^*;H^1(\Omega))}
		\le C \epsilon^{1/2}\ln (\epsilon^{-1} + 1).
	\end{aligned}
\end{equation}
Here, the constant $C$ depends on $C_{\mathrm{coe}}$, $a_{\mathrm{min}}$ and $a_{\mathrm{max}}$, but is independent of $\epsilon$.
Moreover, revisit the definition of $\e^\epsilon (\bx)$ in \eqref{ansatz}, we have
\begin{equation}\label{express of error}
	\e^\epsilon (\bx) =
	\m^\epsilon(\bx) - \m_0(\bx) - \epsilon\m_1\left(\bx,\frac{\bx}{\epsilon}\right) - \epsilon^2\m_2\left(\bx,\frac{\bx}{\epsilon}\right).
\end{equation}
Based on the definition of $\m_2$ in \eqref{second-order corrector} and Assumption \eqref{assumption: coefficient}, there exists some constant $C$ independent of $\epsilon$, such that
\begin{equation}\label{m2 bounded}
	\epsilon^2 \Big\Vert \mathbf{m}_2\left(\bx,\frac{\bx}{\epsilon}\right) \Big\Vert_{H^1(\Omega)}
	\le 
	\epsilon \Vert \mathbf{m}_2(\bx,\by) \Vert_{W^{1,\infty}(\Omega\times Y)}
	\le C \epsilon.
\end{equation}
For the 3D case, combining \eqref{eqn:conclution in thm 1 form 2}, \eqref{express of error} with \eqref{m2 bounded}, the estimates \eqref{convergence estimate of periodic} in \cref{thm: convergence result of periodic case}  can be derived. 

For the 2D case, the stray field degenerates to a simple form \eqref{stray field 2D}, and the error $\boldsymbol{\Theta}^\varepsilon_\mathrm{sf}$ disappears. Therefore, by setting $\boldsymbol{\Theta}^\varepsilon_\mathrm{sf}=0$, one can prove that Lemma \ref{thm: Consistency estimate of two-scale approximation}, Lemma \ref{theorem:stability} and Lemma \ref{regularety 3} still hold. Finally, the 2D part of Theorem \ref{thm: convergence result of periodic case} can be similarly deduced. 
The proof is completed.

\section{Numerical experiments}\label{sec4}

This section provides several numerical examples to verify the convergence results presented in \cref{sec3}. Moreover, it demonstrates the proposed framework's robustness in simulating the multiscale LLG equation under different boundary conditions and effective fields.
To simplify the representation of approximate error 
in \cref{thm: convergence result of periodic case} and \cref{thm:neumann}, the following notations are introduced.
\begin{equation}\label{error}
	\begin{gathered}
		e_0=\left\|\boldsymbol{m}^{\varepsilon}(\bx,t)-\boldsymbol{m}_0(\bx,t)\right\|_{L^2(\Omega)},\quad \tilde{e}_0=e_0/\left\|\boldsymbol{m}^{\varepsilon}(\bx,t)\right\|_{L^2(\Omega)},\\
		e_1=\Big\|\boldsymbol{m}^{\varepsilon}(\bx,t)-\boldsymbol{m}_0(\bx,t)-\varepsilon\boldsymbol{\chi}\left(\frac{\bx}{\epsilon}\right)\nabla\mathbf{m}_0(\bx,t)\Big\|_{H^1(\Omega)},\quad \tilde{e}_1=e_1/\left\|\boldsymbol{m}^{\varepsilon}(\bx,t)\right\|_{H^1(\Omega)},\\
		e_2=\left\|\boldsymbol{m}^{\varepsilon}(\bx, t)-\boldsymbol{m}_0(\bx, t)-(\boldsymbol{\Phi}^\epsilon-\boldsymbol{x})\nabla\mathbf{m}_0(\bx, t)\right\|_{H^1(\Omega)},\quad \tilde{e}_2=e_2/\left\|\boldsymbol{m}^{\varepsilon}(\bx,t)\right\|_{H^1(\Omega)}.
	\end{gathered}
\end{equation}
\par For the 2D examples, the problems \eqref{multi model} are considered in the domain $\Omega=[0,1]^2$. The finite element solutions with spatial mesh size $h=1/180$ are employed as the reference.
The reference solutions and the homogenized solutions share the consistent time step size $\Delta t = 10^{-6}$ and damping constant $\alpha=1$. The initial data $\m_{\mathrm{init}}^{0}(\bx)$ of the homogenized system is defined as
\begin{equation*}
	\mathbf{m}_{\mathrm{init}}^{0}(\boldsymbol{x})= 
	\begin{cases}
		(0,0,-1) & \text { for }|\tilde{\bx}| \geq 1 / 2 \\
		\left(2 \tilde{x}_1 A, 2 \tilde{x}_2 A, A^2-|\tilde{\bx}|^2\right) /\left(A^2+|\tilde{\bx}|^2\right) & \text { for }|\tilde{\bx}| \leq 1 / 2
	\end{cases},
\end{equation*}
where $\tilde{\bx}=(\tilde{x}_1,\tilde{x}_2)=(x_1-1/2,x_2-1/2),$ and $A:=(1-2|\tilde{\bx}|)^4$.

\subsection{2D Periodic problem}\label{2D-P}
This example presents the simulation results for the 2D Periodic problem involving the exchange field. The exchange field is the second-order term of the effective field and plays a dominant role in the evolution of the multiscale LLG system. In this case, the multiscale LLG equation is given by
\begin{equation*}
	\left\{\begin{aligned}
		&\partial_t \mathbf{m}^{\varepsilon}-\alpha \mathbf{m}^{\varepsilon} \times \partial_t \mathbf{m}^{\varepsilon}
		=-\left(1+\alpha^2\right) \m^\epsilon \times \bdiv\left(\ba^\varepsilon \nabla \mathbf{m}^\varepsilon\right)
		\quad \text {in } \Omega, \\
		&\mathbf{m}^{\varepsilon}(\boldsymbol{x}, 0) =\mathbf{m}_{\mathrm{init}}^{\varepsilon}(\boldsymbol{x}), \quad\left|\mathbf{m}_{\mathrm{init}}^{\varepsilon}(\boldsymbol{x})\right|=1 \quad \text {in } \Omega,\\
	&\mathbf{m}^{\varepsilon}(\boldsymbol{x}, t),\, \mathbf{m}_{\mathrm{init}}^\varepsilon(\boldsymbol{x})  \quad \text {are periodic on } \partial \Omega .
	\end{aligned}\right.
\end{equation*}
The corresponding homogenized LLG equation is
\begin{equation*}
	\left\{\begin{aligned}
		&\partial_t \mathbf{m}_0-\alpha \mathbf{m}_0 \times \partial_t \mathbf{m}_0=-\left(1+\alpha^2\right) \mathbf{m}_0 \times \bdiv\left(\ba^0 \nabla \mathbf{m}_0\right) \quad \text {in } \Omega, \\
		&\mathbf{m}_0(\boldsymbol{x}, 0) =\mathbf{m}_{\mathrm{init}}^0(\boldsymbol{x}), \quad\left|\mathbf{m}_{\mathrm{init}}^0(\boldsymbol{x})\right|=1 \quad \text {in } \Omega,\\
	&\mathbf{m}_0(\boldsymbol{x}, t),\, \mathbf{m}_{\mathrm{init}}^0(\boldsymbol{x})  \quad \text {are periodic on } \partial \Omega .
	\end{aligned}\right.
\end{equation*}
The initial data of the multiscale LLG equation is obtained by the expansion method
\begin{equation*}
	\m_{\mathrm{init}}^\epsilon(\bx)
	= \m_{\mathrm{init}}^0(\bx) + \epsilon\boldsymbol{\chi}\left(\frac{\bx}{\epsilon}\right)\nabla\mathbf{m}_{\mathrm{init}}^0(\bx).
\end{equation*}
For the 2D Periodic problem only with the exchange field, the convergence results are given as follows
\begin{equation*}
	\begin{gathered}
		\left\|\boldsymbol{m}^{\varepsilon}(\bx,t)-\boldsymbol{m}_0(\bx,t)\right\|_{L^2(\Omega)} \leq C \varepsilon^{1},\\
		\left\|\boldsymbol{m}^{\varepsilon}(\bx,t)-\boldsymbol{m}_0(\bx,t)-\varepsilon\boldsymbol{\chi}\left(\frac{\bx}{\epsilon}\right)\nabla\mathbf{m}_0(\bx,t)\right\|_{H^1(\Omega)} \leq C \varepsilon^{1}.
	\end{gathered}
\end{equation*}

Under assumption \eqref{periodic perturbation assumption}, the exchange coefficient $\ba^\varepsilon$ can be defined by $\ba(\by)$
\begin{equation*}
    \ba(\by)=\Big(1.1+0.25\cos{\big(2\pi(y_1-0.5)\big)}\Big)\Big(1.1+0.25\cos{\big(2\pi(y_2-0.5)\big)}\Big)\times I_2,
\end{equation*}
where $I_2$ is the $2\times 2$ identity matrix. 

\begin{figure}[h!]
    \centering
	\subfigure[]{\includegraphics[width=0.44\textwidth]{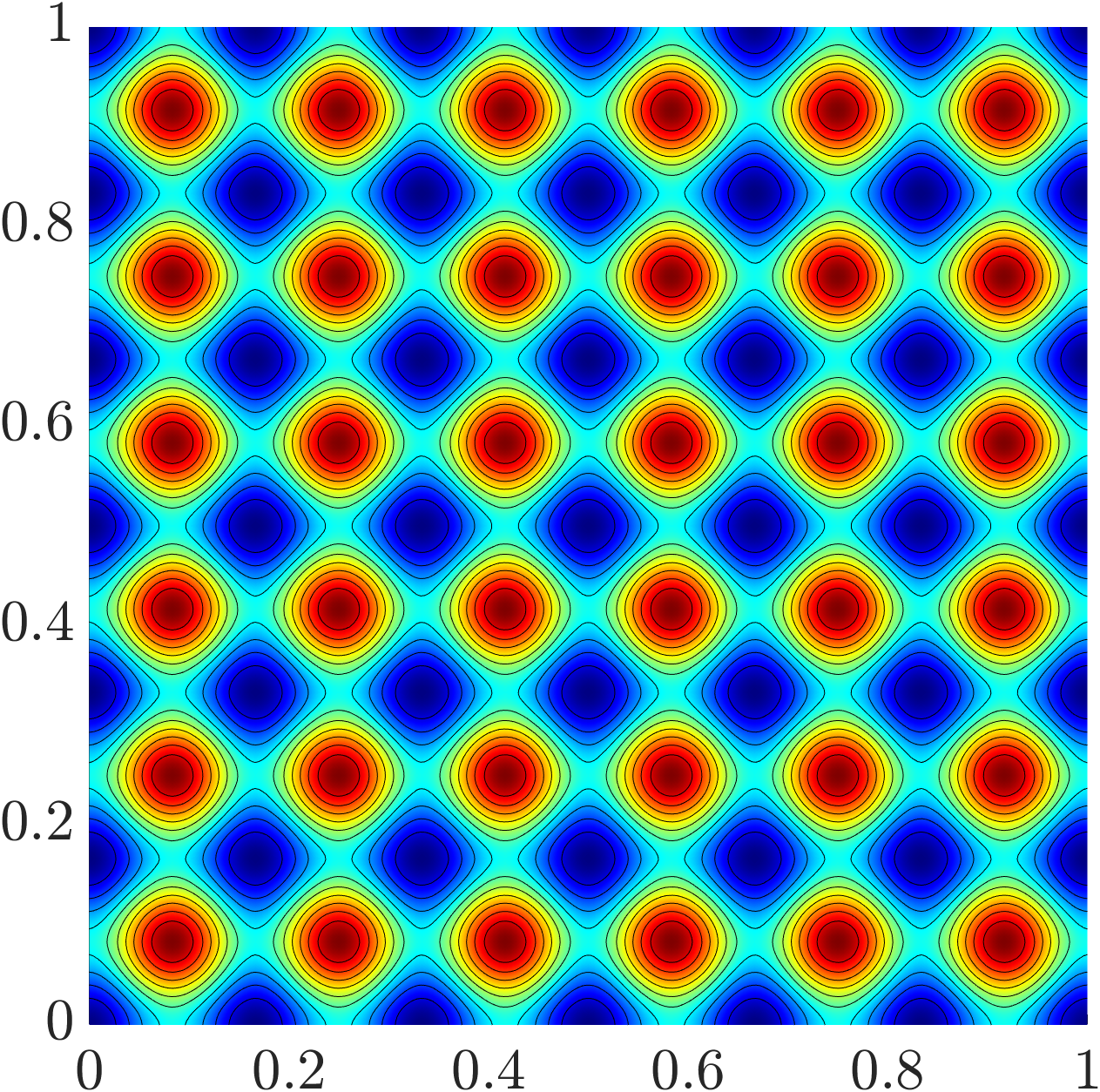}}
    \subfigure[]{\includegraphics[width=0.5\textwidth]{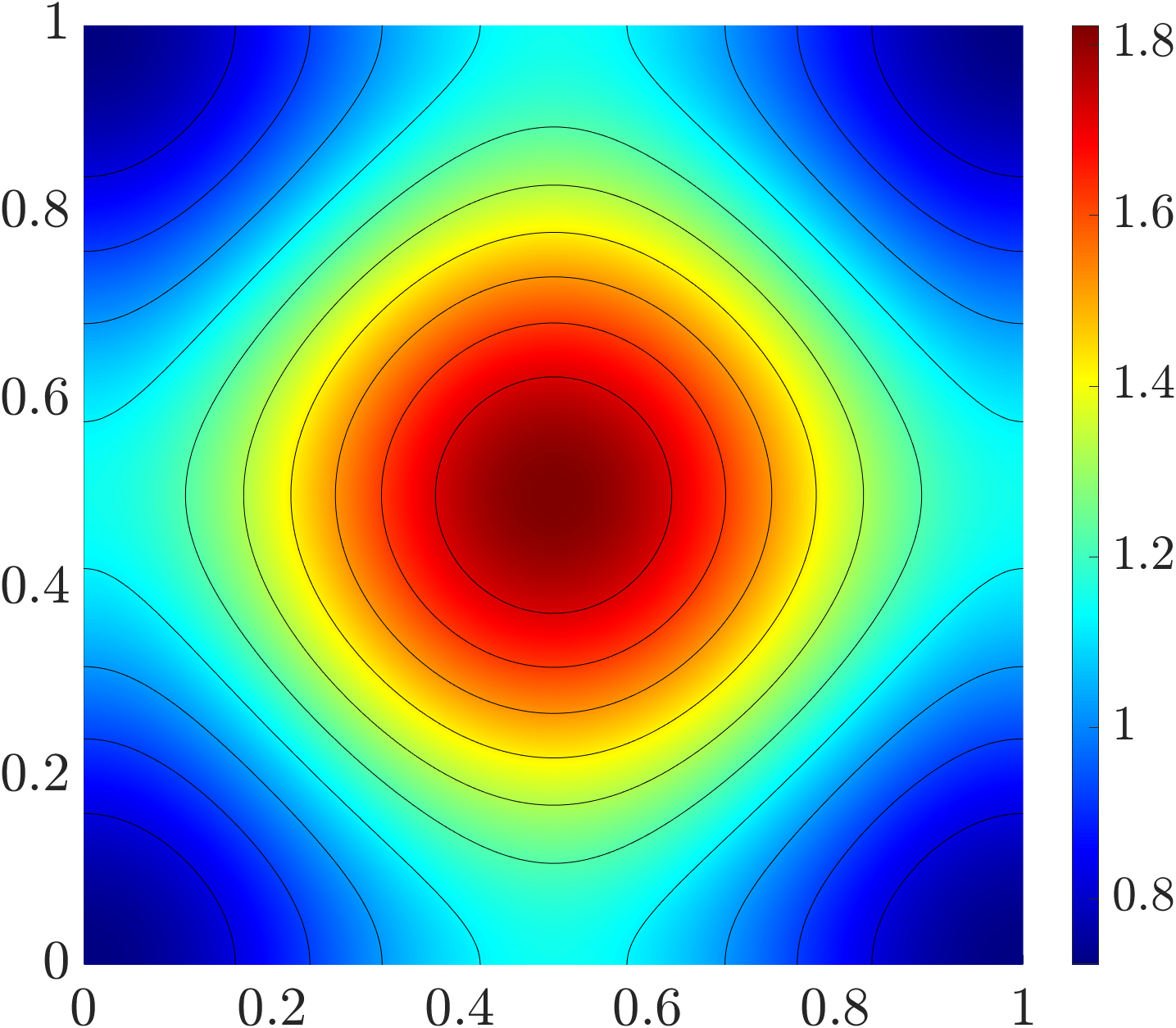}}    
	\caption{Multiscale exchange coefficient $\ba^\varepsilon(\bx)$ defined in (a) the whole computational domain $\Omega$, and (b) the reference unit cell $Y$.}
\end{figure}

\cref{fig:cell function} shows two auxiliary functions $\chi_1$ and $\chi_2$ of the first-order cell problem \eqref{first-order cell problem}, which are utilized to obtain the corresponding homogenized coefficients $\ba^0$ by \eqref{homogenized coefficient}.

\begin{equation*}
    \ba^0=\left(\begin{aligned}
        &1.178 & 1.073\times 10^{-9}\\
        &1.073\times 10^{-9} & 1.178\\
    \end{aligned}
    \right).
\end{equation*}


\begin{figure}[t!]
    \centering
    \subfigure[$\chi_1$]{\includegraphics[width=0.445\textwidth]{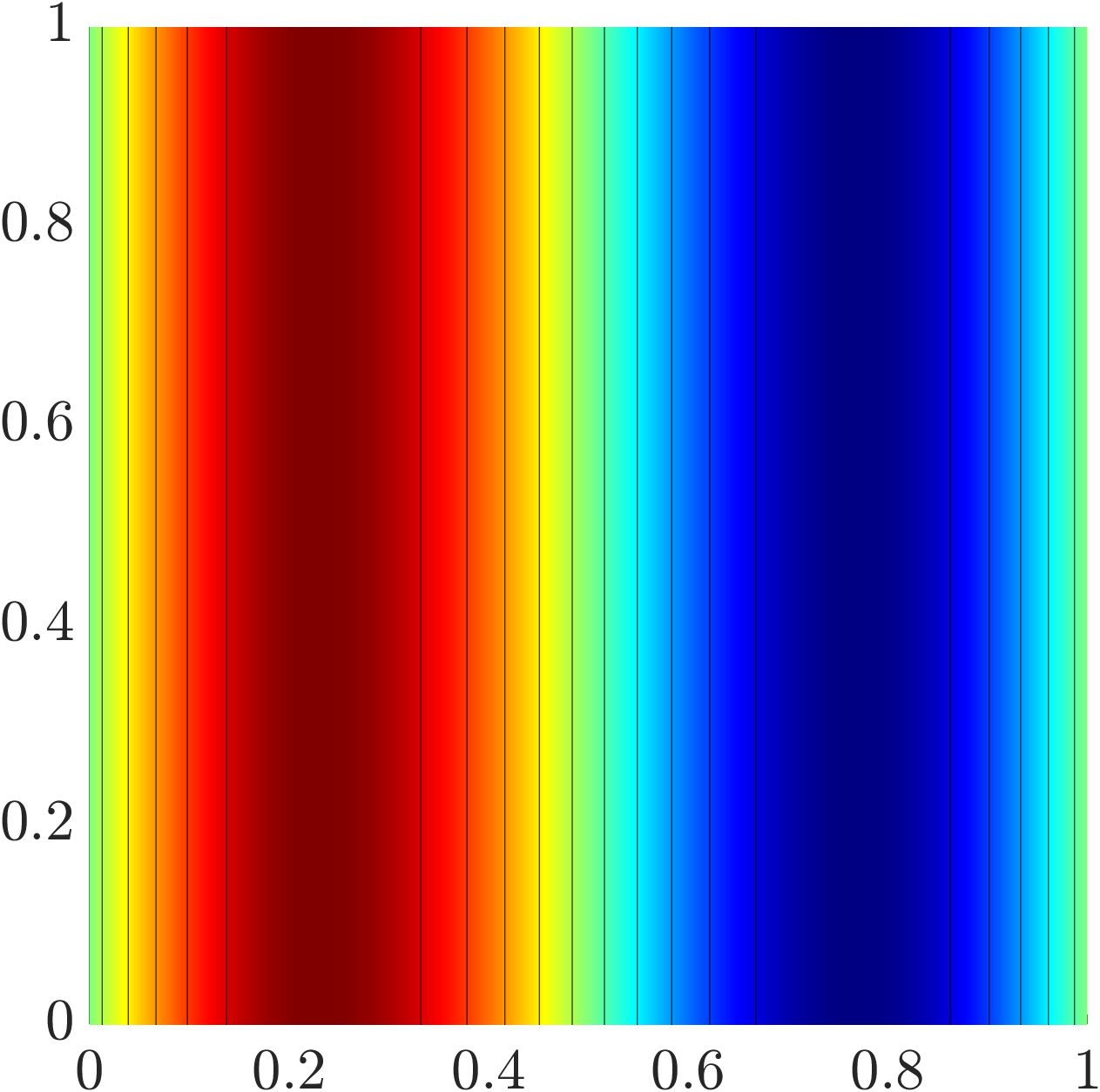}}
    \subfigure[$\chi_2$]{\includegraphics[width=0.522\textwidth]{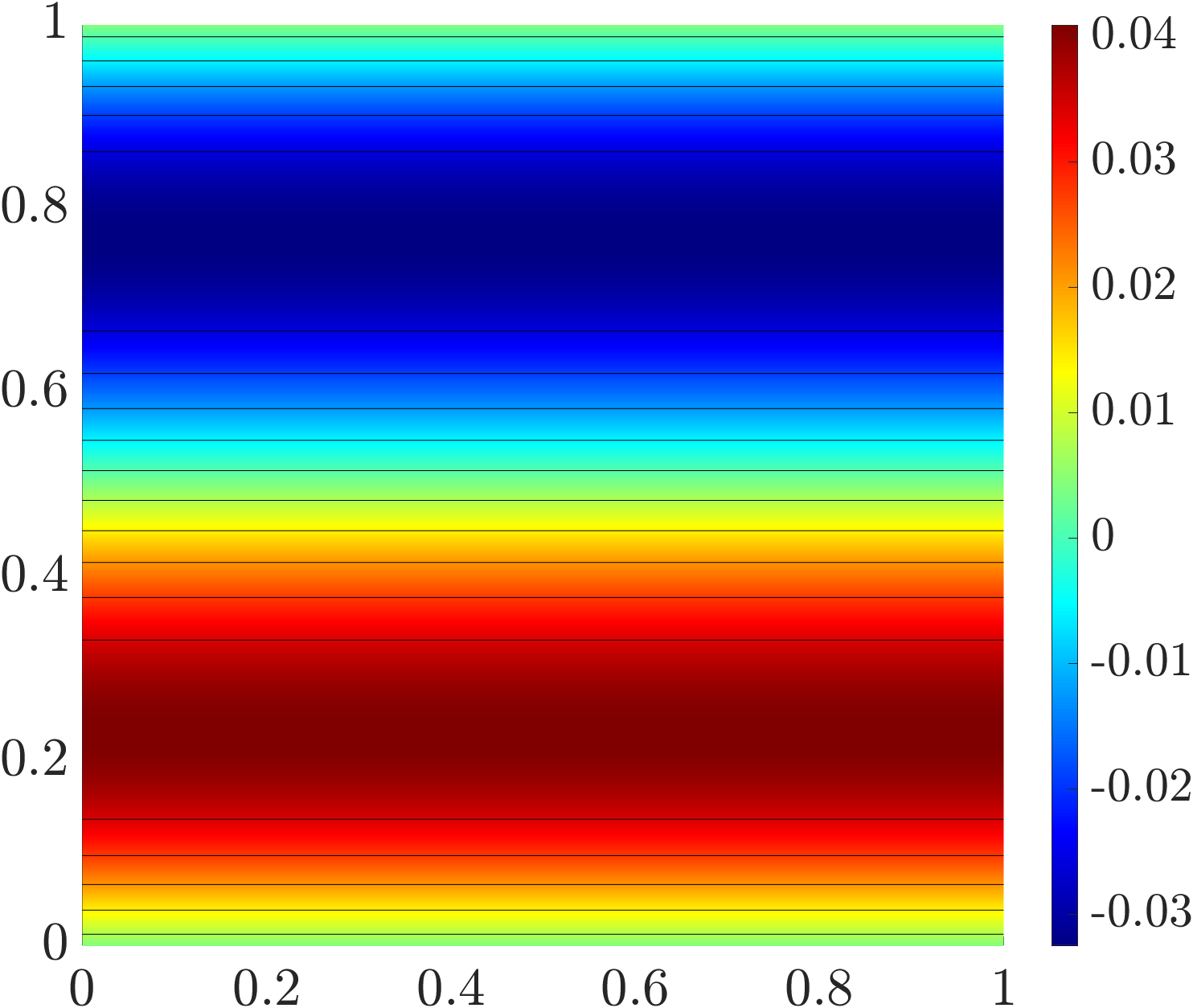}}
    \caption{Demonstration of the auxiliary functions $\chi_1,\chi_2$.}
	\label{fig:cell function}
    \vspace{-2mm}
\end{figure}

\cref{tab: periodic} presents the error results at different time $t=j\Delta t$ with various numbers of periods $n$, where the corresponding error convergence orders are depicted in \cref{fig: periodic}. 
From the figure, the error convergence orders obtained from the proposed framework is consistent with the theoretical results. As time increases, the computational errors inevitably accumulate, eventually leading to a slight decrease in the convergence order at $j=10^3$.

\begin{table}[h!]
    \centering
    \resizebox{\linewidth}{!}
    {\begin{tabular}{|c|c|c|c|c|c|c|}
        \hline
        \multirow{2}*{$n$} & \multicolumn{2}{c|}{j=10}  &\multicolumn{2}{c|}{j=$10^2$} &\multicolumn{2}{c|}{j=$10^3$}\\ \cline{2-7}  
                        & $e_0$  &  $\tilde{e}_0$  
                        &  $e_0$  &  $\tilde{e}_0$ &  $e_0$  &  $\tilde{e}_0$ \\   
        \hline
        $2$   & $1.18\times 10^{-1}$ & $1.18\times 10^{-1}$ & $1.17\times 10^{-1}$ & $1.17\times 10^{-1}$  & $1.09\times 10^{-1}$ & $1.09\times 10^{-1}$\\
        \hline
        $3$   & $6.34\times 10^{-2}$ & $6.34\times 10^{-2}$ & $6.26\times 10^{-2}$ & $6.26\times 10^{-2}$ & $5.81\times 10^{-2}$ & $5.81\times 10^{-2}$\\
        \hline
        $4$   & $4.46\times 10^{-2}$ & $4.46\times 10^{-2}$ & $4.45\times 10^{-2}$ & $4.45\times 10^{-2}$ & $4.52\times 10^{-2}$ & $4.52\times 10^{-2}$\\
        \hline
        $5$   & $3.77\times 10^{-2}$ & $3.78\times 10^{-2}$ & $3.76\times 10^{-2}$ & $3.76\times 10^{-2}$ & $3.91\times 10^{-2}$ & $3.91\times 10^{-2}$\\
        \hline
        $6$   & $3.25\times 10^{-2}$ & $3.25\times 10^{-2}$ & $3.21\times 10^{-2}$ & $3.21\times 10^{-2}$ & $3.27\times 10^{-2}$ & $3.27\times 10^{-2}$\\
        \hline
		\quad &  $e_1$  &  $\tilde{e}_1$ &  $e_1$  &  $\tilde{e}_1$ &  $e_1$  &  $\tilde{e}_1$ \\   
		\hline
        $2$   & $2.14$ & $3.33\times 10^{-1}$ & $2.06$ & $3.23\times 10^{-1}$  & $1.54$ & $2.59\times 10^{-1}$\\
        \hline
        $3$   & $1.13$ & $1.55\times 10^{-1}$ & $1.10$ & $1.52\times 10^{-1}$ & $9.38\times 10^{-1}$ & $1.42\times 10^{-1}$\\
        \hline
        $4$   & $7.94\times 10^{-1}$ & $1.18\times 10^{-1}$ & $7.76\times 10^{-1}$ & $1.17\times 10^{-1}$ & $7.05\times 10^{-1}$ & $1.13\times 10^{-1}$\\
        \hline
        $5$   & $6.72\times 10^{-1}$ & $9.94\times 10^{-2}$ & $6.57\times 10^{-1}$ & $9.80\times 10^{-2}$ & $6.34\times 10^{-1}$ & $1.01\times 10^{-1}$\\
        \hline
        $6$   & $5.81\times 10^{-1}$ & $8.35\times 10^{-2}$ & $5.70\times 10^{-1}$ & $8.28\times 10^{-2}$ & $5.64\times 10^{-1}$ & $8.87\times 10^{-2}$\\
        \hline
	\end{tabular}}
    \caption{Error of the 2D Periodic problem under various numbers of periods $n=1/\varepsilon$ at time $t=j\Delta t$.}
    \label{tab: periodic}
\end{table}

\begin{figure}[h!]
    \centering
    \includegraphics[width=1.0\textwidth]{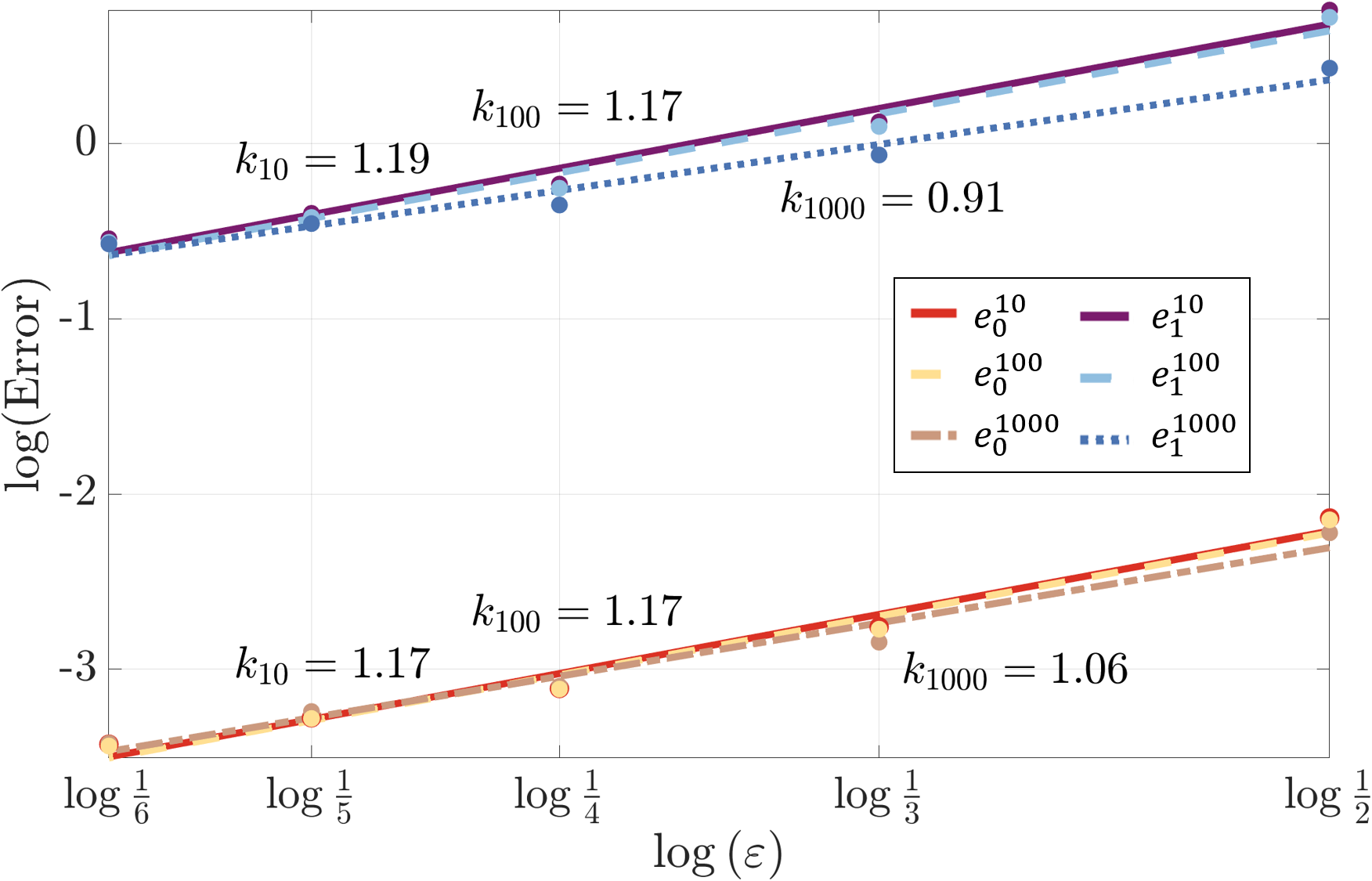}
    \caption{Variation of error $e_0^j$ and $e_1^j$ relative to the cell size $\varepsilon$ across different time steps $j$ in the 2D Periodic problem.} 
    \vspace{-6mm}
    \label{fig: periodic}
\end{figure}

\subsection{2D Neumann problem}\label{2D-N}

For the Neumann problem, the effective field $\bh_{\mathrm{eff}}^\varepsilon$ and the multiscale exchange coefficient $\ba^\varepsilon$ are considered with the same situation of \cref{2D-P}. In this case, the multiscale LLG equation is expressed as
\begin{equation*}
	\left\{\begin{aligned}
		&\partial_t \mathbf{m}^{\varepsilon}-\alpha \mathbf{m}^{\varepsilon} \times \partial_t \mathbf{m}^{\varepsilon} =-\left(1+\alpha^2\right) \mathbf{m}^{\varepsilon} \times \bdiv\left(\ba^\varepsilon \nabla \mathbf{m}^\varepsilon\right) \quad \text {in } \Omega,\\
		&\mathbf{m}^{\varepsilon}(\boldsymbol{x}, 0) =\mathbf{m}_{\mathrm{init}}^{\varepsilon}(\boldsymbol{x}), \quad\left|\mathbf{m}_{\mathrm{init}}^{\varepsilon}(\boldsymbol{x})\right|=1 \quad \text {in } \Omega,\\
		&\boldsymbol{\nu} \cdot \ba^{\varepsilon} \nabla \mathbf{m}^{\varepsilon}(\boldsymbol{x}, t) =0, \quad \text {on } \partial \Omega \times[0, T], \\
		&\boldsymbol{\nu} \cdot \ba^{\varepsilon} \nabla \mathbf{m}_{\mathrm{init}}^{\varepsilon}(\boldsymbol{x}) =0, \quad \text {on } \partial \Omega. \\
	\end{aligned}\right.
\end{equation*}
The corresponding homogenized LLG equation is
\vspace{-2mm}
\begin{equation*}
    \left\{\begin{aligned}
        &\partial_t \mathbf{m}_0-\alpha \mathbf{m}_0 \times \partial_t \mathbf{m}_0=-\left(1+\alpha^2\right) \mathbf{m}_0 \times \bdiv\left(\ba^0 \nabla \mathbf{m}_0\right) \quad \text { in } \Omega, \\
        &\mathbf{m}_{0}(\boldsymbol{x}, 0)=\mathbf{m}_{\mathrm{init}}^{0}(\boldsymbol{x}), \quad\left|\mathbf{m}_{\mathrm{init}}^{0}(x)\right|=1 \quad \text { in } \Omega,\\
        &\boldsymbol{\nu} \cdot \ba^{0} \nabla \mathbf{m}_0(\boldsymbol{x}, t) = 0, \quad \text { on } \partial \Omega \times[0, T], \\
        &\boldsymbol{\nu} \cdot \ba^{0} \nabla \mathbf{m}_{\mathrm{init}}^{0}(\boldsymbol{x})=0, \quad \text { on } \partial \Omega. \\
    \end{aligned}\right.
\end{equation*}
Here, we employ the expansion method to obtain the initial data of the multiscale LLG equation
\begin{equation*}
	\m_{\mathrm{init}}^\epsilon(\bx)
	= \m_{\mathrm{init}}^0(\bx) + (\boldsymbol{\Phi}^\epsilon-\boldsymbol{x})\nabla\mathbf{m}_{\mathrm{init}}^0(\bx),
\end{equation*}
where $\boldsymbol{\Phi}^\epsilon$ is the Neumann corrector defined in \eqref{neumann correct}. For the 2D Neumann problem only with the exchange field, the convergence results are given as follows
\begin{equation*}
	\begin{gathered}
		\left\|\boldsymbol{m}^{\varepsilon}(\bx,t)-\boldsymbol{m}_0(\bx,t)\right\|_{L^2(\Omega)} \leq C \varepsilon^{1},\\
         \left\|\boldsymbol{m}^{\varepsilon}(\bx, t)-\boldsymbol{m}_0(\bx, t)-(\boldsymbol{\Phi}^\epsilon-\boldsymbol{x})\nabla\mathbf{m}_0(\bx, t)\right\|_{H^1(\Omega)} \leq C \varepsilon^{1}.
	\end{gathered}
\end{equation*}


The error results of the Neumann problem are shown in \cref{tab: Neumann}. 
As depicted in \cref{fig: Neumann}, the obtained convergence orders align well with theoretical results. The similar trends of convergence orders are observed for the Periodic problem and the Neumann problem. 

\begin{table}[h!]
    \centering
    \resizebox{\linewidth}{!}
    {\begin{tabular}{|c|c|c|c|c|c|c|}
        \hline
        \multirow{2}*{$n$} & \multicolumn{2}{c|}{j=10}  &\multicolumn{2}{c|}{j=$10^2$} &\multicolumn{2}{c|}{j=$10^3$}\\ \cline{2-7}  
                        &  $e_0$  &  $\tilde{e}_0$  
                        &  $e_0$  &  $\tilde{e}_0$ &  $e_0$  &  $\tilde{e}_0$ \\   
        \hline
        $2$   & $1.13\times 10^{-1}$ & $1.13\times 10^{-1}$ & $1.13\times 10^{-1}$ & $1.13\times 10^{-1}$  & $1.05\times 10^{-1}$ & $1.05\times 10^{-1}$\\
        \hline
        $3$   & $5.80\times 10^{-2}$ & $5.80\times 10^{-2}$ & $5.72\times 10^{-2}$ & $5.72\times 10^{-2}$ & $5.30\times 10^{-2}$ & $5.30\times 10^{-2}$\\
        \hline
        $4$   & $4.01\times 10^{-2}$ & $4.01\times 10^{-2}$ & $4.00\times 10^{-2}$ & $4.00\times 10^{-2}$ & $4.14\times 10^{-2}$ & $4.14\times 10^{-2}$\\
        \hline
        $5$   & $3.42\times 10^{-2}$ & $3.42\times 10^{-2}$ & $3.41\times 10^{-2}$ & $3.41\times 10^{-2}$ & $3.63\times 10^{-2}$ & $3.63\times 10^{-2}$\\
        \hline
        $6$   & $2.96\times 10^{-2}$ & $2.96\times 10^{-2}$ & $2.93\times 10^{-2}$ & $2.93\times 10^{-2}$ & $3.03\times 10^{-2}$ & $3.03\times 10^{-2}$\\
        \hline
		\quad &  $e_2$  &  $\tilde{e}_2$ &  $e_2$  &  $\tilde{e}_2$ &  $e_2$  &  $\tilde{e}_2$ \\   
		\hline
        $2$   & $2.02$ & $3.14\times 10^{-1}$ & $1.93$ & $3.04\times 10^{-1}$  & $1.44$ & $2.42\times 10^{-1}$\\
        \hline
        $3$   & $1.03$ & $1.41\times 10^{-1}$ & $1.01$ & $1.39\times 10^{-1}$ & $8.71\times 10^{-1}$ & $1.32\times 10^{-1}$\\
        \hline
        $4$   & $7.18\times 10^{-1}$ & $1.07\times 10^{-1}$ & $7.02\times 10^{-1}$ & $1.06\times 10^{-1}$ & $6.57\times 10^{-1}$ & $1.06\times 10^{-1}$\\
        \hline
        $5$   & $6.07\times 10^{-1}$ & $8.97\times 10^{-2}$ & $5.94\times 10^{-1}$ & $8.87\times 10^{-2}$ & $5.97\times 10^{-1}$ & $9.56\times 10^{-2}$\\
        \hline
        $6$   & $5.27\times 10^{-1}$ & $7.57\times 10^{-2}$ & $5.18\times 10^{-1}$ & $7.52\times 10^{-2}$ & $5.35\times 10^{-1}$ & $8.41\times 10^{-2}$\\
        \hline
	\end{tabular}}
    \caption{Error of the 2D Neumann problem under various numbers of periods $n=1/\varepsilon$ at time $t=j\Delta t$.}
    \vspace{-8mm}
    \label{tab: Neumann}
\end{table}

\begin{figure}[h!]
    \centering
    \includegraphics[width=1.0\textwidth]{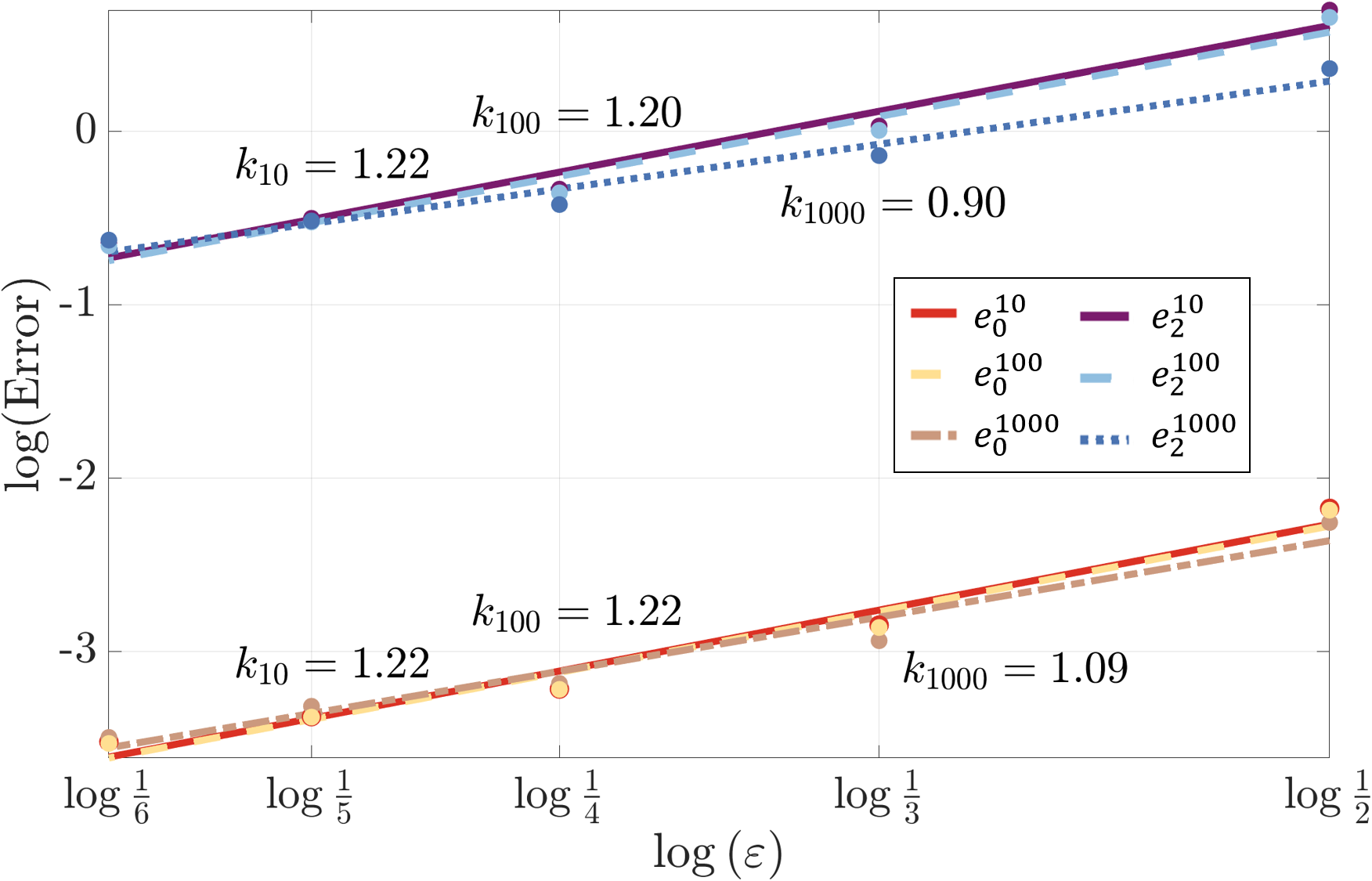}
    \caption{Variation of error $e_0^j$ and $e_2^j$ relative to the cell size $\varepsilon$ across different time steps $j$ in the 2D Neumann problem.}	\label{fig: Neumann}
    \vspace{-6mm}
\end{figure}

\subsection{2D Periodic problem with degenerated stray field}
The dominant exchange field is considered in previous two numerical experiments. However, the stray field 
also plays an important role in practical engineering, and  numerical simulation of its corresponding potential function is challenging.


In this numerical experiment, we consider the Periodic problem involving the exchange field and the stray field. For simplicity, the problem is restricted to the thin film materials. In this case, the stray field degenerates into a linear field, which can significantly decrease the complexity of the simulation.
The multiscale LLG equation is given by
\begin{equation*}
	\left\{\begin{aligned}
		&\partial_t \mathbf{m}^{\varepsilon} - \alpha \mathbf{m}^{\varepsilon} \times \partial_t \mathbf{m}^{\varepsilon} = -\left(1+\alpha^2\right) \mathbf{m}^{\varepsilon} \times \left[\bdiv\left(\ba^{\varepsilon} \nabla \mathbf{m}^{\varepsilon}\right) + \mu^{\varepsilon}\left(\mathbf{m}^{\varepsilon}\cdot \mathbf{e}_3\right)\mathbf{e}_3\right]
		\quad \text {in } \Omega, \\
		&\mathbf{m}^{\varepsilon}(\boldsymbol{x}, 0) =\mathbf{m}_{\mathrm{init}}^{\varepsilon}(\boldsymbol{x}), \quad\left|\mathbf{m}_{\mathrm{init}}^{\varepsilon}(\boldsymbol{x})\right|=1 \quad \text {in } \Omega,\\
	&\mathbf{m}^{\varepsilon}(\boldsymbol{x}, t),\, \mathbf{m}_{\mathrm{init}}^\varepsilon(\boldsymbol{x})  \quad \text {are periodic on } \partial \Omega .
	\end{aligned}\right.
\end{equation*}
The corresponding homogenized LLG equation can be expressed as
\begin{equation*}
	\left\{\begin{aligned}
		&\partial_t \mathbf{m}_0 - \alpha \mathbf{m}_0 \times \partial_t \mathbf{m}_0 = -\left(1+\alpha^2\right) \mathbf{m}_0 \times \left[\bdiv\left(\ba^0 \nabla \mathbf{m}_0\right) + \mu^0\left(\mathbf{m}_0\cdot \mathbf{e}_3\right)\mathbf{e}_3\right] \quad \text {in } \Omega, \\
		&\mathbf{m}_0(\boldsymbol{x}, 0) =\mathbf{m}_{\mathrm{init}}^0(\boldsymbol{x}), \quad\left|\mathbf{m}_{\mathrm{init}}^0(\boldsymbol{x})\right|=1 \quad \text {in } \Omega,\\
	&\mathbf{m}_0(\boldsymbol{x}, t),\, \mathbf{m}_{\mathrm{init}}^0(\boldsymbol{x})  \quad \text {are periodic on } \partial \Omega .
	\end{aligned}\right.
\end{equation*}
Here, we also obtain the initial data of the multiscale LLG equation by the expansion method \eqref{expansion method}. For the Periodic problem involving the exchange field and the degenerated stray field, the theoretical results remain consistent with the case considering only the exchange field. This is because only a linear term is introduced into the effective field. The corresponding convergence results are given as follows
\begin{equation*}
	\begin{gathered}
		\left\|\boldsymbol{m}^{\varepsilon}(\bx,t)-\boldsymbol{m}_0(\bx,t)\right\|_{L^2(\Omega)} \leq C \varepsilon^{1},\\
		\left\|\boldsymbol{m}^{\varepsilon}(\bx,t)-\boldsymbol{m}_0(\bx,t)-\varepsilon\boldsymbol{\chi}\left(\frac{\bx}{\epsilon}\right)\nabla\mathbf{m}_0(\bx,t)\right\|_{H^1(\Omega)} \leq C \varepsilon^{1}.
	\end{gathered}
\end{equation*}

In this case, the positive multiscale coefficient $\mu^\varepsilon$ is defined by 
\begin{equation*}
	\mu^\varepsilon(\by) = \big(1.1 + 0.25\sin{(2\pi y_1)}\big)\big(1.1 + 0.25\sin{(2\pi y_2)}\big). 
\end{equation*}
Moreover, the exchange coefficient $\ba^\varepsilon$ is consistent with \cref{2D-P} and \cref{2D-N}. Then, \cref{tab: periodicstray} and \cref{fig: periodicstray} present the error results. Considering the effect of stray field, we still obtain the computational results, which closely align with the theoretical results.

\begin{table}[h!]
    \centering
    \resizebox{\linewidth}{!}
    {\begin{tabular}{|c|c|c|c|c|c|c|}
        \hline
        \multirow{2}*{$n$} & \multicolumn{2}{c|}{j=10}  &\multicolumn{2}{c|}{j=$10^2$} &\multicolumn{2}{c|}{j=$10^3$}\\ \cline{2-7}  
                        &  $e_0$  &  $\tilde{e}_0$  
                        &  $e_0$  &  $\tilde{e}_0$ &  $e_0$  &  $\tilde{e}_0$ \\   
        \hline
        $2$   & $1.18\times 10^{-1}$ & $1.18\times 10^{-1}$ & $1.17\times 10^{-1}$ & $1.17\times 10^{-1}$  & $1.09\times 10^{-1}$ & $1.09\times 10^{-1}$\\
        \hline
        $3$   & $6.34\times 10^{-2}$ & $6.34\times 10^{-2}$ & $6.26\times 10^{-2}$ & $6.26\times 10^{-2}$ & $5.81\times 10^{-2}$ & $5.81\times 10^{-2}$\\
        \hline
        $4$   & $4.46\times 10^{-2}$ & $4.46\times 10^{-2}$ & $4.45\times 10^{-2}$ & $4.45\times 10^{-2}$ & $4.52\times 10^{-2}$ & $4.52\times 10^{-2}$\\
        \hline
        $5$   & $3.77\times 10^{-2}$ & $3.78\times 10^{-2}$ & $3.76\times 10^{-2}$ & $3.76\times 10^{-2}$ & $3.91\times 10^{-2}$ & $3.91\times 10^{-2}$\\
        \hline
        $6$   & $3.25\times 10^{-2}$ & $3.25\times 10^{-2}$ & $3.21\times 10^{-2}$ & $3.21\times 10^{-2}$ & $3.26\times 10^{-2}$ & $3.27\times 10^{-2}$\\
        \hline
		\quad &  $e_1$  &  $\tilde{e}_1$ &  $e_1$  &  $\tilde{e}_1$ &  $e_1$  &  $\tilde{e}_1$ \\   
		\hline
        $2$   & $2.14$ & $3.33\times 10^{-1}$ & $2.06$ & $3.23\times 10^{-1}$  & $1.54$ & $2.58\times 10^{-1}$\\
        \hline
        $3$   & $1.13$ & $1.55\times 10^{-1}$ & $1.10$ & $1.52\times 10^{-1}$ & $9.37\times 10^{-1}$ & $1.42\times 10^{-1}$\\
        \hline
        $4$   & $7.94\times 10^{-1}$ & $1.18\times 10^{-1}$ & $7.76\times 10^{-1}$ & $1.17\times 10^{-1}$ & $7.05\times 10^{-1}$ & $1.13\times 10^{-1}$\\
        \hline
        $5$   & $6.72\times 10^{-1}$ & $9.94\times 10^{-2}$ & $6.57\times 10^{-1}$ & $9.80\times 10^{-2}$ & $6.34\times 10^{-1}$ & $1.01\times 10^{-1}$\\
        \hline
        $6$   & $5.81\times 10^{-1}$ & $8.35\times 10^{-2}$ & $5.70\times 10^{-1}$ & $8.28\times 10^{-2}$ & $5.64\times 10^{-1}$ & $8.87\times 10^{-2}$\\
        \hline
	\end{tabular}}
    \caption{Error of the 2D Periodic problem with exchange field and degenerated stray field under various numbers of periods $n=1/\varepsilon$ at time $t=j\Delta t$.}
    \label{tab: periodicstray}
    \vspace{-4mm}
\end{table}

\begin{figure}[h!]
    \centering
    \includegraphics[width=1.0\textwidth]{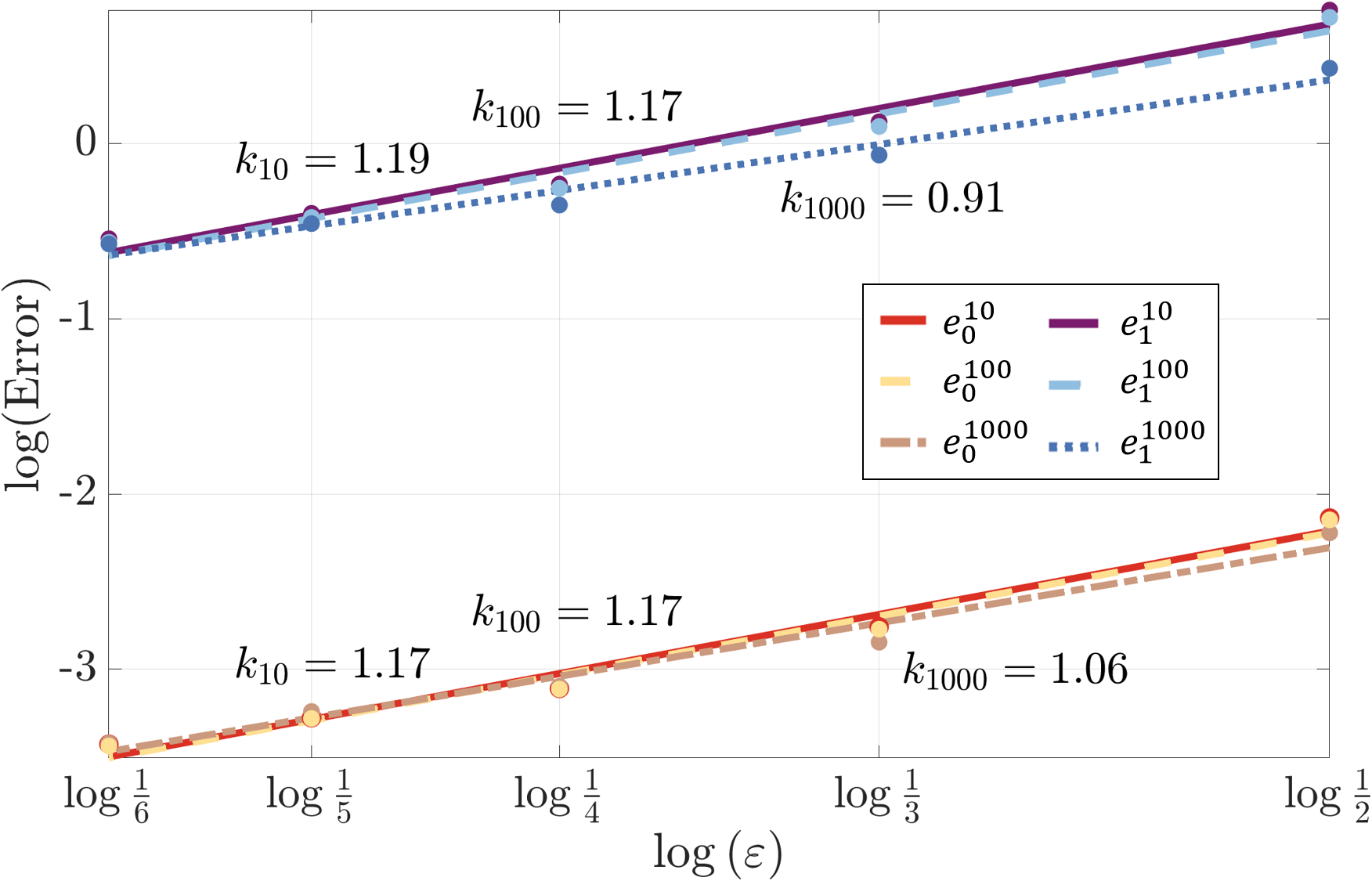}
    \caption{Variation of error $e_0^j$ and $e_1^j$ relative to the cell size $\varepsilon$ across different time steps $j$ in the 2D Periodic problem with exchange field and degenerated stray field.}
	\label{fig: periodicstray}
    \vspace{-6mm}
\end{figure}

\subsection{3D Periodic problem}
For the 3D case, the Periodic problem with the exchange field is cosidered in domain $\Omega=[0,1]^3$. In this case, the multiscale and homogenized LLG equation are consistent with \cref{2D-P}.
The convergence results can be expressed as
\vspace{-2mm}
\begin{equation*}
	\begin{gathered}
		\left\|\boldsymbol{m}^{\varepsilon}(\bx,t)-\boldsymbol{m}_0(\bx,t)\right\|_{L^2(\Omega)} \leq C \varepsilon^{1},\\
		\left\|\boldsymbol{m}^{\varepsilon}(\bx,t)-\boldsymbol{m}_0(\bx,t)-\varepsilon\boldsymbol{\chi}\left(\frac{\bx}{\epsilon}\right)\nabla\mathbf{m}_0(\bx,t)\right\|_{H^1(\Omega)} \leq C \varepsilon^{1}.
	\end{gathered}
\end{equation*}

The spatial grid sizes of the reference and the homogenized solutions are $h=1/30$ and $h=1/24$, and the same time step size $\Delta t=5\times 10^{-5}$ is used. Moreover, the exchange coefficient is given by
\begin{equation*}
	\begin{aligned}
    	\ba(\by)=&\Big(1.1+0.25\cos{\big(2\pi(y_1-0.5)\big)}\Big)\Big(1.1+0.25\cos{\big(2\pi(y_2-0.5)\big)}\Big)\\
		&\cdot\Big(1.1+0.25\cos{\big(2\pi(y_3-0.5)\big)}\Big)\times I_3,		    
	\end{aligned}
\end{equation*}
where $I_3$ is the $3\times 3$ identity matrix. 



\cref{tab: periodic_3D} presents the error results at different time $t=j\Delta t$ with various numbers of periods $n$, where the corresponding error convergence orders are depicted in \cref{fig: periodic_3D}. 
From the figure, it can be observed that the convergence orders of the $L^2$ error closely align with the theoretical results. Due to limitation of the mesh size for the reference solution, the convergence orders of $H^1$ error exhibit a slight deviation. Specifically, the 
deviation at the initial time 
is caused by the approximation of the initial data
for multiscale systems.
However, the convergence orders approach the theoretical values over time.

\begin{table}[h!]
    \centering
    \begin{tabular}{|c|c|c|c|c|}
        \hline
        \multirow{2}*{$n$} & \multicolumn{2}{c|}{j=10}  &\multicolumn{2}{c|}{j=$10^2$}\\ \cline{2-5}  
                        &  $e_0$  &  $\tilde{e}_0$  
                        &  $e_0$  &  $\tilde{e}_0$\\   
        \hline
        $2$   & $9.51\times 10^{-2}$ & $9.55\times 10^{-2}$ & $9.88\times 10^{-2}$ & $9.92\times 10^{-2}$\\
        \hline
        $3$   & $6.48\times 10^{-2}$ & $6.51\times 10^{-2}$ & $6.52\times 10^{-2}$ & $6.55\times 10^{-2}$\\
        \hline
        $5$   & $3.72\times 10^{-2}$ & $3.74\times 10^{-2}$ & $3.78\times 10^{-2}$ & $3.80\times 10^{-2}$\\
        \hline
		\quad &  $e_1$  &  $\tilde{e}_1$ &  $e_1$  &  $\tilde{e}_1$\\   
		\hline
        $2$   & $6.85\times 10^{-1}$ & $1.03\times 10^{-1}$ & $7.65\times 10^{-1}$ & $1.13\times 10^{-1}$\\
        \hline
        $3$   & $6.30\times 10^{-1}$ & $9.47\times 10^{-2}$ & $6.65\times 10^{-1}$ & $9.86\times 10^{-2}$\\
        \hline
        $5$   & $5.44\times 10^{-1}$ & $8.22\times 10^{-2}$ & $5.30\times 10^{-1}$ & $7.93\times 10^{-2}$\\
        \hline
	\end{tabular}
    \caption{Error of the 3D Periodic problem under various numbers of periods $n=1/\varepsilon$ at time $t=j\Delta t$.}
    \vspace{-6mm}
    \label{tab: periodic_3D}
\end{table}

\begin{figure}[h!]
    \centering
    \includegraphics[width=1.0\textwidth]{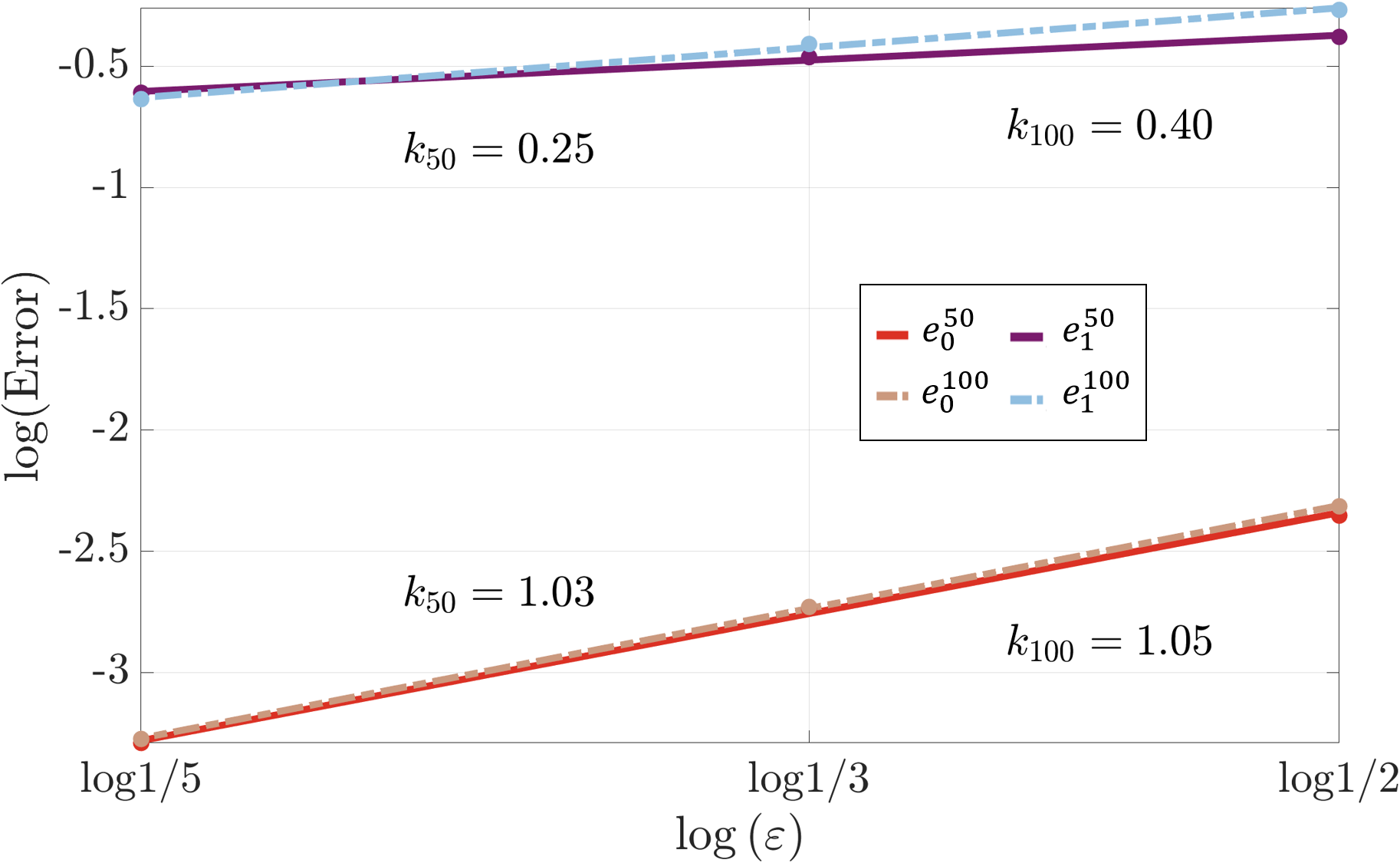}
    \caption{Variation of error $e_0^j$ and $e_1^j$ relative to the cell size $\varepsilon$ across different time steps $j$ in the 3D Periodic problem.}
	\label{fig: periodic_3D}
    \vspace{-6mm}
\end{figure}

\section{Conclusions}\label{sec5}
In this paper, we present novel theoretical results and the corresponding numerical framework for solving the multiscale LLG equation based on the two-scale method.
First, we focus on the more realistic and complex multiscale LLG model, which considers the contribution of the exchange, anisotropy, stray, and external magnetic fields. With the proper second-order corrector, we derive the approximation error of \eqref{two scale approximate} in the $H^1$ norm and its uniform $W^{1,6}$ estimate.
Second, a robust numerical framework is proposed, which is then employed to effectively validate the theoretical convergence order for both the Periodic and Neumann problems.
Third, to overcome the computational prohibitive induced by $\epsilon$, we design an improved implicit scheme for temporal discretization, which significantly reduces the required number of iterations and relaxes the constraints on the time step size. Specifically, the projection and expansion methods are proposed to overcome the inherent non-consistency in the initial data between the multiscale problem and its homogenized problem.


Moreover, magnetostriction induced by the variations in magnetization arouses the interest and finds widespread applications in practical engineering, such as magnetic displacement sensors. This physical quantity can be characterized by the system, which couples the multiscale LLG equation with the elastic wave equation \cite{BPPR:IJNA:2014}. In our forthcoming work, the issue will be discussed.

\appendix
\section{Proof of Lemma \ref{thm: Consistency estimate of two-scale approximation}}\label{first lemma}
This proof refer to the estimate of $\boldsymbol{\Theta}^\varepsilon_\mathrm{sf}$ in \cite{CLS:apa:2022}.
By utilizing the two-scale expansion, $\boldsymbol{\Theta}^\varepsilon_\mathrm{ts}$ defined in \eqref{eqn:equivalent system of m epsilon}
 can be rewritten in the following form
\begin{equation*}\label{expansion of consistance error}
	\begin{aligned}
		\boldsymbol{\Theta}^\varepsilon_\mathrm{ts} = 
		&\ \varepsilon^{-2}
		\boldsymbol{\Theta}^{(-2)}_\mathrm{ts}\left(\boldsymbol{x},\frac{\boldsymbol{x}}{\varepsilon}\right)
		+ \varepsilon^{-1}\boldsymbol{\Theta}^{(-1)}_\mathrm{ts}\left(\boldsymbol{x},\frac{\boldsymbol{x}}{\varepsilon}\right)\\ 
		&+ \boldsymbol{\Theta}^{(0)}_\mathrm{ts}\left(\boldsymbol{x},\frac{\boldsymbol{x}}{\varepsilon}\right) 
		+ \varepsilon \boldsymbol{\Theta}^{(1)}_\mathrm{ts}\left(\boldsymbol{x},\frac{\boldsymbol{x}}{\varepsilon}\right) 
		+ \varepsilon^{2}\boldsymbol{\Theta}^{(2)}_\mathrm{ts}\left(\boldsymbol{x},\frac{\boldsymbol{x}}{\varepsilon}\right).
	\end{aligned}
\end{equation*}
For $\widetilde{\mathbf{m}}^\varepsilon$ and $\widetilde{U}^\varepsilon$ defined in \eqref{def two scale approximate sln}, 
it can be easily checked that $\boldsymbol{\Theta}^{(-2)}_\mathrm{ts} = 0$ and $\boldsymbol{\Theta}^{(-1)}_\mathrm{ts} = 0$. 
In the following, $\boldsymbol{\Theta}^{(0)}_\mathrm{ts} = 0$ will be proved.
In fact, the zero-order term $\boldsymbol{\Theta}^{(0)}_\mathrm{ts}$ yields
\begin{equation}\label{eqn:value of f0}
	\begin{aligned}
		\boldsymbol{\Theta}^{(0)}_\mathrm{ts}(\boldsymbol{x},\boldsymbol{y})
		=&\ \partial_t \mathbf{m}_0 - \alpha \big\{ \mathcal{A}_0 \mathbf{m}_2 + \boldsymbol{h}(\boldsymbol{x},\boldsymbol{y}) \big\} \\
		&+\mathbf{m}_0 \times \big\{ \mathcal{A}_0 \mathbf{m}_2 + \boldsymbol{h}(\boldsymbol{x},\boldsymbol{y}) \big\}  
		- \alpha g(\boldsymbol{x},\boldsymbol{y})\mathbf{m}_0.
	\end{aligned}
\end{equation}
where $\bh(\bx,\by)$ and $g(\bx,\by)$ are defined as
\begin{equation}\label{define h and g}
	\hspace{-1mm}
        \left\{\begin{aligned}
		\bh(\bx,\by) &= \mathcal{A}_1 \m_1 +  \mathcal{A}_2 \m_0 + \mu\nabla_{\bx} U_0 + \mu\nabla_{\by} U_1 - K\left(\mathbf{m}_0 \cdot \mathbf{u}\right) \mathbf{u} + M_s \mathbf{h}_{\mathrm{a}},\\
		g(\boldsymbol{x},\boldsymbol{y}) &= \ba|\nabla_{\boldsymbol{x}} \mathbf{m}_0 + \nabla_{\boldsymbol{y}} \mathbf{m}_1|^2 
		-\mu \mathbf{m}_0 \cdot  (\nabla_{\boldsymbol{x}} U_0 + \nabla_{\boldsymbol{y}} U_1) + K\left(\mathbf{m}_0 \cdot \mathbf{u}\right)^2 - M_s \m_0\cdot\mathbf{h}_{\mathrm{a}}.
	\end{aligned}\right.
\end{equation}
The classical solution $\mathbf{m}_0$ of \eqref{homogenized model periodic} satisfies
\begin{gather}\label{eqn:equivalent system of m0}
	\partial_t\mathbf{m}_0 - \alpha \bh_{\mathrm{eff}}^0
	+ \mathbf{m}_0 \times \bh_{\mathrm{eff}}^0
	+ \alpha
	(\m_0\cdot \bh_{\mathrm{eff}}^0) \mathbf{m}_0
	= 0.
\end{gather}
By substituting \eqref{eqn:equivalent system of m0} into \eqref{eqn:value of f0}, we can derive that
\begin{equation}\label{eqn:value of f0 rewrite}
	\begin{aligned}
		\boldsymbol{\Theta}^{(0)}_\mathrm{ts}(\boldsymbol{x},\boldsymbol{y})
		=& - \alpha \big\{ \mathcal{A}_0 \mathbf{m}_2 + \boldsymbol{h}(\boldsymbol{x},\boldsymbol{y}) - \bh_{\mathrm{eff}}^0 \big\} \\
		&+ \mathbf{m}_0 \times \big\{ \mathcal{A}_0 \mathbf{m}_2 + \boldsymbol{h}(\boldsymbol{x},\boldsymbol{y}) - \bh_{\mathrm{eff}}^0 \big\}  
		- \alpha \big\{g(\boldsymbol{x},\boldsymbol{y}) + (\m_0\cdot \bh_{\mathrm{eff}}^0)\big\}\mathbf{m}_0.
	\end{aligned}
\end{equation}
Applying $\mathrm{div}_{\boldsymbol{y}} (\ba\nabla_{\boldsymbol{x}}) + \mathrm{div}_{\boldsymbol{x}} (\ba\nabla_{\boldsymbol{y}})$ to the both sides of geometric constraint $\mathbf{m}_0 \cdot\mathbf{m}_1 = 0$, we have
\begin{equation*}\label{eqn:value of f0 form3}
	\begin{aligned} 
		\mathrm{div}_{\boldsymbol{y}} (\ba\nabla_{\boldsymbol{x}} \mathbf{m}_0 ) \cdot \mathbf{m}_1
		&+
		2\ba \nabla_{\boldsymbol{x}}\mathbf{m}_0 \cdot \nabla_{\boldsymbol{y}}\mathbf{m}_1\\
        &+
         \mathbf{m}_0 \cdot \mathrm{div}_{\boldsymbol{x}} (\ba\nabla_{\boldsymbol{y}}\mathbf{m}_1)
		+
		\mathbf{m}_0 \cdot \mathrm{div}_{\boldsymbol{y}} (\ba\nabla_{\boldsymbol{x}} \mathbf{m}_1 ) =0.
	\end{aligned}
\end{equation*}
Moreover, based on the definition in \eqref{first-order corrector}-\eqref{first-order cell problem}, it can be proved that
\begin{equation*}
	\mathbf{m}_1 \cdot \mathrm{div}_{\boldsymbol{y}} (\ba\nabla_{\boldsymbol{y}} \mathbf{m}_1 )
	+
	\mathbf{m}_1 \cdot  \mathrm{div}_{\boldsymbol{y}} (\ba\nabla_{\boldsymbol{x}} \mathbf{m}_0 )=0.
\end{equation*}
With the above results, it yields
\begin{equation*}
	\begin{aligned} 
 2\ba \nabla_{\boldsymbol{x}}\mathbf{m}_0 \cdot \nabla_{\boldsymbol{y}}\mathbf{m}_1
 = &
 \mathbf{m}_1 \cdot \mathrm{div}_{\boldsymbol{y}} (\ba\nabla_{\boldsymbol{y}} \mathbf{m}_1 )\\
		&-
         \mathbf{m}_0 \cdot \mathrm{div}_{\boldsymbol{x}} (\ba\nabla_{\boldsymbol{y}}\mathbf{m}_1)
		-
		\mathbf{m}_0 \cdot \mathrm{div}_{\boldsymbol{y}} (\ba\nabla_{\boldsymbol{x}} \mathbf{m}_1 )\\
  =&
  \m_1\cdot \mathcal{A}_0 \m_1 - \m_0\cdot \mathcal{A}_1 \m_1.
	\end{aligned}
\end{equation*}
Utilizing $\ba|\nabla_{\boldsymbol{x}} \mathbf{m}_0|^2 = -\m_0\cdot \mathcal{A}_2 \m_0$, we have
\begin{equation}\label{formula from constrate}
	\begin{aligned} 
 \ba|\nabla_{\boldsymbol{x}} \mathbf{m}_0 + \nabla_{\boldsymbol{y}} \mathbf{m}_1|^2 
  =&
  - \m_0\cdot (\mathcal{A}_2 \m_0 + \mathcal{A}_1 \m_1)
+   \m_1\cdot \mathcal{A}_0 \m_1 + \ba|\nabla_{\boldsymbol{y}} \mathbf{m}_1|^2 .
	\end{aligned}
\end{equation}
Combining \eqref{formula from constrate} with 
\eqref{define h and g}
, it can be derived that 
\begin{equation}\label{A5 g}
    \begin{aligned}
        g(\boldsymbol{x},\boldsymbol{y}) =&  - \m_0\cdot (\mathcal{A}_2 \m_0 + \mathcal{A}_1 \m_1)
+   \m_1\cdot \mathcal{A}_0 \m_1 + \ba|\nabla_{\boldsymbol{y}} \mathbf{m}_1|^2\\
		&+\mu \mathbf{m}_0 \cdot  (\nabla_{\boldsymbol{x}} U_0 + \nabla_{\boldsymbol{y}} U_1) - K\left(\mathbf{m}_0 \cdot \mathbf{u}\right)^2 + M_s \m_0\cdot\mathbf{h}_{\mathrm{a}}\\
  =&- \m_0 \cdot \bh(\bx,\by)
  +
  \m_1\cdot \mathcal{A}_0 \m_1 + \ba|\nabla_{\boldsymbol{y}} \mathbf{m}_1|^2.
    \end{aligned}
\end{equation}
Substituting \eqref{A5 g} into \eqref{eqn:value of f0 rewrite}, we have
\begin{equation}\label{theta simp}
	\begin{aligned}
		\boldsymbol{\Theta}^{(0)}_\mathrm{ts}(\boldsymbol{x},\boldsymbol{y})
		=& - \alpha \big\{ \mathcal{A}_0 \mathbf{m}_2 + \boldsymbol{h}(\boldsymbol{x},\boldsymbol{y}) - \bh_{\mathrm{eff}}^0 \big\} 
		+ \mathbf{m}_0 \times \big\{ \mathcal{A}_0 \mathbf{m}_2 + \boldsymbol{h}(\boldsymbol{x},\boldsymbol{y}) - \bh_{\mathrm{eff}}^0 \big\} \\ 
		&+ \alpha \m_0 \cdot \big\{\boldsymbol{h}(\boldsymbol{x},\boldsymbol{y}) - \bh_{\mathrm{eff}}^0\big\}\mathbf{m}_0
  -
  \alpha(\m_1\cdot \mathcal{A}_0 \m_1 + \ba|\nabla_{\boldsymbol{y}} \mathbf{m}_1|^2) \mathbf{m}_0.
	\end{aligned}
\end{equation}
Combing \eqref{eqn:system m2 form3} with \eqref{theta simp}, 
it can derive that $\boldsymbol{\Theta}^{(0)}_\mathrm{ts} = 0$. With the above results, we have 
\begin{equation}\label{theta_ts remain}
	\boldsymbol{\Theta}^{\varepsilon}_\mathrm{ts}\left(\boldsymbol{x},\frac{\boldsymbol{x}}{\varepsilon}\right)
	=
	\varepsilon \boldsymbol{\Theta}^{(1)}_\mathrm{ts}\left(\boldsymbol{x},\frac{\boldsymbol{x}}{\varepsilon}\right)
	+
	\varepsilon^2 \boldsymbol{\Theta}^{(2)}_\mathrm{ts}\left(\boldsymbol{x},\frac{\boldsymbol{x}}{\varepsilon}\right).
\end{equation}
Furthermore, it can be checked that the remainder terms $\boldsymbol{\Theta}^{(1)}_\mathrm{ts}$ and $\boldsymbol{\Theta}^{(2)}_\mathrm{ts}$ satisfy
\begin{equation}\label{theta1 theta2 est}
	\left\Vert \boldsymbol{\Theta}^{(1)}_\mathrm{ts}\left(\boldsymbol{x},\frac{\boldsymbol{x}}{\varepsilon}\right) \right\Vert_{L^2(\Omega)} 
	+
	\left\Vert \boldsymbol{\Theta}^{(2)}_\mathrm{ts}\left(\boldsymbol{x},\frac{\boldsymbol{x}}{\varepsilon}\right) \right\Vert_{L^2(\Omega)} 
	\le C,
\end{equation}
where the constant $C$ depends on $\Vert \mathbf{m}_0 \Vert_{W^{4,\infty}(\Omega)}$ and $C_{\mathrm{coe}}$. Then, \cref{thm: Consistency estimate of two-scale approximation} can be derived by \eqref{theta_ts remain} and \eqref{theta1 theta2 est}. The proof is completed.

\section{Proof of Lemma \ref{theorem:stability}}\label{second lemma}
The following proposition about the stray field has been proved in \cite{D:ZFAUIA:2004}.
\begin{proposition}\label{pro: stray field}
	Let $\dot{W}^{1,p}$ denote the Beppo-levi space $(1< p< +\infty)$. For any vector function $\mathbf{f}\in [L^p(\mathbb{R}^3)]^3$, there exists unique solution $U \in \dot{W}^{1,p}(\mathbb{R}^3)$ such that
	\begin{equation*}
		\Delta U = \bdiv\, \mathbf{f},\quad \text{in $D'(\mathbb{R}^3)$},
	\end{equation*}
	and the solution satisfies the estimate
	\begin{equation*}
		\Vert \nabla U \Vert_{L^p(\mathbb{R}^3)}\le C \Vert \mathbf{f} \Vert_{L^p(\mathbb{R}^3)}.
	\end{equation*}
\end{proposition}
Based on \cref{pro: stray field}, for $1<p<+\infty$, one can derive that
\begin{equation}\label{estimate of Phi}
	\begin{aligned}
		&\Vert \nabla \Psi^\varepsilon\Vert_{L^p(\Omega)}
		\le C_p
		\Vert \mathbf{e}^\varepsilon\Vert_{L^p(\Omega)},\quad
		\Vert \nabla U^\varepsilon\Vert_{L^p(\Omega)}
		\le C_p
		\Vert \m^\varepsilon\Vert_{L^p(\Omega)},\\
		&\Vert \nabla \Gamma^\varepsilon\Vert_{L^p(\Omega)}
		\le C_p
		\Vert \widetilde{\m}^\varepsilon\Vert_{L^p(\Omega)}.
	\end{aligned}
\end{equation}
By the $W^{1,p}$ estimate of the oscillatory elliptic problem, the following inequalities are also introduced.
\begin{proposition}\label{lemma: regularity for epslon}
	Given $u \in H^{2}(\Omega)$,
	then it holds that
	\begin{equation*}
		\Vert \nabla u \Vert_{L^{6}(\Omega)}
		\le C
		\Vert \mathcal{A}_\epsilon u \Vert_{L^{2}(\Omega)}
		+ C\Vert \nabla u \Vert_{L^{2}(\Omega)}, 
	\end{equation*}
	where the constant $C$ is independent of $\epsilon$.
\end{proposition}
\begin{proof}
Denote $\boldsymbol{\nu} \cdot \ba^\epsilon\nabla u = g$  on $\partial \Omega$ and refer to Theorem 6.3.2 in \cite{S:SIP:2018},
	we have
 \begin{equation*}
     \Vert \nabla u \Vert_{L^{6}(\Omega)}
		\le C
		\Vert \mathcal{A}_\epsilon u \Vert_{L^{2}(\Omega)}
		+ C\Vert g \Vert_{B^{-1/2,2}(\partial\Omega)}.
 \end{equation*}
 Moreover, one can apply the following trace theorem
 \begin{equation*}
     \Vert g \Vert_{B^{-1/2,2}(\partial\Omega)}
     \le C \Vert \nabla u \Vert_{L^{2}(\Omega)}
     + C \Vert \mathcal{A}_\epsilon u \Vert_{L^{2}(\Omega)}.
 \end{equation*}
 Then the Proposition can be proved.
\end{proof}
With the above results, \cref{theorem:stability} can be proved as follows.\\
\noindent\textbf{Step 1: $L^2(\Omega)$ estimate.}
Taking the $L^2(\Omega)$ inner product of \eqref{system of e} and $\mathbf{e}^\varepsilon$, it can be derived that
\begin{equation}
	\begin{aligned}\label{eqn:product with e}
		&\frac{1}{2} \frac{\mathrm{d}}{\mathrm{d} t} \int_{\Omega} \vert \mathbf{e}^\varepsilon \vert^2 \mathrm{d} \boldsymbol{x}
		-
		\alpha  \int_{\Omega} \widetilde{\mathcal{H}}^\varepsilon_e(\mathbf{e}^\varepsilon, \Psi^\varepsilon) \cdot \mathbf{e}^\varepsilon \mathrm{d} \boldsymbol{x}\\
		= & 
		- \int_{\Omega}
		\mathbf{D}_1(\mathbf{e}^\varepsilon, \Psi^\varepsilon)  \cdot \mathbf{e}^\varepsilon \mathrm{d} \boldsymbol{x} 
		-
		\int_{\Omega}
		\mathbf{D}_2(\mathbf{e}^\varepsilon, \Psi^\varepsilon) \cdot \mathbf{e}^\varepsilon\mathrm{d} \boldsymbol{x}
		+
		\int_{\Omega} \boldsymbol{F} \cdot \mathbf{e}^\varepsilon \mathrm{d} \boldsymbol{x},
	\end{aligned}
\end{equation}
where $\boldsymbol{F} =
- \big(\boldsymbol{\Theta}^\varepsilon_\mathrm{ts}
+
\boldsymbol{\Theta}^\varepsilon_\mathrm{sf}\big)$. 
Using estimate \eqref{estimate of Phi} and
integration by parts for the second term on the left-hand side of \eqref{eqn:product with e}, it follows that
\begin{equation}\label{inter part of lh}
	- \int_{\Omega} \widetilde{\mathcal{H}}^\varepsilon_e(\mathbf{e}^\varepsilon, \Psi^\varepsilon) \cdot \mathbf{e}^\varepsilon \mathrm{d} \boldsymbol{x}
	\ge
	\sum_{i, j= 1}^{n}
	\int_{\Omega} a^\varepsilon_{ij} \frac{\partial\mathbf{e}^\varepsilon}{\partial x_i} \cdot
	\frac{\partial\mathbf{e}^\varepsilon}{\partial x_j} \mathrm{d} \boldsymbol{x}
	- 
	C \int_{\Omega} \vert \mathbf{e}^\varepsilon \vert^2 \mathrm{d} \boldsymbol{x}.
\end{equation}
Based on the same idea of \eqref{eqn:product with e} and \eqref{inter part of lh} with Young's inequality,
the first term on the right-hand side of \eqref{eqn:product with e} can be estimated by
\begin{equation*}\label{bounded of D_1 e cdot e}
	\begin{aligned}
		- \int_{\Omega}
		\mathbf{D}_1(\mathbf{e}^\varepsilon, \Psi^\varepsilon)  \cdot \mathbf{e}^\varepsilon \mathrm{d} \boldsymbol{x} 
		\le
		& C \int_{\Omega} \vert \mathbf{e}^\varepsilon \vert^2 \mathrm{d} \boldsymbol{x}
		+
		\delta C \int_{\Omega} \vert \nabla \mathbf{e}^\varepsilon \vert^2 \mathrm{d} \boldsymbol{x}.
	\end{aligned}
\end{equation*}
For the second term on the right-hand side of \eqref{eqn:product with e}, the following estimates are used
\begin{equation*}
	\left\{	\begin{aligned}
		&\int_{\Omega}
		\big( \nabla\e^\epsilon \cdot \ba^\epsilon\nabla\m^\epsilon 
		\big)\m^\epsilon \cdot \e^\epsilon\d \bx
		\le  C
		\Vert \nabla \m^\epsilon \Vert_{L^{6}(\Omega)}
		\Vert \e^\epsilon\Vert_{L^{3}(\Omega)}
		\Vert \nabla \e^\epsilon\Vert_{L^{2}(\Omega)}, \\
		&\int_{\Omega}
		\big( \mu^\epsilon\e^\epsilon \cdot \nabla U^\epsilon 
		\big)\m^\epsilon \cdot \e^\epsilon\d \bx
		\le  C
		\Vert \nabla U^\epsilon \Vert_{L^{6}(\Omega)}
		\Vert \e^\epsilon\Vert_{L^{3}(\Omega)}
		\Vert \e^\epsilon\Vert_{L^{2}(\Omega)},\\
		&\int_{\Omega}
		\big( \mu^\epsilon \widetilde{\m}^\epsilon \cdot \nabla \Psi^\epsilon 
		\big)\m^\epsilon \cdot \e^\epsilon\d \bx
		\le C
		\Vert \nabla \Psi^\epsilon  \Vert_{L^{6}(\Omega)}
		\Vert \e^\epsilon\Vert_{L^{3}(\Omega)}
		\Vert \e^\epsilon\Vert_{L^{2}(\Omega)},
	\end{aligned}\right.
\end{equation*}
Moreover, applying \eqref{estimate of Phi} and Young's inequality, together with the embedding theorem of the $H^1(\Omega)$ space into $L^p(\Omega)$ space for $1\le p\le 6$, we have 
\begin{equation*}\label{bounded of D_2 e cdot e}
	\begin{aligned}
		- \int_{\Omega}
		\mathbf{D}_2(\e^\epsilon, \Psi^\epsilon )  \cdot \e^\epsilon\d \bx 
		\le 
		C \Vert \e^\epsilon\Vert_{L^{2}(\Omega)}^2
		+
		\delta C \Vert \e^\epsilon\Vert_{H^{1}(\Omega)}^2,
	\end{aligned}
\end{equation*}
where $C$ depends on $\Vert \nabla \m^\epsilon  \Vert_{L^{6}(\Omega)}$, $\Vert \nabla \widetilde{\m}^\epsilon  \Vert_{L^{6}(\Omega)}$,  $\Vert \nabla U^\epsilon  \Vert_{L^{6}(\Omega)}$,  $\Vert \nabla \Gamma^\epsilon  \Vert_{L^{6}(\Omega)}$, $\Vert \m^\epsilon \Vert_{L^{\infty}(\Omega)}$ and  $\Vert \widetilde{\m}^\epsilon \Vert_{L^{\infty}(\Omega)}$.
For the last term in \eqref{eqn:product with e}, the following estimate can be deduced by Young's inequality and Sobolev inequality 
\begin{equation*}
	\begin{aligned}
		\int_{\Omega} \boldsymbol{F} \cdot \mathbf{e}^\epsilon \mathrm{d} \boldsymbol{x}
		\le& C
		\Vert \boldsymbol{\Theta}^\varepsilon_\mathrm{ts} \Vert_{L^2(\Omega)}^2
		+ C
		\Vert \mathbf{e}^\epsilon \Vert_{L^2(\Omega)}
		+ C
		\Vert \boldsymbol{\Theta}^\varepsilon_\mathrm{sf} \Vert_{L^{6/5}(\Omega)}^2
		+ \delta C
		\Vert \mathbf{e}^\epsilon \Vert_{H^1(\Omega)},
	\end{aligned}
\end{equation*}
with any small $\delta>0$. 
Utilizing the above estimates, it can be obtained from \eqref{eqn:product with e} that
\begin{equation}\label{increasing inequality}
	\begin{aligned}
		&\frac{1}{2}\frac{\d}{\d t} \Vert \e^\epsilon\Vert_{L^{2}(\Omega)}^2
		+
		(\alpha a_{\mathrm{min}} - C \delta) \Vert \nabla \e^\epsilon\Vert_{L^{2}(\Omega)}^2\\
		\le 
		&C \Vert \e^\epsilon\Vert_{L^{2}(\Omega)}^2
		+ 
		C
		\Vert \boldsymbol{\Theta}^\varepsilon_\mathrm{ts} \Vert_{L^2(\Omega)}^2
		+ C
		\Vert \boldsymbol{\Theta}^\varepsilon_\mathrm{sf} \Vert_{L^{6/5}(\Omega)}^2.
	\end{aligned}
\end{equation}
Let $\delta\to 0$, the estimate \eqref{ineq of e^eps_b n=3} in \cref{theorem:stability} can be directly deduced by Gr\"{o}nwall's inequality.

\noindent\textbf{Step 2: $H^1(\Omega)$ estimate.}
Taking the $L^2(\Omega)$ inner product of \eqref{system of e} and $\widetilde{\mathcal{H}}^\epsilon_e(\e^\epsilon, \Psi^\epsilon )$, we obtain 
\begin{equation}\label{eqn:product with H}
	\begin{aligned}
		& - \int_{\Omega} \partial_t \mathbf{e}^\epsilon \cdot \widetilde{\mathcal{H}}^\epsilon_e(\e^\epsilon, \Psi^\epsilon ) \mathrm{d} \boldsymbol{x}
		+
		\alpha  \int_{\Omega} \widetilde{\mathcal{H}}^\epsilon_e(\e^\epsilon, \Psi^\epsilon ) \cdot \widetilde{\mathcal{H}}^\epsilon_e(\e^\epsilon, \Psi^\epsilon ) \mathrm{d} \boldsymbol{x}\\
		= &
		\int_{\Omega}
		\mathbf{D}_1 (\e^\epsilon, \Psi^\epsilon )  \cdot \widetilde{\mathcal{H}}^\epsilon_e(\e^\epsilon, \Psi^\epsilon ) \mathrm{d} \boldsymbol{x} 
		+
		\int_{\Omega}
		\mathbf{D}_2(\e^\epsilon, \Psi^\epsilon ) \cdot \widetilde{\mathcal{H}}^\epsilon_e(\e^\epsilon, \Psi^\epsilon ) \mathrm{d} \boldsymbol{x}\\
		&-
		\int_{\Omega}
		\boldsymbol{F} \cdot \widetilde{\mathcal{H}}^\epsilon_e(\e^\epsilon, \Psi^\epsilon ) \mathrm{d} \boldsymbol{x}.
	\end{aligned}
\end{equation}
Subsequently, the term-by-term estimates for \eqref{eqn:product with H} are provided.
The integration by parts on the first term of \eqref{eqn:product with H} yields
\begin{align*}
	- \int_{\Omega} \partial_t \mathbf{e}^\epsilon \cdot \widetilde{\mathcal{H}}^\epsilon_e(\e^\epsilon, \Psi^\epsilon ) \mathrm{d} \boldsymbol{x}
	= &
	\frac{\mathrm{d}}{\mathrm{d} t} \mathcal{G}_{\mathcal{L}}^\epsilon[\mathbf{e}^\epsilon, \Psi^\epsilon].
\end{align*}
Here, $\mathcal{G}_{\mathcal{L}}^\epsilon[\mathbf{e}^\epsilon,\Psi^\epsilon] = \int_{\Omega} g_l^\epsilon (\mathbf{e}^\epsilon,\Psi^\epsilon) \mathrm{d} \boldsymbol{x}$ is the total energy, with the energy density $g_l^\epsilon$ defined in \eqref{energy density of epsilon}.
Using $\mathbf{m}^\epsilon\times \widetilde{\mathcal{H}}^\epsilon_e(\e^\epsilon, \Psi^\epsilon ) \cdot\widetilde{\mathcal{H}}^\epsilon_e(\e^\epsilon, \Psi^\epsilon )=0$, the first term on the right-hand side of \eqref{eqn:product with H} can be estimated by Sobolev inequality
\begin{align*}
	\int_{\Omega}
	\mathbf{D}_1 (\e^\epsilon, \Psi^\epsilon ) \cdot \widetilde{\mathcal{H}}^\epsilon_e(\e^\epsilon, \Psi^\epsilon ) \mathrm{d} \boldsymbol{x} 
	&\le 
	C \Vert \mathcal{H}^\epsilon_e (\widetilde{\mathbf{m}}^\epsilon, \Gamma^\epsilon) \Vert_{L^{3}(\Omega)}
	\Vert \mathbf{e}^\epsilon \Vert_{L^{6}(\Omega)}
	\Vert \widetilde{\mathcal{H}}^\epsilon_e(\e^\epsilon, \Psi^\epsilon ) \Vert_{L^{2}(\Omega)} \nonumber\\
	&\le
	C \Vert \mathbf{e}^\epsilon \Vert_{H^{1}(\Omega)}^2
	+
	\delta C\Vert \widetilde{\mathcal{H}}^\epsilon_e(\e^\epsilon, \Psi^\epsilon ) \Vert_{L^{2}(\Omega)}^2, \label{bounded of D_1 e dot H e}
\end{align*}
where $C = C^0\big(1 + \Vert \mathcal{H}^\epsilon_e (\widetilde{\mathbf{m}}^\epsilon, \Gamma^\epsilon) \Vert_{L^{3}(\Omega)}^2 \big) $. 
For the second term on the right-hand side of \eqref{eqn:product with H}, the following estimates are used
\begin{equation*}
        \hspace{-3mm}
	\left\{	\begin{aligned}
		&\int_{\Omega}
		\big( \nabla\e^\epsilon \cdot \ba^\epsilon\nabla\m^\epsilon 
		\big)\m^\epsilon \cdot \widetilde{\mathcal{H}}^\epsilon_e(\e^\epsilon, \Psi^\epsilon )\d \bx
		\le  C
		\Vert \nabla \m^\epsilon \Vert_{L^{6}(\Omega)}
		\Vert\nabla\e^\epsilon \Vert_{L^{3}(\Omega)}
		\Vert \widetilde{\mathcal{H}}^\epsilon_e(\e^\epsilon, \Psi^\epsilon )\Vert_{L^{2}(\Omega)}, \\
		&\int_{\Omega}
		\big( \mu^\epsilon\e^\epsilon \cdot \nabla U^\epsilon 
		\big)\m^\epsilon \cdot \widetilde{\mathcal{H}}^\epsilon_e(\e^\epsilon, \Psi^\epsilon )\d \bx
		\le  C
		\Vert \nabla U^\epsilon \Vert_{L^{6}(\Omega)}
		\Vert \e^\epsilon\Vert_{L^{3}(\Omega)}
		\Vert \widetilde{\mathcal{H}}^\epsilon_e(\e^\epsilon, \Psi^\epsilon )\Vert_{L^{2}(\Omega)},\\
		&\int_{\Omega}
		\big( \mu^\epsilon \widetilde{\m}^\epsilon \cdot \nabla \Psi^\epsilon 
		\big)\m^\epsilon \cdot \widetilde{\mathcal{H}}^\epsilon_e(\e^\epsilon, \Psi^\epsilon )\d \bx
		\le 
		C \Vert \widetilde{\m}^\epsilon \Vert_{L^{\infty}(\Omega)}
		\Vert \nabla \Psi^\epsilon \Vert_{L^{2}(\Omega)}
		\Vert \widetilde{\mathcal{H}}^\epsilon_e(\e^\epsilon, \Psi^\epsilon )\Vert_{L^{2}(\Omega)}.
	\end{aligned}\right.
\end{equation*}
Moreover, employing interpolation inequality and Proposition \ref{lemma: regularity for epslon}, it can be derived
\begin{equation*}
	\begin{aligned}
		\Vert \nabla \e^\epsilon\Vert_{L^{3}(\Omega)}
		\le &
		\Vert \nabla \e^\epsilon\Vert_{L^{2}(\Omega)}^{1/2} 
		\Vert \nabla \e^\epsilon\Vert_{L^{6}(\Omega)}^{1/2} \\
		\le & C
		\Vert \nabla \e^\epsilon\Vert_{L^{2}(\Omega)}^{1/2} 
		(\Vert\mathcal{A}_\epsilon\e^\epsilon \Vert_{L^{2}(\Omega)} 
		+ 
		\Vert \nabla \e^\epsilon\Vert_{L^{2}(\Omega)})^{1/2}\\
		\le & C
		\Vert \nabla \e^\epsilon\Vert_{L^{2}(\Omega)}^{1/2} 
		(1 + \Vert\widetilde{\mathcal{H}}^\epsilon_e(\e^\epsilon, \Psi^\epsilon ) \Vert_{L^{2}(\Omega)} 
		+ 
		\Vert \nabla \e^\epsilon\Vert_{L^{2}(\Omega)})^{1/2}.
	\end{aligned}
\end{equation*}
Combining the above estimates with \eqref{estimate of Phi}, it yields
\begin{equation*}\label{bounded of D_2 e dot H e}
	\int_{\Omega}
	\mathbf{D}_2 (\e^\epsilon, \Psi^\epsilon )  \cdot \widetilde{\mathcal{H}}^\epsilon_e(\e^\epsilon, \Psi^\epsilon )  \d \bx 
	\le C +
	C \Vert \e^\epsilon\Vert_{H^{1}(\Omega)}^2
	+
	\delta \Vert \widetilde{\mathcal{H}}^\epsilon_e(\e^\epsilon, \Psi^\epsilon ) \Vert_{L^{2}(\Omega)}^2,
\end{equation*}
where the constant $C$ depends on 
$\Vert \m^\epsilon  \Vert_{W^{1,\infty}(\Omega)}$, 
$\Vert \widetilde{\m}^\epsilon  \Vert_{W^{1,\infty}(\Omega)}$,  $\Vert \nabla U^\epsilon  \Vert_{L^{6}(\Omega)}$, and $\Vert \nabla \Gamma^\epsilon  \Vert_{L^{6}(\Omega)}$.
For the last term in \eqref{eqn:product with H}, we have
\begin{equation*}
	\begin{aligned}
		\int_{\Omega} \boldsymbol{F} \cdot \widetilde{\mathcal{H}}^\epsilon_e(\e^\epsilon, \Psi^\epsilon ) \d \bx
		\le& C
		\Vert \boldsymbol{F} \Vert_{L^2(\Omega)}^2
		+ \delta
		\Vert \widetilde{\mathcal{H}}^\epsilon_e(\e^\epsilon, \Psi^\epsilon )\Vert_{L^2(\Omega)}^2.
	\end{aligned}
\end{equation*}
Substituting the above results into \eqref{eqn:product with H}, 
it can be derived that
\begin{align*}
	\frac{\d}{\d t} \mathcal{G}_{\mathcal{L}}^\epsilon[\mathbf{e}^\epsilon,\Psi^\epsilon]
	+ &
	(\alpha - C \delta)\Vert \widetilde{\mathcal{H}}^\epsilon_e(\e^\epsilon, \Psi^\epsilon ) \Vert_{L^{2}(\Omega)}^2
	\le 
	C \big(\Vert \e^\epsilon\Vert_{H^1(\Omega)}^2 
	+
	\Vert \boldsymbol{F} \Vert_{L^{2}(\Omega)}^2 \big).
\end{align*}
Integrating the above inequality over $[0,t]$ for $0<t< T$ and using the following equation
\begin{align*}
	\mathcal{G}_{\mathcal{L}}^\epsilon[\mathbf{e}^\epsilon,\Psi^\epsilon]
	\ge &
	\frac{a_{\mathrm{min}}}{2}\Vert \nabla \mathbf{e}^\epsilon \Vert_{L^{2}(\Omega)}^2 - C,
\end{align*}
one can finally derive
\begin{multline*}
	\frac{a_{\mathrm{min}}}{2}\Vert \nabla \mathbf{e}^\epsilon(t) \Vert_{L^{2}(\Omega)}^2
	+ 
	(\alpha - C \delta)\int_0^t
	\Vert \widetilde{\mathcal{H}}^\epsilon_e(\e^\epsilon, \Psi^\epsilon ) \Vert_{L^2(\Omega)}^2 \mathrm{d} \tau \\
	\le 
	C \int_0^t \big(\Vert \mathbf{e}^\epsilon \Vert_{L^2(\Omega)}^2 
	+
	\Vert \nabla\mathbf{e}^\epsilon \Vert_{L^2(\Omega)}^2 
	+
	\Vert \boldsymbol{F} \Vert_{L^{2}(\Omega)}^2 
	\big) \mathrm{d} \tau.
\end{multline*}
Let $\delta\to 0$, the estimate \eqref{stability result in H^1} in \cref{theorem:stability} can be directly deduced by Gr\"{o}nwall's inequality. The proof is completed.

\section{Proof of Lemma \ref{regularety 3}}\label{third lemma}
The following estimates will be used:
\begin{gather}\label{eqn:nabla m dot nabla m}
	a_{\mathrm{max}}^{-1}
	\Vert \m^\epsilon \cdot\mathcal{A}_\epsilon \m^\epsilon \Vert_{L^{3}(\Omega)}^3
	\le
	\Vert \nabla \m^\epsilon \Vert_{L^{6}(\Omega)}^6
	\le 
	a_{\mathrm{min}}^{-1}
	\Vert \mathcal{A}_\epsilon \m^\epsilon \Vert_{L^{3}(\Omega)}^3,\\
	\label{relation between A and H}
	\Vert \mathcal{A}_\epsilon \m^\epsilon \Vert_{L^{p}(\Omega)} - C_p
	\le 
	\Vert \mathcal{H}^\epsilon_e(\m^\epsilon,U^\epsilon) \Vert_{L^{p}(\Omega)}
	\le 
	\Vert \mathcal{A}_\epsilon \m^\epsilon \Vert_{L^{p}(\Omega)} + C_p,
\end{gather}
with $1<p<+\infty$. Here, \eqref{eqn:nabla m dot nabla m} can be derived by $- \ba^\epsilon \abs{\nabla \m^\epsilon}^2 = \m^\epsilon \cdot\mathcal{A}_\epsilon \m^\epsilon$ with  $\abs{\m^\epsilon} = 1$ and the symmetry of $\ba^\epsilon$ in Assumption \eqref{assumption: coefficient}. \eqref{relation between A and H} can be derived by the estimate \eqref{estimate of Phi}. For any vector $\mathbf{v}$, 
an orthogonal decomposition is given by
\begin{equation}\label{eqn: orthogonality decomposition}
	\mathbf{v}
	=
	(\m^\epsilon \cdot \mathbf{v}) \m^\epsilon
	-
	\m^\epsilon \times (\m^\epsilon \times \mathbf{v}).
\end{equation}
In the following, we introduce an interpolation inequality of the effective field $\mathcal{H}^\epsilon_e$ defined in \eqref{mathcal H}.

\begin{lemma}\label{lemma:estimate of H in L3}
	Given $\m^\epsilon\in H^3(\Omega)$, $U^\epsilon$ that satisfies \eqref{eqn:LLG system form 3}, $\abs{\m^\epsilon} = 1$ and the Neumann boundary condition $\boldsymbol{\nu}\cdot \ba^\epsilon\nabla \m^\epsilon = 0$. Then it holds for any $0<\delta<1$
	\begin{equation}\label{term of high order}
		\Vert \mathcal{H}^\epsilon_e(\m^\epsilon,U^\epsilon) \Vert_{L^{3}(\Omega)}^{3}
		\le 
		C_\delta + C_\delta\Vert \mathcal{H}^\epsilon_e (\m^\epsilon,U^\epsilon) \Vert_{L^{2}(\Omega)}^{6}
		+ \delta
		\Vert \m^\epsilon \times \nabla \mathcal{H}^\epsilon_e(\m^\epsilon,U^\epsilon) \Vert_{L^{2}(\Omega)}^{2},
	\end{equation}
	where the constant $C_\delta$ depends on $\delta$, but is independent of $\epsilon$.
\end{lemma}

\begin{proof}
	Applying decomposition \eqref{eqn: orthogonality decomposition} with $\mathbf{v} = \mathcal{H}^\epsilon_e(\m^\epsilon,U^\epsilon)$, we have
	\begin{equation}\label{case of n le 2}
		\begin{aligned}
			\Vert \mathcal{H}^\epsilon_e(\m^\epsilon,U^\epsilon) \Vert_{L^{3}(\Omega)}^{3}
			\le &
			\int_{\Omega}   
			\vert \m^\epsilon \cdot\mathcal{H}^\epsilon_e(\m^\epsilon,U^\epsilon) \vert^3  \d \bx\\
			& +
			\int_{\Omega}
			\vert \m^\epsilon \times \mathcal{H}^\epsilon_e(\m^\epsilon,U^\epsilon) \vert^3  \d \bx
			=: \mathcal{I}_1 + \mathcal{I}_2.
		\end{aligned}
	\end{equation}
    
    Then, the right-hand side of \eqref{case of n le 2} is estimated as follows.
    For $\mathcal{I}_1$, it can be derived by applying \eqref{eqn:nabla m dot nabla m}-\eqref{relation between A and H} and Proposition \ref{lemma: regularity for epslon}
	\begin{equation}\label{I1}
		\begin{aligned}
			\mathcal{I}_1
			\le &
			C + C
			\Vert \nabla \m^\epsilon \Vert_{L^{6}(\Omega)}^6  
			\le 
			C + C\Vert \mathcal{H}^\epsilon_e (\m^\epsilon,U^\epsilon) \Vert_{L^{2}(\Omega)}^{6}.
		\end{aligned}
	\end{equation}
	As for $\mathcal{I}_2$, with Sobolev inequality for $n\le 3$, we have
    \begin{equation}\label{I2}
		 \mathcal{I}_2
		\leq C +C\Vert \m^\epsilon \times \mathcal{H}^\epsilon_e(\m^\epsilon,U^\epsilon) \Vert_{L^{2}(\Omega)}^{6}
		+\delta^*\Vert \m^\epsilon \times \mathcal{H}^\epsilon_e(\m^\epsilon,U^\epsilon) \Vert_{H^{1}(\Omega)}^{2}. 
    \end{equation}
	For the last term of \eqref{I2}, by applying \eqref{eqn:nabla m dot nabla m}-\eqref{relation between A and H}, it can be proved that
    \begin{equation*}\label{eqn:process in regurality}
		\begin{aligned}
			\delta^* \Vert \nabla \m^\epsilon \times \mathcal{H}^\epsilon_e(\m^\epsilon) \Vert_{L^{2}(\Omega)}^{2}
			\le &
			\delta^* \Vert \nabla \m^\epsilon \Vert_{L^{6}(\Omega)}^{6}
			+
			\delta^* \Vert \mathcal{H}^\epsilon_e(\m^\epsilon,U^\epsilon) \Vert_{L^{3}(\Omega)}^{3}\\
			\le &
			C + C\delta^* \Vert \mathcal{H}^\epsilon_e(\m^\epsilon,U^\epsilon) \Vert_{L^{3}(\Omega)}^{3}.
		\end{aligned}
	\end{equation*}
	Combining \eqref{I1}, \eqref{I2} with \eqref{case of n le 2}, one can finally derive
	\begin{equation*}
		\begin{aligned}
			(1 - C \delta^*)\Vert \mathcal{H}^\epsilon_e(\m^\epsilon,U^\epsilon) \Vert_{L^{3}(\Omega)}^{3}
			\le &
			C + C\Vert \mathcal{H}^\epsilon_e (\m^\epsilon,U^\epsilon) \Vert_{L^{2}(\Omega)}^{6}\\
			&+ \delta^*
			\Vert \m^\epsilon \times \nabla \mathcal{H}^\epsilon_e(\m^\epsilon,U^\epsilon) \Vert_{L^{2}(\Omega)}^{2}.
		\end{aligned}
	\end{equation*}
	Let $\delta^* < \frac{1}{2C}$, \eqref{term of high order} can be derived by $\delta = \delta^*/(1 - C \delta^*)< 1$.
\end{proof}

In the following, we present the results about uniform boundedness.
The multiscale LLG equation \eqref{eqn:LLG system form 3} can be rewritten into a degenerate form by using the vector triple product formula \eqref{eqn:vec product}
\begin{equation}\label{eqn:LLG system form 2}
	\partial_t\m^\epsilon 
	+ \alpha \m^\epsilon \times 
	\big(\m^\epsilon \times \mathcal{H}^\epsilon_e(\m^\epsilon,U^\epsilon) \big) 
	+ \m^\epsilon \times \mathcal{H}^\epsilon_e(\m^\epsilon,U^\epsilon) 
	=0.
\end{equation}
Multiplying \eqref{eqn:LLG system form 2} by $\mathcal{H}^\epsilon_e(\m^\epsilon,U^\epsilon)$ and integrating over $(0,t)$, the energy dissipation identity can be given by
\begin{equation}\label{energy decay}
	\begin{aligned}
		\mathcal{G}_{\mathcal{L}}^\epsilon[\m^\epsilon(t),U^\epsilon(t)] + 
		\alpha \int_0^t \Vert \m^\epsilon \times \mathcal{H}^\epsilon_e(\m^\epsilon) \Vert_{L^2(\Omega)}^2 \d \tau 
		= 
		\mathcal{G}_{\mathcal{L}}^\epsilon[\m^\epsilon(0),U^\epsilon(0)],
	\end{aligned}
\end{equation}
The integrability of kinetic energy $\Vert \partial_t \m^\epsilon \Vert_{L^2(\Omega)}^2$ can by induced by \eqref{eqn:LLG system form 2} and \eqref{energy decay}
\begin{align}\label{integrable of kinetic energy}
	\int_0^t \Vert \partial_t \m^\epsilon \Vert_{L^2(\Omega)}^2 \d \tau &
	\le 
	C.
\end{align}
The energy identity \eqref{energy decay} implies the uniform boundedness of $\Vert \m^\epsilon \times \mathcal{A}_\epsilon\m^\epsilon \Vert_{L^2(\Omega)}^2$. However, due to the degeneracy caused by the outer product, it is not enough to prove the boundedness of $\Vert \mathcal{A}_\epsilon\m^\epsilon \Vert_{L^2(\Omega)}^2$. To handle this problem, the following lemma is introduced.
%

\begin{lemma}\label{regularety}
	Let $\m^\epsilon\in L^2([0,T];H^3(\Omega))$ be a solution to \eqref{eqn: Periodic problem}, then there exists $T^*\in(0, T]$ independent of $\epsilon$, such that for $0\le t\le T^*$,
	\begin{equation*}
		\Vert \mathcal{A}_\epsilon \m^\epsilon  (t) \Vert_{L^{2}(\Omega)}^2 
		+
		\int_0^{t} \big\Vert \m^\epsilon \times \nabla \mathcal{H}^\epsilon_e(\m^\epsilon,U^\epsilon)  
		(\tau) \big\Vert_{L^{2}(\Omega)}^2 \d \tau  \le C.
	\end{equation*}
	Therefore, by Proposition \ref{lemma: regularity for epslon}, it holds
	\begin{equation*}
		\Vert \nabla \m^\epsilon  (\cdot, t)\Vert_{L^{6}(\Omega)}^2
		\le
		C,
	\end{equation*}
	where $C$ is a constant independent of $\epsilon$ and $t$.
	
\end{lemma}

\begin{proof}
	Applying $\nabla$ to \eqref{eqn:LLG system form 2} and multiplying it by $ \ba^\epsilon \nabla \mathcal{H}^\epsilon_e(\m^\epsilon,U^\epsilon)$, it leads to
	\begin{equation}\label{eqn:regu esti in 3d}
		\begin{aligned}
			& -
			\int_{\Omega} \nabla (\partial_t \m^\epsilon) \cdot \ba^\epsilon \nabla \mathcal{H}^\epsilon_e(\m^\epsilon,U^\epsilon) \d \bx \\
			= & 
			\ \alpha \int_{\Omega}  \nabla 
			\Big(\m^\epsilon\times\big(\m^\epsilon\times \mathcal{H}^\epsilon_e(\m^\epsilon,U^\epsilon)\big) \Big)
			\cdot \ba^\epsilon \nabla \mathcal{H}^\epsilon_e(\m^\epsilon,U^\epsilon)  \d \bx\\
			& +
			\sum_{i,j=1}^{n}
			\int_{\Omega} \frac{\partial}{\partial x_i} \m^\epsilon \times \mathcal{H}^\epsilon_e(\m^\epsilon,U^\epsilon) 
			\cdot a^\epsilon_{ij} \frac{\partial}{\partial x_j} \mathcal{H}^\epsilon_e(\m^\epsilon,U^\epsilon) \d \bx  
			=: 
			\mathcal{J}_1 + \mathcal{J}_2.
		\end{aligned}
	\end{equation}
	Denote $\Gamma^\epsilon(\m^\epsilon,U^\epsilon) = \mathcal{H}^\epsilon_e(\m^\epsilon,U^\epsilon)  - \mathcal{A}_\epsilon\m^\epsilon$. With integration by parts, the left-hand side of \eqref{eqn:regu esti in 3d} becomes
	\begin{equation*}
		\begin{aligned}
			& -
			\int_{\Omega} \nabla (\partial_t \m^\epsilon) \cdot \ba^\epsilon \nabla \mathcal{H}^\epsilon_e(\m^\epsilon,U^\epsilon) \d \bx 
			= 
			\int_{\Omega} \mathcal{A}_\epsilon (\partial_t \m^\epsilon) \cdot \big( \mathcal{A}_\epsilon\m^\epsilon + \Gamma^\epsilon(\m^\epsilon,U^\epsilon) \big) \d \bx,
		\end{aligned}
	\end{equation*}
	where the right-hand side can be rewritten as 
	\begin{equation*}\label{eqn:regu esti in 3d 2}
		\begin{aligned}
			\frac{1}{2}\frac{\d }{\d t}
			\int_{\Omega} \vert \mathcal{A}_\epsilon \m^\epsilon \vert^2 \d \bx 
			+
			\frac{\d }{\d t}
			\int_{\Omega} \mathcal{A}_\epsilon  \m^\epsilon \cdot \Gamma^\epsilon(\m^\epsilon,U^\epsilon) \d \bx  
			-
			\int_{\Omega} \mathcal{A}_\epsilon \m^\epsilon \cdot  \Gamma^\epsilon(\partial_t\m^\epsilon,\partial_tU^\epsilon) \d \bx.
		\end{aligned}
	\end{equation*}
	Then, the right-hand side of \eqref{eqn:regu esti in 3d} are considered. For $\mathcal{J}_1$, it can be  derived by swapping the order of mixed product
	\begin{equation*}\label{eqn:regu esti in 3d 3}
		\begin{aligned}
			\mathcal{J}_1
			=  -
			\alpha \sum_{i, j= 1}^n\int_{\Omega}   
			\Big(\m^\epsilon\times \frac{\partial}{\partial x_i} \mathcal{H}^\epsilon_e(\m^\epsilon,U^\epsilon)\Big) 
			\cdot 	a^\epsilon_{ij} \Big(\m^\epsilon\times 
			\frac{\partial}{\partial x_j}\mathcal{H}^\epsilon_e(\m^\epsilon,U^\epsilon)\Big)  \d \bx 
			+ F_1 ,
		\end{aligned}
	\end{equation*}
	where $F_1$ is the low-order term, and the first term on right-hand side is sign-preserved due to the uniform coerciveness of $\ba^\epsilon$ in Assumption \eqref{assumption: coefficient}.
	As for $\mathcal{J}_2$, applying \eqref{eqn: orthogonality decomposition} with $\mathbf{v} = a^\epsilon_{ij} \partial_j \mathcal{H}^\epsilon_e(\m^\epsilon)$, it leads to
	\begin{equation}\label{F_2}
		\begin{aligned}
			\mathcal{J}_2 = & \sum_{i,j=1}^{n}
			\int_{\Omega} \m^\epsilon \times\Big( \frac{\partial}{\partial x_i} \m^\epsilon \times \mathcal{H}^\epsilon_e(\m^\epsilon,U^\epsilon) \Big) 
			\cdot
			\Big( \m^\epsilon \times  
			a^\epsilon_{ij} \frac{\partial}{\partial x_j} \mathcal{H}^\epsilon_e(\m^\epsilon,U^\epsilon) \Big) \d \bx\\
			& - \sum_{i,j=1}^{n}
			\int_{\Omega} \Big( \m^\epsilon \times \mathcal{H}^\epsilon_e(\m^\epsilon,U^\epsilon) \cdot \frac{\partial}{\partial x_i} \m^\epsilon  \Big)
			\m^\epsilon  
			\cdot a^\epsilon_{ij} \frac{\partial}{\partial x_j} \mathcal{H}^\epsilon_e(\m^\epsilon,U^\epsilon) \d \bx.
		\end{aligned}
	\end{equation}
	Applying the vector triple product formula \eqref{eqn:vec product} to the first term, and performing integration by parts on the second term, \eqref{F_2} becomes
	\begin{equation*}\label{eqn:transform with F_1}
		\begin{aligned}
			\mathcal{J}_2 =& 2 \sum_{i,j=1}^{n}
			\int_{\Omega} \big( \m^\epsilon \cdot \mathcal{H}^\epsilon_e(\m^\epsilon,U^\epsilon) \big)
			\Big( \m^\epsilon \times  
			a^\epsilon_{ij} \frac{\partial}{\partial x_j} \mathcal{H}^\epsilon_e(\m^\epsilon,U^\epsilon)
			\cdot \frac{\partial}{\partial x_i} \m^\epsilon \Big) \d \bx + F_2\\
			\le & C
			\Vert \nabla \m^\epsilon \Vert_{L^{6}(\Omega)}^6
			+
			C
			\Vert \mathcal{H}^\epsilon_e(\m^\epsilon,U^\epsilon) \Vert_{L^{3}(\Omega)}^3
			+
			\delta \Vert \m^\epsilon \times \nabla \mathcal{H}^\epsilon_e(\m^\epsilon,U^\epsilon) \Vert_{L^{2}(\Omega)}^{2} + F_2,
		\end{aligned}
	\end{equation*}
	where $F_2$ is the low-order term. It can be proved that $F_i\ (i=1,2)$ satisfy the following estimate by utilizing \eqref{relation between A and H} and H\"{o}lder's inequality 
	\begin{equation*}\label{bound of F_1 process result}
		\begin{aligned}
			F_i
			\le & C + C
			\Vert \nabla \m^\epsilon \Vert_{L^{6}(\Omega)}^6
			+
			C
			\Vert \mathcal{H}^\epsilon_e(\m^\epsilon,U^\epsilon) \Vert_{L^{3}(\Omega)}^3.
		\end{aligned}
	\end{equation*}
	Substituting the above results into \eqref{eqn:regu esti in 3d}, applying estimate \eqref{eqn:nabla m dot nabla m} and Lemma \ref{lemma:estimate of H in L3}, we can finally derive
 
	\begin{equation}\label{esimate of regular in 3d}
		\begin{aligned}
			& \frac{1}{2}\frac{\d }{\d t}
			\Vert \mathcal{A}_\epsilon  \m^\epsilon \Vert_{L^2(\Omega)}^2 
			+
			(\alpha a_{\mathrm{min}} - C\delta) \Vert \m^\epsilon \times \nabla \mathcal{H}^\epsilon_e(\m^\epsilon,U^\epsilon) \Vert_{L^{2}(\Omega)}^{2}  \\
			\le & C +
			C \Vert \mathcal{A}_\epsilon  \m^\epsilon \Vert_{L^2(\Omega)}^6 
			+ C
			\Vert \partial_t \m^\epsilon \Vert_{L^2(\Omega)}^2
			-
			\frac{\d }{\d t}
			\int_{\Omega} \mathcal{A}_\epsilon  \m^\epsilon \cdot \Gamma^\epsilon(\m^\epsilon,U^\epsilon) \d \bx.
		\end{aligned}
	\end{equation}
	Integrating \eqref{esimate of regular in 3d} over $[0,t]$ and utilizing \eqref{integrable of kinetic energy} along with the following inequality
	\begin{align*}
		\int_{\Omega} \mathcal{A}_\epsilon  \m^\epsilon \cdot \Gamma^\epsilon(\m^\epsilon,U^\epsilon) \d \bx & 
		\le
		C\Vert \Gamma^\epsilon(\m^\epsilon)  \Vert_{L^{2}(\Omega)}^{2}
		+ \frac{1}{4}\Vert \mathcal{A}_\epsilon  \m^\epsilon \Vert_{L^{2}(\Omega)}^{2},
	\end{align*}
	it can be proved that for any $t\in (0,T]$
	\begin{equation}\label{esimate of regular by integrating}
		\begin{aligned}
			\frac{1}{4} \Vert \mathcal{A}_\epsilon  \m^\epsilon (t) \Vert_{L^2(\Omega)}^2 
			\le & C +
			C \int_0^t \Vert \mathcal{A}_\epsilon  \m^\epsilon (\tau)\Vert_{L^2(\Omega)}^6 \d \tau,
		\end{aligned}
	\end{equation}
	where $C$ depends on $\Vert\mathcal{A}_\epsilon \m^\epsilon(0)\Vert_{L^{2}(\Omega)}$, $\mathcal{G}_{\mathcal{L}}^\epsilon[\m^\epsilon(0),U^\epsilon(0)]$. With \cref{lemma: regularity for epslon} and Assumption \eqref{assumption: initial data}, it can be proved that $C$ is independent of $\epsilon$ and $t$.
    Let $F(t)$ denote the right-hand side of \eqref{esimate of regular by integrating}, it follows that
	\begin{equation*}
		\frac{\d}{\d t} F(t) \le C F^3(t).
	\end{equation*}
	By the Cauchy-Lipshitz-Picard Theorem \cite{BB:S:2011} and the comparison principle, there exists $T^*\in (0,T]$ independent of $\epsilon$, such that $F(t)$ is uniformly bounded on $[0,T^*]$. Thus, $\Vert \mathcal{A}_\epsilon  \m^\epsilon (t) \Vert_{L^2(\Omega)}^2 $ is uniformly bounded by \eqref{esimate of regular by integrating}.
\end{proof}

Finally, \cref{regularety 3} can be derived by \cref{regularety}. The proof is completed.


\bibliographystyle{siamplain}
\bibliography{ref}
\end{sloppypar}
\end{document}

%% file: ex_shared.tex
\usepackage{lipsum}
\usepackage{amsfonts}
\usepackage{graphicx}
\usepackage{epstopdf}
\usepackage{algorithmic}
\ifpdf
  \DeclareGraphicsExtensions{.eps,.pdf,.png,.jpg}
\else
  \DeclareGraphicsExtensions{.eps}
\fi

\newsiamremark{remark}{Remark}
\newsiamremark{hypothesis}{Hypothesis}
\crefname{hypothesis}{Hypothesis}{Hypotheses}
\newsiamthm{claim}{Claim}

\headers{Two-scale Analysis for Multiscale LLG Equation}{X. Guan, H. Qi, and Z. Sun}

\title{Two-scale Analysis for Multiscale Landau-Lifshitz-Gilbert Equation:\\ Theory and Numerical Methods\thanks{Submitted to the editors DATE.
\funding{The work was supported by the National Natural Science Foundation of China (No. 12271409), the Natural Science Foundation of Shanghai (No. 21ZR1465800), the Interdisciplinary Project in Ocean Research of Tongji University and the Fundamental Research Funds for the Central Universities.
}}}

\author{Xiaofei Guan\thanks{School of Mathematical Sciences, Tongji University, Shanghai, 200092, China 
  (\email{guanxf@tongji.edu.cn}, \email{qihang7750@gmail.com}).}
\and Hang Qi\footnotemark[2]
\and Zhiwei Sun\thanks{School of Mathematical Sciences, Soochow University, Suzhou, 215006, China 
(\email{sunzhiwei1029@gmail.com}).}}

\usepackage{amsopn}
